\documentclass[english]{article}
\usepackage[margin=2.5cm]{geometry}
\usepackage[hyphens]{url}
\usepackage{hyperref}
\usepackage[hyphenbreaks]{breakurl}
\usepackage{amsthm,amsfonts,amssymb,amsmath}
\usepackage{graphicx}
\usepackage{pgfplots}
\usepackage{xcolor}
\usepackage{bbm}
\usepackage{paralist}

\usepackage[labelfont=bf,labelsep=period]{caption}

\usepackage[nameinlink,capitalise,noabbrev]{cleveref}
\crefformat{equation}{#2(#1)#3}

\newtheorem{theorem}{Theorem}[section]
\newtheorem{proposition}[theorem]{Proposition}
\newtheorem{lemma}[theorem]{Lemma}
\newtheorem{conjecture}[theorem]{Conjecture}

\newtheorem{corollary}[theorem]{Corollary}
\newtheorem{definition}[theorem]{Definition}

\theoremstyle{remark}
\newtheorem{remark}{Remark}[section]

\numberwithin{equation}{section}
\numberwithin{figure}{section}
\allowdisplaybreaks

\usepackage{babel}
\usepackage[T1]{fontenc}
\usepackage[utf8]{inputenc}

\newcommand{\pvec}[1]{\vec{#1}\mkern2mu\vphantom{#1}'}

\newcommand{\su}{\subseteq}
\newcommand{\sm}{\setminus}

\newcommand{\PP}{\mathbb{P}}
\newcommand{\EE}{\mathbb{E}}

\newcommand{\RR}{\mathbb{R}}
\newcommand{\ZZ}{\mathbb{Z}}
\newcommand{\rank}{\operatorname{rank}}

\newcommand{\mb}{\mathbb}

\newcommand{\mc}{\mathcal}

\newcommand{\on}{\operatorname}

\let\originalleft\left
\let\originalright\right
\renewcommand{\left}{\mathopen{}\mathclose\bgroup\originalleft}
\renewcommand{\right}{\aftergroup\egroup\originalright}

\begin{document}
\title{Resolution of the quadratic Littlewood--Offord problem}
\author{Matthew Kwan\thanks{Institute of Science and Technology Austria (ISTA), 3400 Klosterneuburg, Austria. Email address: \url{matthew.kwan@ist.ac.at}. Supported by ERC Starting Grant ``RANDSTRUCT'' No.\ 101076777.}\and Lisa Sauermann\thanks{Institute for Applied Mathematics, University of Bonn, 53115 Bonn, Germany. Email address: \url{sauermann@iam.uni-bonn.de}. Part of this work was completed while this author was affiliated with the Mathematics Department at the Massachusetts Institute of Technology. Supported in part by NSF Award DMS-2100157 and a Sloan Research Fellowship, and in part by the DFG Heisenberg Program.}}

\maketitle

\begin{abstract}\noindent
Consider a quadratic polynomial $Q(\xi_{1},\dots,\xi_{n})$ of independent
Rademacher random variables $\xi_{1},\dots,\xi_{n}$. To what extent
can $Q(\xi_{1},\dots,\xi_{n})$ concentrate on a single value? This quadratic version of the classical Littlewood--Offord problem was popularised by Costello, Tao and Vu in their study of symmetric random matrices.
In this paper, we obtain an essentially optimal bound for this problem, as conjectured by Nguyen and Vu.

Specifically, if $Q(\xi_{1},\dots,\xi_{n})$ ``robustly depends on
at least $m$ of the $\xi_{i}$'' in the sense that there is no way to pin
down the value of $Q(\xi_{1},\dots,\xi_{n})$ by fixing values for
fewer than $m$ of the variables $\xi_{i}$, then we have $\Pr[Q(\xi_{1},\dots,\xi_{n})=0]\le O(1/\sqrt{m})$. This also implies a similar result in the case where
$\xi_{1},\dots,\xi_{n}$ have arbitrary distributions.

Our proof combines a number of ideas that may be of independent interest, including an \emph{inductive decoupling scheme} that reduces quadratic anticoncentration problems to high-dimensional linear anticoncentration problems. Also, one application of our main result is the resolution of a conjecture of Alon, Hefetz, Krivelevich
and Tyomkyn related to \emph{graph inducibility}.
\end{abstract}

\section{Introduction}

\makeatletter
\newcommand*{\transpose}{%
  {\mathpalette\@transpose{}}%
}
\newcommand*{\@transpose}[2]{%
  % #1: math style
  % #2: unused
  \raisebox{\depth}{$\m@th#1\intercal$}%
}
\makeatother

Some of the most fundamental theorems in probability theory (and its applications) are \emph{concentration} inequalities, which show that certain types of random variables are likely to lie in a small interval around their mean. In the other direction, \emph{anticoncentration} inequalities give \emph{upper} bounds on the probability that a random variable falls into a small interval or is equal to a particular value. In this area, one of the most important directions is the (polynomial) \emph{Littlewood--Offord problem}. Roughly speaking, the problem is as follows\footnote{This is the ``discrete'' form of the polynomial Littlewood--Offord problem; one can also consider the ``continuous'' form, where we are interested in the maximum probability that $P(\xi_1,\dots,\xi_n)$ lies in an interval of given length. In the following historical discussion, we specialise all results to this discrete setting.}. Consider an $n$-variable polynomial $P\in \mb R[x_1,\dots,x_n]$, and consider independent Rademacher random variables $\xi_1,\dots,\xi_n\in\{-1,1\}$ (meaning that $\Pr[\xi_{i}=-1]=\Pr[\xi_{i}=1]=1/2$ for each $i$). What upper bounds can we prove on the \emph{point probabilities} of the form $\Pr[P(\xi_1,\dots,\xi_n)=z]$? More specifically, how large can the maximum point probability $\sup_{z\in \RR} \Pr[P(\xi_1,\dots,\xi_n)=z]$ be, without making any strong assumptions about the polynomial $P$?

Historically, the starting point for this problem was the \emph{linear} case, where $P$ is a degree-1 polynomial. Indeed, consider a random variable $X=a_{1}\xi_{1}+\dots+a_{n}\xi_{n}$, where
$a_{1},\dots,a_{n}\in\RR$ are nonzero real numbers and $\xi_{1},\dots,\xi_{n}\in\{-1,1\}$
are independent Rademacher random variables. As part of their study of random polynomials, in 1943
Littlewood and Offord~\cite{LO43} proved that
\[
\sup_{z\in\RR}\Pr[X=z]\le O\left(\frac{\log n}{\sqrt{n}}\right).
\]
Littlewood and Offord's result was famously sharpened in 1945 by Erd\H os~\cite{Erd45},
who found a purely combinatorial proof of what is now usually called the \emph{Erd\H os--Littlewood--Offord
theorem}: under the same assumptions,
\begin{equation}\label{eq:elo}
\sup_{z\in\mb R}\Pr[X=z]\le\binom{n}{\lfloor n/2\rfloor}\cdot 2^{-n}=O\left(\frac{1}{\sqrt{n}}\right).
\end{equation}
This result is best-possible, as can be seen by considering the case
where the coefficients $a_{i}$ are all equal.

\begin{remark}\label{rem:levy}
It is worth noting that the general topic of anticoncentration of
sums of independent random variables was first considered in 1936
by Doeblin and L\'evy~\cite{DL36}, and this led to a parallel line
of research with many similar results (the two lines of research seem
to have not been aware of each others' existence until quite recently).
In particular, in the above setting, the bound $\sup_{z\in\RR}\Pr[X=z]\le O(1/\sqrt{n})$ follows from a general result claimed in a 1939 paper
of Doeblin~\cite{Doe39}, preceding Littlewood and Offord by several
years. However, this general result is also the subject of a 1958
paper of Kolmogorov~\cite{Kol58}, which claims that Doeblin's paper
did not provide a full proof.
\end{remark}

By now, the linear Littlewood--Offord problem is very well understood, and many variations and strengthenings are available. For example, there are very powerful \emph{inverse theorems} that relate the maximum concentration probability to the arithmetic structure of the coefficients $a_1,\dots,a_n$, and these theorems have had a huge impact in random matrix theory (see for example the survey \cite{NV13}). Also, there are versions of the Erd\H os--Littlewood--Offord theorem for general
distributions (i.e., allowing the variables $\xi_{1},\dots,\xi_{n}$ to take non-Rademacher
distributions), assuming that the distributions of the variables $\xi_{i}$ do
not themselves concentrate too strongly. (Approximate results along these lines follow directly from the L\'evy--Doeblin--Kolmogorov theorem mentioned in \cref{rem:levy}, or can be deduced from the Erd\H os--Littlewood--Offord theorem; in some sense the Rademacher case is the ``hardest case''. For exact results see \cite{Jon78,Jus22,LR94}.)

After the linear case, the next case to consider is the \emph{quadratic} case: what bounds can we
give on the point concentration of a quadratic polynomial $Q(\xi_{1},\dots,\xi_{n})$
of independent Rademacher random variables $\xi_{1},\dots,\xi_{n}$?
This question came to the forefront in a 2005 paper by Costello, Tao and
Vu~\cite{CTV06}, when they used such a bound in their proof of Weiss'
conjecture on singularity of random symmetric matrices. Specifically,
they proved that if a quadratic polynomial $Q\in\RR[x_{1},\dots,x_{n}]$
has at least $cn^{2}$ nonzero coefficients\footnote{This is not exactly the assumption that appeared in \cite{CTV06}, but it is a simple assumption which is sufficient for essentially all results in this area; see \cref{rem:nondegeneracy-conditions} for more discussion.} (for some constant $c>0$),
then $X=Q(\xi_{1},\dots,\xi_{n})$ (for independent Rademacher random variables
$\xi_{1},\dots,\xi_{n}$) satisfies
\begin{equation}
\sup_{z\in\RR}\Pr[X=z]\le O\left(\frac{1}{n^{1/8}}\right).\label{eq:CTV}
\end{equation}

\begin{remark}\label{rem:RS}
Parallelling the situation described in \cref{rem:levy}, we remark that Costello,
Tao and Vu were actually not the first to prove this inequality: in
1996, Rosi\'nski and Samorodnitsky~\cite{RS96} had proved essentially
the same result (in fact, a generalisation of it to higher degree
polynomials), in their study of zero-one laws for L\'evy chaos (\emph{chaoses} are polynomials of independent random variables; they are classical and very well-studied objects in probability theory, statistics and applied mathematics). Seemingly unaware of
Rosi\'nski and Samorodnitsky's work, in 2013 Razborov and Viola~\cite{RV13}
considered a similar higher-degree generalisation (for applications in the theory of Boolean functions).
\end{remark}

The authors of \cite{CTV06,RS96} already recognised that \cref{eq:CTV} was likely not optimal. Indeed, just as for the linear Littlewood--Offord problem, one expects a bound of the form $\sup_{z\in\RR}\Pr[X=z]\le O(1/\sqrt{n})$ (this bound is attained in the case $Q(x_{1},\dots,x_{n})=(x_{1}+\dots+x_{n})^{2}$, for example). A conjecture to this effect has been attributed to Nguyen and Vu (see \cite{MNV16,RV13}).

The first improvement on \cref{eq:CTV} was by Costello and Vu~\cite{CV08}, who showed how to adapt the arguments in \cite{CTV06} to prove a bound of the form $O(n^{-1/4})$. Introducing several new ideas, Costello~\cite{Cos13} then managed to obtain the nearly optimal bound $O(n^{-1/2+\varepsilon})$ (for
any constant $\varepsilon>0$). Via a completely different approach, this bound was further refined to
$\exp(O((\log\log n)^{2}))/\sqrt{n}$ by Meka, Nguyen and Vu\footnote{The results in \cite{MNV16} are more general, holding for polynomials of any bounded degree.} (see arXiv version v4 of \cite{MNV16}), before it was observed that the bound $(\log n)^{O(1)}/\sqrt{n}$ follows from a powerful general result of Kane~\cite{Kan14} (see the journal version of \cite{MNV16} for more discussion).

In this paper we finally resolve the quadratic Littlewood--Offord problem (up to constant factors), obtaining an optimal bound of $O(1/\sqrt{n})$. We are also able to make a weaker assumption on $Q$ than in previous
work.
\begin{theorem}\label{thm:main-shiny}
Let $Q\in\RR[x_{1},\dots,x_{n}]$ be a polynomial of degree at most 2, and let $\xi_{1},\dots,\xi_{n}\in\{-1,1\}$ be independent Rademacher random variables.
For some positive integer $m$, suppose that $Q(\xi_{1},\dots,\xi_{n})$ ``robustly depends on at least $m$ of the variables $\xi_{i}$'' in the sense that the value of $Q(\xi_{1},\dots,\xi_{n})$ cannot be determined by specifying any outcomes of any $m-1$ of the variables $\xi_{1},\dots,\xi_n\in \{-1,1\}$. Then,
\[
\sup_{z\in\RR}\Pr[Q(\xi_{1},\dots,\xi_{n})=z]\le \frac{C}{\sqrt{m}},
\]
for some absolute constant $C$.
\end{theorem}

\begin{remark}\label{rem:nondegeneracy-conditions}
Recall that the Erd\H os--Littlewood--Offord theorem has an assumption
that each linear coefficient $a_i$ is nonzero. Of course, zero
coefficients can be ignored,
so without such an assumption one immediately
obtains a bound of the form $\sup_{z\in \mb R}\Pr[a_1\xi_1+\dots+a_n\xi_n=z]\le O(1/\sqrt{m})$ where $m$ is the number
of nonzero $a_{i}$.
Unfortunately, the situation is not so simple
in the quadratic case. One could make the very strong assumption that
\emph{every} degree-2 coefficient is nonzero, but this is too strong
of an assumption for most applications\footnote{We also remark that this strong assumption doesn't seem to make the problem much easier.}. Alternatively, one might wish
to consider the very weak assumption that every variable $x_{i}$
features in at least one nonzero term of $Q(x_{i})$, but unfortunately
this is too weak to get a sensible bound: indeed, consider for example the
case where $Q(x_{1},\dots,x_{n})=(1+\xi_{1})(\xi_{1}+\dots+\xi_{n})$,
which is zero whenever $\xi_{1}=-1$, and therefore $\Pr[Q(\xi_{1},\dots,\xi_{n})=0]\ge 1/2$. 
More generally, if it is possible to determine the value of $Q(\xi_{1},\dots,\xi_{n})$ by fixing the outcomes of a small number of variables to certain values, then this automatically leads to a large point probability for $Q(\xi_{1},\dots,\xi_{n})$. So it is necessary to make an assumption guaranteeing that $Q$ ``robustly depends on many of the $\xi_{i}$''.

We believe that our assumption in \cref{thm:main-shiny} captures this ``robust dependence on many of the $\xi_{i}$'' in a natural way.
To compare to previous work: our assumption is slightly weaker than
the assumption in \cite{MNV16} (which says that $Q(x_{1},\dots,x_{n})$ has many
quadratic terms featuring disjoint variables), and is much weaker
than the assumptions in \cite{Cos13,CTV06} (which say, in slightly different ways, that $Q(x_{1},\dots,x_{n})$ has a huge number of nonzero coefficients).
\end{remark}
\vspace{1pt}

\begin{remark}
Of course, we could ask for the optimal constant factor $C$ in \cref{thm:main-shiny}. It is not necessarily clear what to expect: one may guess that the polynomial $Q(x_1,\dots,x_m)=(x_1+\dots+x_m)(x_1+\dots+x_m+2)$ or $Q(x_1,\dots,x_m)=(x_1+\dots+x_m+1)(x_1+\dots+x_m-1)$ (depending on whether $m$ is odd or even) is the worst case, but recent developments on the so-called \emph{Gotsman--Linial conjecture} (see \cref{subsec:further-directions}) suggest that this might be too na\"ive. It is also quite possible that the optimal value for the constant factor in the bound on $\sup_{z\in\RR}\Pr[Q(\xi_{1},\dots,\xi_{n})=z]$ is sensitive to the precise assumption one makes on the quadratic polynomial $Q$ (to ensure that it ``robustly depends on many of the variables $\xi_i$''; see \cref{rem:nondegeneracy-conditions}).
\end{remark}

In \cref{sec:outline} we give a brief summary of the methods that had previously been applied to the (quadratic) Littlewood–Offord problem, their limitations, and the new ideas in our proof of \cref{thm:main-shiny}. In particular, our key contribution is a new \emph{inductive decoupling scheme}: we take the well-known technique of \emph{decoupling}, usually viewed as a tool to inefficiently ``reduce from a quadratic problem to a linear problem'', and reinterpret it as a tool to efficiently ``reduce the relative dimension of the quadratic part of a problem''.  We also develop a new way to study anticoncentration of random vectors, via a technique we call \emph{witness-counting}. We believe both these aspects of our proof may have broader applications.

We also remark that, while the details of the proof of \cref{thm:main-shiny} are rather involved, there is a certain special case of \cref{thm:main-shiny} (the case where the quadratic part of $Q$ has ``bounded rank'') which permits a much simpler proof. This case is already interesting, and we present its proof in \cref{sec:geometric} to serve as an accessible illustration of our inductive decoupling scheme (and the results of \cref{sec:geometric} will also be used later in the paper).

Finally, we remark that just as for the linear Littlewood--Offord problem, one can deduce a version of \cref{thm:main-shiny} in which the variables $\xi_{i}$ are allowed to take essentially any discrete distribution (not just the Rademacher distribution), as follows.

\begin{theorem}\label{thm:general-distributions}
    Fix $0<\delta<1$. Let $Q\in\RR[x_{1},\dots,x_{n}]$ be a polynomial of degree at most $2$, and let $\zeta_{1},\dots,\zeta_{n}\in \RR$ be independent discrete random variables.

    For nonempty subsets $R_1,\dots,R_n$ of the supports of $\zeta_{1},\dots,\zeta_{n}$, respectively,  say that the product $R_1\times \dots\times R_n$ is a \emph{fixing box} for $Q$ if the polynomial $Q$ is constant on $R_1\times \dots\times R_n$. For some positive integer $m$, suppose that for any fixing box $R_1\times \dots\times R_n$ there are at least $m$ indices $i\in \{1,\dots,n\}$ such that $\Pr[\zeta_{i}\in R_i]\le 1-\delta$. Then, we have
\[
\sup_{z\in\RR}\Pr[Q(\zeta_{1},\dots,\zeta_{n})=z]\le \frac{C_\delta}{\sqrt{m}},
\]
for some constant $C_\delta$ only depending on $\delta$.
\end{theorem}
Note that if $\zeta_1,\dots,\zeta_n$ are independent uniformly random integers in $\{-B,-(B-1),\dots,B-1,B\}$, then $\Pr[Q(\zeta_{1},\dots,\zeta_{n})=z]\cdot (2B+1)^n$ is the number of integer solutions to $Q(x_1,\dots,x_n)=z$ among integers $x_1,\dots,x_n$ with absolute value (``height'') at most $B$. We remark that quantities of this form have been extensively studied in analytic number theory, in the regime where $n$ is constant and $B$ is large (in contrast, our estimates are effective in the regime where $B$ is constant and $n$ is large); see for example \cite{HB02,Pil95} and \cite[Theorem~1.11]{BG17}.

\subsection{An application to edge-statistics}

Apart from the intrinsic value of \cref{thm:main-shiny}, of course it also enables us to improve bounds in any place where quadratic Littlewood--Offord inequalities had previously been applied (see for example \cite{AE14,CTV06,CV10,CV08,FKSS23,GKSS,KST19}).  Here we highlight one particular application: we can resolve a conjecture of Alon, Hefetz, Krivelevich and Tyomkyn~\cite{AHKT20} related to the so-called \emph{graph inducibility problem}. Specifically, let $\mathcal{G}_{n}$ be the set of $n$-vertex graphs, and for a graph $G$ let $N_{G}(k,\ell)$ be the number of sets of $k$ vertices of $G$ inducing exactly $\ell$ edges. Then, define the \emph{edge-inducibility} (with parameters $k$ and $\ell$) by
\[
\on{ind}(k,\ell)=\lim_{n\to\infty}\max_{G\in\mathcal{G}_{n}}\frac{N_{G}(k,\ell)}{\binom{n}{k}}.
\]
This parameter measures, for large graphs,
the maximum possible fraction of
$k$-vertex subsets which induce $\ell$ edges. By considering complete or empty graphs we have 
\[
\on{ind}(k,0)=\on{ind}\bigl(k,\textstyle{\binom{k}{2}}\bigr)=1,
\]
but for $0<\ell<\binom k 2$ we have $\on{ind}(k,\ell)<1$  (this follows easily from \emph{Ramsey's theorem}, which says that large graphs must have large complete or empty subgraphs). In fact,
for $0<\ell<\binom k 2$, and large $k$, we now know that $\on{ind}(k,\ell)$
cannot be much larger than $1/e$; this was the content of the \emph{Edge-Statistics conjecture}, 
proved in a combination of papers by Kwan, Sudakov and Tran~\cite{KST19},
Fox and Sauermann~\cite{FS20}, and Martinsson, Mousset, Noever and
Truji\'c~\cite{MMNT19}. One can use \cref{thm:main-shiny} to prove the following much stronger bound when $\ell$ is far from
zero and far from $\binom{k}{2}$, which was conjectured by Alon, Hefetz, Krivelevich and Tyomkyn (see \cite[Conjecture~6.2]{AHKT20}):
\begin{theorem}\label{thm:AHKT}
For $0<\ell<\binom k 2$ we have
\[\on{ind}(k,\ell)=O\left(\frac 1{\sqrt{\min\bigl(\ell,\binom k2-\ell\bigr)/k}}\right).\]
\end{theorem}
The deduction of \cref{thm:AHKT} from \cref{thm:main-shiny} is exactly as in the proof of \cite[Theorem~1.1]{KST19} (which uses the weaker quadratic Littlewood--Offord inequality from \cite{MNV16}), so we do not include it here. The idea is that for any large graph $G$, the number of edges in a random set of $k$ vertices of $G$ can be interpreted as a quadratic polynomial of independent Rademacher random variables (via a certain coupling).

\textit{Notation.} As, usual, for a nonnegative integer $n$, we write $[n]:=\{1,\dots,n\}$ (note that for $n=0$, this means $[n]=\emptyset$). For an $n\times n$ matrix $A$ and subsets $I,J\su [n]$, we write $A[I\times J]$ for the $|I|\times |J|$ submatrix of $A$ consisting of the rows with indices in $I$ and the columns with indices in $J$. Similarly for a vector $\vec v\in \RR^n$ and a subset $I\su [n]$ we write $\vec v[I]\in \RR^{I}$ for the vector obtained from $\vec v$ by only taking the coordinates with indices in $I$. Finally, for a vector $\vec v\in \RR^n$ and $i\in [n]$, we write $\vec v[i]\in \RR$ for the $i$-th entry of $\vec v$.

In this paper, we say that $Q\in \mathbb{R}[x_1,\dots,x_n]$ is a quadratic polynomial if its degree is \emph{at most} $2$. For $t>0$, we write $\log(t)$ for the base-2 logarithm of $t$.

\section{Key ideas, in comparison with previous work}\label{sec:outline}

By now there are several different proofs of the $O(1/\sqrt{n})$
bound in the (linear) Erd\H os--Littlewood--Offord theorem. As far as we
know, they all take advantage of at least one of two very special
properties of random variables of the form $X=a_{1}\xi_{1}+\dots+a_{n}\xi_{n}$.
First, Erd\H os' original proof~\cite{Erd45} used the \emph{monotonicity} of $X$
(for every index $i$, changing $\xi_{i}$ from $-1$ to $1$ will always make the value
of $X$ increase, or always make the value of $X$ decrease). Second,
one can take advantage of the fact that $X$ is a \emph{sum of independent
random variables} (each with a very simple distribution), so its Fourier transform is very well-behaved.
(The first Fourier-analytic proof seems to have been by Hal\'asz~\cite{Hal77};
see also the very simple proof of a $O(1/\sqrt n)$ bound due to Croot~\cite{Cro11}).

Unfortunately, in the quadratic setting (where $X=Q(\xi_{1},\dots,\xi_{n})$
for some quadratic polynomial $Q$), both of the above properties
of $X$ may fail in a very strong way. There
are two general approaches that have been most successful so far:
Gaussian approximation, and a technique called \emph{decoupling}.
We briefly discuss both these approaches and their limitations, before
describing our new ideas.

\subsection{Gaussian approximation and combinatorial partitioning}

Whether one is interested in the linear or the quadratic cases of the Littlewood--Offord problem, perhaps the most natural starting point is to try to leverage some of the vast literature
in probability theory on distributional approximation: if one can approximate the entire distribution of
a random variable $X$, then anticoncentration should be an easy corollary.

This angle of attack is especially compelling in the linear case, since $X=a_{1}\xi_{1}+\dots+a_{n}\xi_{n}$ is a sum of independent random variables; it is tempting to try to apply a central limit theorem. One cannot be too na\"ive here, as the limiting distribution of $X$ could actually be very far from Gaussian (consider for example the case where $a_{i}=2^{i}$ for each $i$). However, in their foundational paper, Littlewood and Offord~\cite{LO43}
were in fact able to prove their $O(\log n/\sqrt{n})$ bound via Gaussian
approximation. The key idea was to partition the coefficients $a_{i}$
into $O(\log n)$ ``buckets'' according to their orders of
magnitude, in such a way that Gaussian approximation is effective
within each bucket.

It is far from obvious how to extend this type of strategy to the
higher-degree case (when $X=Q(\xi_{1},\dots,\xi_{n})$ for a
bounded-degree polynomial $Q$), but this is more or less what was
accomplished by Meka, Nguyen and Vu~\cite{MNV16} and by Kane~\cite{Kan14}
in their bounds of the form $\exp(O((\log\log n)^{2}))/\sqrt{n}$
and $(\log n)^{O(1)}/\sqrt{n}$, respectively. Instead of a central
limit theorem, one needs a \emph{Gaussian invariance principle} (which
provides sufficient conditions under which one can approximate a polynomial
of independent Rademacher random variables by a polynomial of independent
Gaussian random variables), and instead of a simple ``bucketing''
argument one needs a \emph{regularity lemma} to describe $Q(\xi_{1},\dots,\xi_{n})$
in terms of a ``low-complexity decision tree''.

These types of methods are very powerful and very flexible, but unfortunately
it seems that one inevitably needs to ``pay logarithmic factors''
in order to deconstruct an arbitrary\footnote{If one makes strong assumptions about the quadratic polynomial, it
is sometimes possible to prove exact bounds via Gaussian approximation;
see for example \cite{KSSS23}.} quadratic polynomial into ``well-behaved pieces'' for which Gaussian
approximation is effective. Even in the linear case, we are not aware of a way to prove an optimal $O(1/\sqrt n)$ bound via Gaussian approximation. Perhaps this is not surprising: in some sense the entire philosophy of the Littlewood--Offord problem is that anticoncentration is an extremely robust phenomenon that holds under much weaker assumptions than central limit theorems, and we should not expect to be able to \emph{use} central limit theorems to prove optimal anticoncentration.

\subsection{Decoupling}\label{subsec:decoupling}

A completely different technique was employed in the papers of Costello, Tao and Vu~\cite{CTV06} and
Rosi\'nski and Samorodnitsky~\cite{RS96} which first studied the quadratic Littlewood--Offord problem. This technique is now usually called \emph{decoupling}, following \cite{CTV06}. Roughly speaking, decoupling is a general technique to ``reduce a quadratic anticoncentration problem to a linear one''\footnote{The term ``decoupling'' refers more generally to a class of techniques used to reduce from dependent situations to independent situations, see for example the book-length treatment in \cite{dlPG12}.}. Below we sketch the basic idea, incorporating an improvement due to Costello and Vu~\cite{CV08}\footnote{\cite{CTV06,RS96} featured a ``more symmetric'' decoupling argument, which gives worse bounds.}.

Let $Q\in \RR[x_1,\dots,x_n]$ be a quadratic polynomial and let $\vec{\xi}=(\xi_{1},\dots,\xi_{n})$ be a sequence of independent Rademacher
random variables. Given a partition of the index set $\{1,\dots,n\}$
into two subsets $I$ and $J$, we can break the random vector $\vec{\xi}$ into two parts 
$\vec{\xi}[I]\in\{-1,1\}^{I}$ and $\vec{\xi}[J]\in\{-1,1\}^{J}$.
Then, a simple application of the Cauchy--Schwarz inequality (see
\cref{lem:decoupling}) shows that if $\pvec\xi[I]$ is an independent copy of $\vec{\xi}[I]$,
we have
\begin{align}
\Pr\Bigl[Q(\vec \xi)=z\Bigr]=\Pr\Bigl[Q(\vec{\xi}[I],\vec{\xi}[J])=z\Bigr] & \le\Pr\Bigl[Q(\vec{\xi}[I],\vec{\xi}[J])=z\text{ and }Q(\pvec \xi[I],\vec{\xi}[J])=z\Bigr]^{1/2}\label{eq:decoupling-basic-step}\\
 & \le \Pr\Bigl[Q(\vec{\xi}[I],\vec{\xi}[J])-Q(\pvec \xi[I],\vec{\xi}[J])=0\Bigr]^{1/2}.\label{eq:linear-quadratic-decoupling}
\end{align}
Now, $Q(\vec{\xi}[I],\vec{\xi}[J])-Q(\pvec \xi[I],\vec{\xi}[J])$
can be interpreted as a \emph{linear} function of $\vec{\xi}[J]$,
with coefficients that depend on $(\vec{\xi}[I],\pvec \xi[I])$ (since the terms that are quadratic in $\vec \xi[J]$ cancel out). Furthermore, this linear function typically has many nonzero coefficients (since it is unlikely that most of the ``cross terms'' between $\vec \xi[I]$ and $\pvec \xi[I]$ cancel out). So, after conditioning on a typical outcome of $(\vec{\xi}[I],\pvec \xi[I])$, one can apply the Erd\H os--Littlewood--Offord theorem, to obtain
a bound of the form
$\Pr[Q(\vec \xi)=z]\le (O(1/\sqrt n))^{1/2}=O(n^{-1/4})$.

The great advantage of decoupling is that the resulting \emph{linear} anticoncentration problem is much easier: one has access to the much wider range of tools available to study sums of independent random variables. However, the inequality between \cref{eq:decoupling-basic-step} and \cref{eq:linear-quadratic-decoupling} is rather lossy, and one tends to end up with bounds that are at least ``a square root away'' from best-possible\footnote{But we remark that there are some circumstances where one can ``do decoupling in Fourier space'' in such a way that the resulting square-root loss in Fourier space corresponds to a much smaller loss in physical space, see \cite{FS20,KSSS23}.}.

Costello's paper~\cite{Cos13}, proving the first bound of the form $n^{-1/2+o(1)}$ for the quadratic Littlewood--Offord problem, combined decoupling with a structural dichotomy. Namely, Costello discovered that if the quadratic polynomial $Q$ is in a certain sense ``robustly irreducible'', then number-theoretic arguments give the stronger bound $\Pr[Q(\vec{\xi}[I],\vec{\xi}[J])-Q(\pvec \xi[I],\vec{\xi}[J])=0]\le n^{-1+o(1)}$, and so even after the ``square-root loss'' of decoupling one has $\Pr[Q(\vec{\xi}[I],\vec{\xi}[J])=z]\le n^{-1/2+o(1)}$. He then gave a separate argument to handle the case where $Q$ does not satisfy the robust irreducibility condition (i.e., when $Q$ essentially splits into linear factors), based on the Szemer\'edi--Trotter theorem from discrete geometry.

We cannot conclusively rule out the possibility that one could prove an optimal $O(1/\sqrt n)$ bound with a similar kind of case analysis, but this seems to be extremely difficult. In particular, we are not aware of any suitable candidate for a condition on $Q$ which ensures that $\Pr[Q(\vec{\xi}[I],\vec{\xi}[J])-Q(\pvec \xi[I],\vec{\xi}[J])=0]\le O(1/n)$ (Costello's notion of robust irreducibility, and similar notions of ``robust rank'' or ``robust sum-of-squares complexity'', are only suitable for an $n^{-1+o(1)}$ bound on such probabilities). Also, Costello's Szemer\'edi--Trotter-based proof for the nearly-reducible case does not seem to easily generalise to other simple-looking families of quadratic polynomials (for example, polynomials which can be written as the sum of four squares; see the discussion in \cite[Section~10]{FKS23}).

\subsection{A geometric point of view, and an inductive decoupling scheme}
\label{subsec:decoupling-scheme}

In this paper, we consider a different perspective on decoupling: instead of using decoupling to immediately reduce from a quadratic problem to a linear one, we reinterpret decoupling as a tool to obtain a problem with a linear and a quadratic part, and to inductively ``reduce the proportion of our problem that is quadratic''. To explain this, it is helpful to take a geometric perspective.

Specifically, for a quadratic polynomial $Q\in \RR[x_1,\dots,x_n]$ and a random vector $\vec{\xi}\in \{1,-1\}^n$, note that $\Pr[Q(\vec \xi)=z]$ can be interpreted as the probability that $\vec \xi$ lies in the quadric (quadratic variety) $\mc Z$ given by $\mc Z=\{\vec x \in \RR^n: Q(\vec x)=z\}\subseteq\RR^n$. Similarly, the expression $\Pr\bigl[Q(\vec{\xi}[I],\vec{\xi}[J])=z\text{ and }Q(\pvec \xi[I],\vec{\xi}[J])=z\bigr]$ appearing in \cref{eq:decoupling-basic-step} can be interpreted as the probability that $\vec \xi[J]$ lies in the variety
\[\mc Z^{(1)}=\Bigr\{\vec x\in \RR^J: Q(\vec{\xi}[I],\vec x)=z\text{ and }Q(\pvec \xi[I],\vec x)=z\Bigr\}=\mathcal Z_{\vec \xi[I]}\cap \mathcal Z_{\pvec \xi[I]}\subseteq \RR^J,\]
where, for $\vec u\in \RR^I$, we write $\mathcal Z_{\vec u}$ for the set of all $\vec x\in \RR^J$ with $(\vec u,\vec x)\in \mc Z$ (typically, this is a quadric in $\RR^J$). Using this language, \cref{eq:decoupling-basic-step} can be restated as
\begin{equation}
\Pr\Bigl[Q(\vec \xi)=z\Bigr]=\Pr\Bigl[\vec \xi\in \mc Z\Bigr]\le \Pr\Bigl[\vec \xi[J]\in \mc Z^{(1)}\Bigr]^{1/2}.\label{eq:decoupling-restatement}
\end{equation}
The next step leading to the traditional decoupling inequality \cref{eq:linear-quadratic-decoupling} can geometrically be phrased as observing that $\mc Z^{(1)}$ lies inside the affine-linear subspace
\[\mc W^{(1)}=\Bigr\{\vec x \in \RR^J :Q(\vec{\xi}[I],\vec{x})-Q(\pvec \xi[I],\vec{x})=0\Bigr\}.\]
Indeed, this yields the probability bound
\[\Pr\Bigl[Q(\vec \xi)=z\Bigr]\le \Pr\Bigl[\vec \xi[J]\in \mc Z^{(1)}\Bigr]^{1/2}\le \Pr\Bigl[\vec \xi[J]\in \mc W^{(1)}\Bigr]^{1/2}\] stated in \cref{eq:linear-quadratic-decoupling}. 
One can then forget about $\mc Z^{(1)}$ and restrict one's attention to the affine-linear subspace $\mc W^{(1)}$, where the relevant probabilities are  easier to analyse (under suitable assumptions on $Q$,
one can show that $\Pr[\vec \xi[J]\in \mc W^{(1)}]\le O(1/\sqrt{n})$, leading to the bound $\Pr[Q(\vec \xi)=z]\le O(n^{-1/4})$ described in \cref{subsec:decoupling}).

However, it turns out that this ``forgetting'' of $\mathcal{Z}^{(1)}$ is
precisely the cause of the square-root loss usually associated
with decoupling. Indeed, $\mathcal W^{(1)}$ is a variety with codimension
1 (being defined by a single equation), while $\mathcal{Z}^{(1)}$ is typically\footnote{There are certain degenerate situations where $\mc Z^{(1)}$ does not have codimension at least 2. We will ignore degeneracies of this type throughout this outline, but they do cause challenges in the actual proof.} a variety with codimension (at least) 
2 (being defined by two equations), so
intuitively we should be much less likely to have $\vec{\xi}[J]\in\mathcal{Z}^{(1)}$
than $\vec{\xi}[J]\in\mathcal{W}^{(1)}$. More specifically, in order to have $\vec{\xi}[J]\in\mathcal{Z}^{(1)}$,
we need $\vec{\xi}[J]$ to satisfy two different equations simultaneously, and we might expect each of these to be satisfied with probability $O(1/\sqrt n)$. So for typical
outcomes of $\vec \xi[I],\pvec \xi[I]$ (which determine $\mc Z^{(1)}$) we might expect $\Pr\bigl[\vec{\xi}[J]\in\mathcal{Z}^{(1)}\,\big|\,\vec \xi[I],\pvec \xi[I]\bigr]\le\left(O(1/\sqrt{n})\right)^{2}$. If we could prove this, we would be able to deduce \cref{thm:main-shiny} (recalling \cref{eq:decoupling-restatement}).

So, we choose not to ``forget'' $\mathcal{Z}^{(1)}$, and our task
is to show that $\Pr[\vec{\xi}[J]\in\mathcal{Z}^{(1)}]\le\left(O(1/\sqrt{n})\right)^{2}$.
While this new task may seem harder than the previous one (as $\mathcal{Z}^{(1)}$ seems like a more complicated object than $\mathcal{Z}$), the key observation is that we have ``reduced the relative proportion
of the quadratic part of our problem''. Indeed, at the start we were interested in a
quadric $\mathcal{Z}$ described by a single quadratic equation, but now we are interested in $\mathcal{Z}^{(1)}$, which can be interpreted as a quadric \emph{inside} the affine-linear subspace $\mathcal{W}^{(1)}\subseteq \RR^J$. That is to say, $\mathcal{Z}^{(1)}$ is described by one linear and one quadratic equation, so now ``only half of our problem is quadratic''.

Crucially, it is possible to iterate this entire procedure: we next fix a partition
$I^{(2)}\cup J^{(2)}$ of $J$, and consider the variety
\[\mathcal{Z}^{(2)}=\mathcal{Z}^{(1)}_{\vec \xi[I^{(2)}]}\cap\mathcal{Z}^{(1)}_{\pvec\xi[I^{(2)}]}\su \RR^{J^{(2)}}\]
(where, for $\vec u\in \RR^{I^{(2)}}$, we write $\mc Z^{(1)}_{\vec u}$ for the set of all $\vec x\in \RR^{J^{(2)}}$ with $(\vec u,\vec x)\in \mc Z^{(1)}$). Now, decoupling, analogously to the inequality in \cref{eq:decoupling-restatement}, yields
\begin{equation}\Pr\Bigl[\vec \xi[J]\in \mc Z^{(1)}\,\Big|\,\vec \xi[I],\pvec \xi[I]\Bigr]\le \Pr\Bigl[\vec \xi[J^{(2)}]\in \mc Z^{(2)}\,\Big|\,\vec \xi[I],\pvec \xi[I]\Bigr]^{1/2}.\label{eq:decoupling-illustration-2}\end{equation}

So, it suffices to show that $\Pr\bigl[\vec \xi[J^{(2)}]\in \mc Z^{(2)}\,\big|\,\vec \xi[I],\pvec \xi[I]\bigr]\le \left(O(1/\sqrt n)\right)^4$ for typical outcomes of $\vec \xi[I]$ and $\pvec \xi[I]$.
Now, for $\mathcal{Z}^{(2)}$ to be nonempty, it must be the case that $\mathcal{W}^{(1)}_{\vec \xi[I^{(2)}]}\cap\mathcal{W}^{(1)}_{\pvec\xi[I^{(2)}]}$ is nonempty, or equivalently that $\mathcal{W}^{(1)}_{\vec \xi[I^{(2)}]}=\mathcal{W}^{(1)}_{\pvec \xi[I^{(2)}]}$ (it is not hard to see that $\mathcal{W}^{(1)}_{\vec \xi[I^{(2)}]}$ and $\mathcal{W}^{(1)}_{\pvec \xi[I^{(2)}]}$ are parallel translates of each other).
That is to say, $\vec \xi[I^{(2)}]-\pvec{\xi}[I^{(2)}]$ must lie in a certain affine-linear subspace which typically has codimension 1; this happens with probability $O(1/\sqrt n)$ by the (linear) Erd\H os--Littlewood--Offord theorem.

If we also condition on outcomes of $\vec \xi[I^{(2)}]$ and $\pvec{\xi}[I^{(2)}]$ such that $\mathcal{W}^{(1)}_{\vec \xi[I^{(2)}]}=\mathcal{W}^{(1)}_{\pvec \xi[I^{(2)}]}$, it is not hard to see that $\mc Z^{(2)}$ is typically a quadric inside the affine-linear subspace $\mathcal W^{(2)}$ of $\mathcal{W}^{(1)}_{\vec \xi[I^{(2)}]}=\mathcal{W}^{(1)}_{\pvec \xi[I^{(2)}]}\su \RR^{J^{(2)}}$ given by the linear equation $Q(\vec \xi[I],\vec \xi[I^{(2)}],\vec x)-Q(\vec \xi[I],\pvec \xi[I^{(2)}],\vec x)=0$ (in much the same way that $\mc Z^{(1)}$ is a quadric inside the affine-linear subspace $\mc W^{(1)}$). That is to say, $\mc Z^{(2)}$ is typically a variety of codimension (at least) 3, described by two linear equations and one quadratic equation. So we might expect that (for typical outcomes of $\vec \xi[I],\pvec{\xi}[I],\vec \xi[I^{(2)}],\pvec{\xi}[I^{(2)}]$ for which $\mathcal{W}^{(1)}_{\vec \xi[I^{(2)}]}=\mathcal{W}^{(1)}_{\pvec \xi[I^{(2)}]}$)
\begin{equation}\Pr\Bigl[\vec \xi[J^{(2)}]\in\mathcal{Z}^{(2)}\,\Big|\,\vec \xi[I],\pvec{\xi}[I],\vec \xi[I^{(2)}],\pvec{\xi}[I^{(2)}]\Bigr]\le\left(O(1/\sqrt{n})\right)^{3}.\label{eq:conditional-hope}\end{equation}

If we were able to prove \cref{eq:conditional-hope}, we would obtain a bound of the form
\[\Pr\Bigl[\vec \xi[J^{(2)}]\in \mc Z^{(2)}\Bigr]\le O(1/\sqrt{n})\cdot \left(O(1/\sqrt{n})\right)^{3}=\left(O(1/\sqrt{n})\right)^{4},\]
which would imply \cref{thm:main-shiny}, tracing back through our decoupling inequalities \cref{eq:decoupling-restatement} and \cref{eq:decoupling-illustration-2}. We have made progress by ``reducing the proportion of our problem that is quadratic'': If, instead of \cref{eq:conditional-hope}, we were only able to prove that
\[\Pr\Bigl[\vec \xi[J^{(2)}]\in\mathcal{Z}^{(2)}\,\Big|\,\vec \xi[I],\pvec{\xi}[I],\vec \xi[I^{(2)}],\pvec{\xi}[I^{(2)}]\Bigr]\le \Pr\Bigl[\vec \xi[J^{(2)}]\in\mathcal{W}^{(2)}\,\Big|\,\vec \xi[I],\pvec{\xi}[I],\vec \xi[I^{(2)}],\pvec{\xi}[I^{(2)}]\Bigr]\le\left(O(1/\sqrt{n})\right)^{2}\]
(``forgetting'' the quadratic part of the problem and only bounding the probability that $\vec \xi[J^{(2)}]$ lies in the affine-linear subspace $\mathcal{W}^{(2)}$ of codimension $2$), we would end up with a final bound of the form  $\Pr[\vec \xi\in \mc Z]\le O(n^{-3/8})$, which is much better than the $O(n^{-1/4})$ bound we obtained with a single decoupling step.

In general, after $k$ steps of this scheme, we will have considered $k$ ``nested'' partitions of the form $J^{(i-1)}=I^{(i)}\cup J^{(i)}$, and defined $k$ varieties $\mc Z^{(1)},\dots,\mc Z^{(k)}$. We will have applied the decoupling inequality $k$ times, and considered various conditional probabilities that the vectors $\vec \xi[I^{(i)}]-\pvec{\xi}[I^{(i)}]$ lie in certain affine-linear subspaces (to ensure that certain intersections $\mathcal{W}^{(i-1)}_{\vec \xi[I^{(i)}]}\cap\mathcal{W}^{(i-1)}_{\pvec\xi[I^{(i)}]}$ are nonempty). After all this, we find ourselves in a position where if we ``forget the quadratic part of the problem'' we obtain a bound of the form $\Pr[Q(\vec \xi)=z]\le O(n^{-1/2+1/2^{k+1}})$. That is to say, if $k$ is at least $\log \log n$, the quadratic part of the problem is so insignificant that ``forgetting'' it only costs us a constant factor in the final bound.

At an extremely high level, this explains our strategy to prove \cref{thm:main-shiny}. However, we omitted a number of important details in this outline. In particular, our strategy heavily depends on being able to obtain suitable upper bounds on probabilities that certain random vectors fall into certain affine-linear subspaces, and there are crucial ``robust rank'' nondegeneracy conditions that must be satisfied in order for such bounds to hold. This is not just a technicality; we require significant new ideas to maintain these robust rank properties during our iterative decoupling scheme, which we discuss next.

\begin{remark}\label{rem:high-degree-issues}
\cref{thm:main-shiny} concerns quadratic polynomials, but one may be interested in a generalisation to cubic polynomials, or to polynomials of any fixed degree. Decoupling still makes sense for general polynomials: for example, if $\mathcal Z$ is a cubic (variety), then decoupling yields an inequality of the form $\Pr[\vec \xi\in \mc Z]\le \Pr[\vec\xi[J]\in \mc Z^{(2)}]^{1/2}$, where $\mc Z^{(2)}$ is the intersection of a quadric and a cubic. As observed by Rosi\'nski and Samorodnitsky~\cite{RS96} and Razborov and Viola~\cite{RV13}, the basic type of decoupling argument described in \cref{subsec:decoupling} generalises quite straightforwardly to higher degrees (but the bounds get worse as the degree increases). However, it is less clear how to generalise the inductive decoupling scheme described in this subsection. Roughly speaking, in the degree-$d$ case, instead of affine-linear subspaces (obtained as intersections of affine-linear hyperplanes), one must work with intersections of degree-$(d-1)$ varieties, which are much more complicated objects. We hope that the relevant complexities can be handled with some kind of multiple-level induction, but so far we were not able to accomplish this.
\end{remark}

\subsection{High-dimensional anticoncentration inequalities, and witness-counting}
\label{subsec:witness-count}

Our proof, as outlined in the previous subsection, relies on bounds on probabilities that random vectors lie in certain affine-linear subspaces. More specifically, for a suitably nondegenerate affine-linear subspace $\mc W\subseteq \RR^n$ of codimension $k$, and a uniformly random vector $\vec{\xi}\in \{1,-1\}^n$, we need a probability bound of the form $\Pr[\vec \xi\in \mc W]\le O(n^{-k/2})$. Intuitively, this is because $\vec \xi$ needs to simultaneously satisfy $k$ different linear equations, each of which is satisfied with probability roughly $n^{-1/2}$. More formally, such a bound follows from a high-dimensional version of the Erd\H os--Littlewood--Offord theorem.

The first such high-dimensional version was due to Hal\'asz~\cite{Hal77}. In linear-algebraic language, it can be phrased as follows: for any fixed $k$, if $M\in \mb R^{k\times n}$ is a matrix that ``robustly has rank $k$'' in the sense that (for some fixed $\delta>0$) one cannot delete $\delta n$ columns of $M$ to obtain a matrix with rank less than $k$, then for a uniformly random vector $\vec \xi\in\{-1,1\}^n$ we have
\[\sup_{\vec w\in \mb R^k}\Pr[M\vec \xi=\vec w]\le O(n^{-k/2}).\]
Note that some kind of ``robust rank $k$'' condition is necessary here: for example, if $n$ is even and $M\in \RR^{2\times n}$ has rows $(1,\dots,1)\in \RR^n$ and $(1,\dots,1,0,0)\in \RR^n$, then it is easy to check that $\Pr[M\vec \xi=\vec 0]$ has order of magnitude $1/\sqrt n$.

Several extensions and variants of Halasz' inequality have since been proved (see for example \cite{FJZ22,FKS21b,HO,TV06}); in particular, Ferber, Jain and Zhao~\cite{FJZ22} proved a version of Halasz' theorem with a much better dependence on $k$ (allowing $k$ to vary with $n$, instead of viewing it as a constant). We state (a corollary of) this theorem as \cref{thm-Halasz-FJZ}.

Of course, whenever we want to apply any Hal\'asz-type theorem, we need a ``robust rank'' condition to hold. So, in order to execute the strategy described in the last subsection, at each step of the decoupling scheme we need a ``robust rank inheritance'' lemma, proving that a robust rank condition is likely to hold for the next step, given that it holds for the current step. The key ingredient for our robust rank inheritance lemma is a new high-dimensional anticoncentration inequality for the probability that a random vector falls in a small ball in the \emph{Hamming} norm. We believe this inequality (and the techniques in its proof) to be of independent interest; an important special case is as follows.

\begin{lemma}\label{lem:baby-key}
For any fixed positive integer $r$, there are constants $C_r>0$ and $c_r>0$ only depending on $r$ such that the following holds. Consider a matrix $A\in \RR ^{m\times n}$ which has rank at least $r$ after deletion of any $t$ rows and $t$ columns (for some positive integer $t$).
Then for a sequence $\vec{\xi}=(\xi_{1},\dots,\xi_{n})\in\{-1,1\}^{n}$ of independent Rademacher random variables,
we have
\[
\sup_{\vec v\in \mb R^m}\Pr[A\vec{\xi}\text{ differs from }\vec{v}\text{ in fewer than }c_r t\text{ coordinates}]\le C_r \cdot t^{-r/2}.
\]
\end{lemma}

The assumption in \cref{lem:baby-key} says that $A$ robustly has rank at least $r$, but in a stronger sense than typical Hal\'asz-type theorems: the rank needs to remain at least $r$ after \emph{row} deletion as well as column deletion. As a result, we are able to obtain a stronger conclusion, namely that for any vector $\vec v$, it is unlikely that $A\vec \xi$ agrees with $\vec v$ in almost all its coordinates. (It is not hard to see that for this stronger conclusion, such a row-deletion assumption is indeed necessary).

We prove \cref{lem:baby-key} (and our more general robust rank inheritance lemma) in \cref{sec:key-lemma} using a \emph{witness-counting} technique, which we outline here. First, note that the most na\"ive strategy to prove \cref{lem:baby-key} would be to simply take a union bound over all sets $I$ of $m-c_r t$ coordinates in which $A\vec \xi$ and $\vec v$ could agree. For some specific $I$, the probability that $A\vec \xi$ and $\vec v$ agree on the coordinates indexed by $I$ can be easily understood using existing tools (i.e., Hal\'asz' inequality and its variants) and is at most on the order of $t^{-r/2}$. However, it is far too wasteful to simply take a union bound summing over all possibilities for $I$ (the number of possibilities is exponential in $t$). 

Instead, for each $I$ we consider small ``witness'' subsets $I'\subseteq I$, such that the submatrix of $A$ consisting of the rows with indices in $I'$ still has (robustly) high rank. Note that for each $I'\su I$, whenever $A\vec \xi$ and $\vec v$ agree on the coordinates indexed by $I$, they also in particular agree on the coordinates indexed by $I'$. Given the high-rank property, for each ``witness'' subset $I'$ we can still easily bound the probability that $A\vec \xi$ and $\vec v$ agree on the coordinates indexed by $I'$ (using Hal\'asz' inequality and its variants).

There are still too many possible ``witness'' subsets $I'$ to be able to simply take a union bound over all possibilities for $I'$, but (roughly speaking) we can show that whenever $A\vec \xi$ and $\vec v$ agree on a large set of coordinates $I$, then they must agree on the coordinates of \emph{many} ``witness'' subsets $I'\su I$.  We can show this is unlikely by computing the expected number of ``witness'' coordinate-subsets on which $A\vec \xi$ and $\vec v$ agree, and applying Markov's inequality.

\begin{remark}\label{rem:small-ball-issues}
One might be interested in adapting \cref{thm:main-shiny} to study small-ball probabilities of the form $\sup_{z\in \RR}\Pr[|Q(\xi_{1},\dots,\xi_{n})-z|\le 1]$ instead of point probabilities $\sup_{z\in \RR}\Pr[Q(\xi_{1},\dots,\xi_{n})=z]$. The robust rank inheritance lemma seems to be the main point of difficulty for such an adaptation; it is not clear how to extend the witness-counting arguments to the small-ball setting (e.g., in the setting of \cref{lem:baby-key}, we would be interested in the event that there are fewer than $c_r t$ coordinates $i$ for which $|(A\vec \xi-\vec v)[i]|\le 1$).
\end{remark}

Before ending this overview section, we remark that there is a way to sidestep the robust rank inheritance issue in the special case where $Q$ has ``bounded rank'', meaning that the quadratic part of $Q$ can be written as $\vec x^\transpose A\vec x$ for some symmetric matrix $A$ of rank $O(1)$. Indeed, in this case we can reduce our entire problem to a certain \emph{bounded-dimensional} geometric anticoncentration problem (involving a quadric), where the robust rank conditions for Halasz' inequality are always automatically satisfied when following the strategy in the previous subsection. In \cref{sec:geometric}, we give a simple self-contained proof of an essentially optimal anticoncentration bound in this setting. We believe this to be a good illustration of the basic principles of our inductive decoupling scheme, with minimal technicalities (the results of \cref{sec:geometric} will actually also be used later in the paper).

\section{Preliminaries}

First, as outlined in \cref{sec:outline}, we need a ``high-dimensional'' version of the Erd\H os--Littlewood--Offord theorem. This will require a robust rank condition, defined as follows.

\begin{definition}\label{def:Halasz}
For integers $0\le k\le n$ and a real number $s\ge 0$, let $\mathcal{H}^{k\times n}(s)$ be the set of matrices $M\in\RR^{k\times n}$ such that $\on{rank} M[[k]\times J]=k$ for all subsets $J\su [n]$ of size $|J|\ge n-s$, i.e., the set of matrices having rank $k$ after any deletion of up to $s$ columns.
\end{definition}

We clearly have $\mathcal{H}^{k\times n}(s)\su \mathcal{H}^{k\times n}(s')$ if $s\ge s'$. Furthermore, for a partition $[n]=S\cup I$ with $|S|\le s$, for every matrix $M\in \mathcal{H}^{k\times n}(s)$ we also have $M[[k]\times I]\in \mathcal{H}^{k\times I}(s-|S|)$. Also note that in the case $k=0$, the (unique) empty matrix $M\in\RR^{0\times n}$ is contained in $\mathcal{H}^{0\times n}(s)$ for all $s\ge 0$.

For integers $0\le k\le n$ and $t\ge 0$, we say that a matrix $M\in \RR^{k\times n}$ contains $t$ disjoint nonsingular $k\times k$ submatrices if there exist disjoint subsets $J_1,\dots, J_t\su [n]$ of size $|J_1|=\dots=|J_t|=k$ such that $\on{rank} M[[k]\times J_i]=k$ for $i=1,\dots,t$. This is the case if and only if among the column vectors $\vec{a}_1,\dots,\vec{a}_n$ of $M$ we can form $t$ disjoint bases of $\RR^{k}$ (here, the vectors $\vec{a}_1,\dots,\vec{a}_n$ are considered with multiplicities, i.e., a vector in $\RR^{k}$ can occur in two different bases if it occurs twice among $\vec{a}_1,\dots,\vec{a}_n$).

\begin{lemma}\label{lem-halasz-equivalent}
For integers $0\le k\le n$ and a real number $s\ge 0$, the following statements hold:
\begin{compactitem}
\item[(i)] If $M\in \RR^{k\times n}$ contains $t$ disjoint nonsingular $k\times k$ submatrices for an integer $t>s$, then $M\in \mathcal{H}^{k\times n}(s)$.
\item[(ii)] If $M\in \mathcal{H}^{k\times n}(s)$ and $k\ge 1$, then $M$ contains $\lceil s/k\rceil$ disjoint nonsingular $k\times k$ submatrices.
\end{compactitem}
\end{lemma}
\begin{proof}
    For (i), let $J_1,\dots, J_t\su [n]$ be disjoint subsets of size $|J_1|=\dots=|J_t|=k$ such that $\on{rank} M[[k]\times J_i]=k$ for $i=1,\dots,t$. Then for any subsets $J\su [n]$ of size $|J|\ge n-s>n-t$, we have $J_i\su J$ for some $i\in \{1,\dots,t\}$ and hence $k\ge \on{rank} M[[k]\times J]\ge \on{rank} M[[k]\times J_i]=k$, meaning that $\on{rank} M[[k]\times J]=k$.

    For (ii), we can greedily find disjoint subsets $J_1,\dots, J_{\lceil s/k\rceil}\su [n]$ of size $|J_1|=\dots=|J_{\lceil s/k\rceil}|=k$ such that $\on{rank} M[[k]\times J_i]=k$ for $i=1,\dots,\lceil s/k\rceil$. Indeed, after having chosen $J_1,\dots, J_{i-1}$ for some $i\le \lceil s/k\rceil$, we have $|J_1\cup \dots\cup J_{i-1}|=(i-1)k<s$ and hence $M$ still has rank $k$ after deleting the columns with indices in $J_1\cup \dots\cup J_{i-1}$. So we can find a subset $J_i\su [n]\sm (J_1\cup \dots\cup J_{i-1})$ of size $|J_i|=k$ with $\on{rank} M[[k]\times J_i]=k$.
\end{proof}

Hal\'asz~\cite{Hal77} proved that for fixed constants $\delta>0$ and $k\in \mb{N}$, if $M\in \mathcal{H}^{k\times n}(\delta n)$ then $\Pr[M\vec \xi=\vec w]\le O(n^{-k/2})$. We will use the following quantitative version of Hal\'asz' theorem, which is a special case\footnote{Specifically, to deduce \cref{thm-Halasz-FJZ} from \cite[Theorem 1.11]{FJZ22}, we take the even number $\ell$ in \cite[Theorem 1.11]{FJZ22} to be $\ell=t$ if $t$ is even, and $\ell=t-1$ if $t$ is odd. In both cases, we can verify the inequality $2^{-\ell}\binom{\ell}{\ell/2}\le t^{-1/2}$, and we obtain the desired bound by taking $\mathcal{A}_1,\dots,\mathcal{A}_\ell$ in \cite[Theorem 1.11]{FJZ22} to be a partition of the columns of $M$ into $\ell$ disjoint subsets each containing a basis of $\RR^k$ (then $r_1=\dots=r_{\ell}=k$).} of a result of Ferber, Jain and Zhao~\cite[Theorem 1.11]{FJZ22} (see also \cite{HO} for a previous result with a weaker dependence on $k$).

\begin{theorem}\label{thm-Halasz-FJZ}
Let $0\le k\le n$ and $t\ge 1$ be integers. Consider a matrix $M\in \RR^{k\times n}$ containing $t$ disjoint nonsingular $k\times k$ submatrices, and a vector $\vec{w}\in \RR^k$. Letting $\vec{\xi}=(\xi_{1},\dots,\xi_{n})\in\{-1,1\}^n$ be a sequence of independent Rademacher random variables, we then have
\[\PP[M\vec{\xi}=\vec{w}]\le t^{-k/2}.\]
\end{theorem}

\begin{corollary}\label{cor:halasz}
Let $0\le d\le k\le n$ and $t\ge 1$ be integers. Consider vectors $\vec{a}_{1},\dots,\vec{a}_{n}\in\mathbb{R}^{k}$
such that one can form $t$ disjoint bases of $\RR^k$ from the vectors $\vec{a}_{1},\dots, \vec{a}_{n}$, and let $\mathcal{W}\su \RR^k$ be a $d$-dimensional affine-linear subspace.
Letting $\vec{\xi}=(\xi_{1},\dots,\xi_{n})\in\{-1,1\}^{n}$ be a sequence of independent Rademacher
random variables, we then have
\[
\Pr[\xi_{1}\vec{a}_{1}+\dots+\xi_{n}\vec{a}_{n}\in\mathcal{W}]\le t^{-(k-d)/2}.
\]
\end{corollary}

\begin{proof}
Write $\tilde{\mathcal{W}}$ for the $d$-dimensional linear subspace
parallel to the affine-linear subspace $\mathcal{W}$ (i.e., $\mathcal{W}=\tilde{\mathcal{W}}+\vec{v}$
for some $\vec{v}\in\RR^{k}$), and
consider a linear map $\phi:\RR^{k}\to\RR^{k-d}$ with kernel $\tilde{\mc W}$ (so $\phi(\mathcal{W})$ consists of a single point $\vec{p}\in \RR^{k-d}$). Writing $M\in\RR^{(k-d)\times n}$ for the matrix whose columns
are the vectors $\phi(\vec{a}_{1}),\dots,\phi(\vec{a}_{n})$, we have
\[
\Pr\left[\xi_{1}\vec{a}_{1}+\dots+\xi_{n}\vec{a}_{n}\in\mathcal{W}\right]=\Pr\left[\xi_{1}\phi(\vec{a}_{1})+\dots+\xi_{n}\phi(\vec{a}_{n})=\vec{p}\right]=\Pr[M\vec{\xi}=\vec{p}]\le t^{-(k-d)/2}.
\]
by \cref{thm-Halasz-FJZ} (note that for each of the $t$ disjoint bases formed among the vectors $\vec{a}_{1},\dots, \vec{a}_{n}$, the image under $\phi$ is a spanning set of $\RR^{k-d}$ and hence contains a basis of $\RR^{k-d}$, so we can find $t$ disjoint bases among the columns of $M$ and consequently $M$ contains $t$ disjoint nonsingular $(k-d)\times (k-d)$ submatrices).
\end{proof}

\begin{corollary}\label{cor:halasz-sub-version}
    Let $1\le k\le n$ be integers and $s >0$ be a real number. Consider a matrix $M\in \mathcal{H}^{k\times n}(s)$ and a vector $\vec{w}\in \RR^k$. Letting $\vec{\xi}=(\xi_{1},\dots,\xi_{n})\in\{-1,1\}^n$ be a sequence of independent Rademacher random variables, we then have
\[\PP[M\vec{\xi}=\vec{w}]\le (s/k)^{-k/2}.\]
\end{corollary}
\begin{proof}
    By \cref{lem-halasz-equivalent}(ii) $M$ contains $\lceil s/k\rceil$ disjoint nonsingular $k\times k$ submatrices, so we can apply \cref{thm-Halasz-FJZ}.
\end{proof}

Next, we recall some basic facts about linear forms. A \emph{linear form} $g\in \mathbb{R}[x_1,\dots,x_n]$ is a linear polynomial with constant term zero. That is, we can write $g(\vec{x})=\vec{v}\cdot \vec{x}=\vec{v}^{\transpose}\vec{x}=\vec{x}^{\transpose}\vec{v}$, where $\vec{v}\in\RR^n$ is the \emph{coefficient vector} of $g$. Note that for two linear forms $g_1,g_2\in \mathbb{R}[x_1,\dots,x_n]$ with coefficient vectors $\vec v_1,\vec v_2\in \RR^n$ we have
\begin{equation}\label{eq:product-linear-forms}
g_1(\vec{x})\cdot g_2(\vec{x})=\frac{1}{2}g_1(\vec{x})\cdot g_2(\vec{x})+\frac{1}{2}g_2(\vec{x})\cdot g_1(\vec{x})=\frac{1}{2}\vec{x}^{\transpose}\vec{v}_1\vec{v}_2^{\transpose}\vec{x}+\frac{1}{2}\vec{x}^{\transpose}\vec{v}_2\vec{v}_1^{\transpose}\vec{x}=\vec{x}^{\transpose}A\vec{x},
\end{equation}
where $A\in \RR^{n\times n}$ is the symmetric matrix given by $A=\frac{1}{2}(\vec{v}_1\vec{v}_2^{\transpose}+\vec{v}_2\vec{v}_1^{\transpose})$. More generally, every quadratic form (i.e., every quadratic polynomial with no linear and no constant term), can be written in the form $\vec{x}^{\transpose}A\vec{x}$ for a unique symmetric matrix $A\in \RR^{n\times n}$.

Another important ingredient for our proof is the following \emph{decoupling lemma}.

\begin{lemma}\label{lem:decoupling}
    If an event $\mc E(X,Y)$ depends on independent random objects $X,Y$, and $X'$ is an independent copy of $X$, then $\Pr[\mc E(X,Y)]\le \Pr[\mc E(X,Y)\text{ and }\mc E(X',Y)]^{1/2}$.
\end{lemma}

\cref{lem:decoupling} is a slight variant of a lemma of Costello, Tao and Vu~\cite[Lemma~4.7]{CTV06}, who popularised decoupling as a tool for polynomial anticoncentration. The particular statement of \cref{lem:decoupling} appears (for example) as \cite[Lemma~14]{Cos13}. For the reader's convenience, we include the proof (which is a simple application of the Cauchy--Schwarz inequality).

\begin{proof}[Proof of \cref{lem:decoupling}]
    By the Cauchy--Schwarz inequality, we have
    \begin{align*}
    \Pr[\mc E(X,Y)\text{ and }\mc E(X',Y)]&=\EE_Y\Bigl[\Pr[\mc E(X,Y)\text{ and }\mc E(X',Y)\,|\, Y]\Bigr]=\EE_Y\Bigl[\Pr[\mc E(X,Y)\,|\, Y]^2\Bigr]\\
    &\ge \EE_Y\Bigl[\Pr[\mc E(X,Y)\,|\, Y]\Bigr]^2=\Pr[\mc E(X,Y)]^2.
    \end{align*}
    Taking square roots on both sides gives the desired inequality.
\end{proof}

We will also need a simple lemma usually attributed to Odlyzko~\cite{Odl88}.

\begin{lemma}\label{lem:odlyzko}
Consider a matrix $M\in \RR^{k\times n}$ with $\rank M=k$ and a vector $\vec{w}\in \RR^k$. For a sequence of independent Rademacher random variables $\vec{\xi}=(\xi_{1},\dots,\xi_{n})\in\{-1,1\}^n$, we then have $\PP[M\vec{\xi}=\vec{w}]\le 2^{-k}$.
\end{lemma}

\begin{proof}
We can interpret $M\vec{\xi}=\vec{w}$ as a system of linear equations (in the variables $\xi_{1},\dots,\xi_{n}$). Bringing this system into row echelon form, it has $n-k$ free variables (which determine the values of the $k$ remaining variables). After exposing the $n-k$ entries of $\vec{\xi}$ corresponding to the free variables, there is at most one possibility for each of the remaining $k$ entries of $\vec{\xi}$ satisfying this system of equations. Thus, the probability of having $M\vec{\xi}=\vec{w}$ is at most $2^{-k}$.
\end{proof}

Finally, we will need a simple numerical inequality.
\begin{lemma}\label{lem:simple-inequality}
For real numbers $a,b,c\ge0$ with $a^{2}\le ab+c$, we have $a\le b+\sqrt{c}$.
\end{lemma}

\begin{proof}Note that for all $x,y\ge 0$ we have the inequality $\sqrt{x+y}\le\sqrt{x}+\sqrt{y}$. Now, by the quadratic formula, $a^{2}\le ab+c$ implies $a\le(b+\sqrt{b^{2}+4c})/2\le (b+\sqrt{b^2}+\sqrt{4c})/2=b+\sqrt{c}$.
\end{proof}

\section{Inductive decoupling from a geometric point of view}\label{sec:geometric}

Note that for a quadratic polynomial $Q$ and a sequence $\vec{\xi}=(\xi_{1},\dots,\xi_{n})\in\{-1,1\}^n$ of independent Rademacher random variables, the event $Q(\vec \xi)=0$ is precisely the event that $\vec \xi$ falls in the vanishing locus of $Q$. Moreover, if the quadratic part of $Q$ has rank less than $r$ (i.e., if $Q$ can be expressed as a linear combination of $r-1$ squares of linear forms, plus an additional linear form, plus a constant term), then it is possible to interpret the event $Q(\vec \xi)=0$ as the event that $\xi_1\vec a_1+\dots +\xi_n\vec a_n\in \mathcal Z$, where $\vec a_1,\dots,\vec a_n\in \RR^{r}$ are vectors in $r$-dimensional space, and $\mc Z$ is the vanishing locus of some $r$-variable quadratic polynomial (indeed, the entries of each $\vec a_i$ correspond to the coefficients of $\xi_i$ in each of the linear forms described above).

In this section, we obtain an essentially optimal bound (in \cref{thm:quadratic-geometric} below) on probabilities of the form $\Pr[\xi_1\vec a_1+\dots +\xi_n\vec a_n\in \mathcal Z]$, where $\vec a_1,\dots,\vec a_n\in \RR^{r}$ are vectors satisfying some robust nondegeneracy condition and $\mc Z$ is a quadric in $\RR^r$ (or more generally, a quadric inside some affine-linear subspace of $\RR^r$), under the assumption that the dimension $r$ does not grow with $n$. This may be viewed as a warm-up to the full proof of \cref{thm:main-shiny}: it is proved via the same inductive decoupling scheme, but has fewer technicalities.

Using the connection described in the first paragraph above (with an additional ``dropping to a subspace'' argument, of the type we will later see in \cref{sec:low-rank}), one can use \cref{thm:quadratic-geometric} to prove \cref{thm:main-shiny} in the special case where the quadratic part of $Q$ has bounded rank. Actually, \cref{thm:quadratic-geometric} is also an ingredient in the full proof of our main theorem (\cref{thm:main-shiny}).
We also remark that the result of this section fits nicely in the ``geometric Littlewood--Offord'' framework of Fox, Kwan and Spink~\cite{FKS23}. In particular, the case of \cref{thm:quadratic-geometric} where $\mathcal Z\subseteq \RR^4$ is a sphere in four dimensions was explicitly raised as the simplest open case of the main problem in \cite{FKS23}.

Before stating the main result of this section, we record some relevant terminology.

\newcommand{\zl}{\on{V}}

\begin{definition}
As usual, for a polynomial $P\in \RR[x_1,\dots,x_d]$ we write $\zl(P)=\{\vec{x}\in \RR^d: P(\vec x)=0\}$ for the vanishing locus of $P$. A \emph{quadric} $\mathcal{Z}\subsetneq\RR^{d}$ is the vanishing locus $\mathcal{Z}=\zl(P)$ of some nonzero quadratic polynomial $P\in\RR[x_{1},\dots,x_{d}]$. We say that a quadric $\mathcal{Z}\subsetneq\RR^{d}$ is \emph{irreducible} if $\mathcal{Z}=\zl(P)$ for an irreducible quadratic polynomial $P\in \RR[x_1,\dots,x_d]$.

For a $d$-dimensional affine-linear subspace $\mathcal{W}\su \RR^{r}$,
we say that $\mathcal{Z}\subsetneq\mathcal{W}$ is a \emph{quadric
on $\mathcal{W}$} if it is the image of some quadric $\zl(P)\subsetneq\RR^{d}$
under an affine-linear isomorphism $\phi:\RR^{d}\to\mathcal{W}$. Equivalently, $\mathcal{Z}\subsetneq\mathcal{W}$ is a quadric on $\mathcal{W}$ if and only if $\mathcal{Z}=\mathcal{W}\cap \zl(P)$ for some quadratic polynomial $P\in \RR[x_1,\dots,x_r]$ with $\mathcal{W}\not\subseteq \zl(P)$.  We say that a quadric $\mathcal{Z}\subsetneq\mathcal{W}$ is \emph{irreducible} if it is the image of some irreducible quadric $\zl(P)\subsetneq\RR^{d}$
under an affine-linear isomorphism $\phi:\RR^{d}\to\mathcal{W}$.
\end{definition}

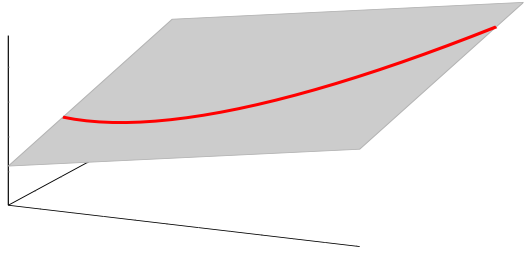
\begin{figure}[t]
\begin{center}
\begin{tikzpicture}[yscale=0.7]
\begin{axis}[mesh/ordering=y varies,colormap/blackwhite,axis lines*=middle,xtick distance=0.09,ytick distance=0.16,ztick distance=0.6,xticklabels=\empty,yticklabels=\empty,zticklabels=\empty,tickwidth=0]
    \addplot3 [surf,fill=black!20,faceted color=black!30] coordinates {
        (0,0,0.7) (0,1,1.7)   (1,0,1.7)
 
        (1,1,2.7)
    };
	\addplot3 [red,ultra thick,domain=0:1,samples y=1] (x,x^2-x/2+1/3,x^2+x/2+1/3+0.7);
	\addplot3 [domain=0:1] (0,0,0.27);
\end{axis}
\end{tikzpicture} 
\end{center}\caption{\label{fig:quadric}An example of a quadric on a 2-dimensional affine plane in $\mb R^3$.}
\end{figure}

As an example, see \cref{fig:quadric} showing a quadric on a 2-dimensional affine plane in $\RR^{3}$. Now we are ready to state the main result of this section.

\begin{theorem}
\label{thm:quadratic-geometric}Let $0\le d<r$ be integers. Let $\mathcal{Z}\subsetneq\mathcal{W}$
be a quadric on a $(d+1)$-dimensional affine-linear
subspace $\mathcal{W}\su\RR^{r}$. Consider vectors $\vec{a}_{1},\dots,\vec{a}_{n}\in\mathbb{R}^{r}$ such that one can form $t$ disjoint bases of $\mathbb{R}^{r}$ from the vectors $\vec{a}_{1},\dots, \vec{a}_n$, where $t$ is a positive integer divisible by $2^d$. Let $(\xi_1,\dots,\xi_n)\in\{-1,1\}^{n}$ be a sequence of independent Rademacher random variables. Then 
\[
\Pr\left[\xi_{1}\vec{a}_{1}+\dots+\xi_{n}\vec{a}_{n}\in\mathcal{Z}\right]\le\frac{2^{dr+1}}{t^{(r-d)/2}}.
\]
\end{theorem}

If $\mathcal{Z}\subsetneq\mathcal{W}$
is a quadric on a $(d+1)$-dimensional affine-linear
subspace $\mathcal{W}$, then $\mathcal{Z}$ has dimension at most $d$ (we do not formally define what ``dimension'' means in this context, as this is not needed for our arguments, but appeal to the reader's intuition). Note that the form of the bound in the above theorem (with $t^{(r-d)/2}$ in the denominator) is the same as in \cref{cor:halasz} for $d$-dimensional affine-linear subspaces.

For the proof of \cref{thm:quadratic-geometric}, we will rely on the following
algebraic fact. 
\begin{lemma}
\label{lem:geometry-lemma}Let $d\ge 2$, and let $\mathcal{Z}\subsetneq\mathbb{R}^{d}$ be an irreducible quadric. Then, at least one of the following holds: 
\begin{compactitem}
\item[(i)] There is a direction $\vec{v}\in \RR^d\setminus\{\vec 0\}$ such that $\mathcal{Z}+\mathbb{R}\vec{v}=\mathcal{Z}$
(i.e., $\mathcal{Z}$ is invariant under translation along the direction
$\vec{v}$), or 
\item [(ii)] for any vectors $\vec{x},\vec{y}\in \RR^d$ with $\vec{x}\ne\vec{y}$, the intersection $(\mathcal{Z}-\vec{x})\cap(\mathcal{Z}-\vec{y})$
is a quadric on a $(d-1)$-dimensional affine-linear subspace $\mathcal{W}_{\vec{x},\vec{y}}\subsetneq\mathbb{R}^{d}$.
\end{compactitem}
\end{lemma}

\begin{proof}
Let $\mathcal{Z}=\zl(P)$ for a nonzero quadratic polynomial $P\in \RR[x_1,\dots,x_d]$, and suppose that (i) does not hold. Note that then $\mathcal{Z}-\vec x=\{\vec{w}\in \RR^d: P(\vec w+\vec x)=0\}$ for any $\vec x\in \RR^d$.  For any $\vec{x},\vec{y}\in \RR^d$ with $\vec{x}\ne\vec{y}$, consider the linear polynomial $L_{\vec{x},\vec{y}}:\vec w\mapsto P(\vec{w}+\vec{y})-P(\vec{w}+\vec{x})$ and write $\mathcal{W}_{\vec{x},\vec{y}}=\zl(L_{\vec{x},\vec{y}})$, so $(\mathcal{Z}-\vec{x})\cap(\mathcal{Z}-\vec{y})=(\mathcal{Z}-\vec{x})\cap\mathcal{W}_{\vec{x},\vec{y}}$. This already shows that $(\mathcal{Z}-\vec{x})\cap(\mathcal{Z}-\vec{y})$ is the vanishing locus of a quadratic polynomial on $\mathcal{W}_{\vec{x},\vec{y}}$. To verify (ii), we will check that $L_{\vec{x},\vec{y}}$ is not the zero polynomial (which shows that $\dim \mathcal{W}_{\vec{x},\vec{y}}=d-1$), and that $\mathcal{W}_{\vec{x},\vec{y}}\not \su \mathcal{Z}-\vec x$ (which implies $(\mathcal{Z}-\vec{x})\cap(\mathcal{Z}-\vec{y})=(\mathcal{Z}-\vec{x})\cap\mathcal{W}_{\vec{x},\vec{y}}\subsetneq \mathcal{W}_{\vec{x},\vec{y}}$).

First, the reason $L_{\vec{x},\vec{y}}$ cannot be the zero polynomial
is that (i) does not hold. Indeed, if $L_{\vec{x},\vec{y}}$ were
the zero polynomial, then for all $\vec{w}\in\RR^{d}$ and $\lambda\in\ZZ$
we would have $P(\vec{w})=P(\vec{w}+\lambda(\vec{x}-\vec{y}))$.
That is to say, for all $\vec{w}\in\RR^{d}$ the quadratic polynomial
$\lambda\mapsto P(\vec{w})-P(\vec{w}+\lambda(\vec{x}-\vec{y}))$ would
have infinitely many zeroes, so would be the zero polynomial, meaning
that $P(\vec{w})=P(\vec{w}+\lambda(\vec{x}-\vec{y}))$ for all $\lambda\in\RR$ and $\vec{w}\in\RR^{d}$. Hence we would have $\mathcal{Z}+\lambda(\vec{x}-\vec{y})=\mathcal{Z}$ for all $\lambda\in\RR$, so (i) would hold for $\vec{v}=\vec{x}-\vec{y}$.

Finally, it remains to show $\mathcal{W}_{\vec{x},\vec{y}}\not \su \mathcal{Z}-\vec x$. Indeed, otherwise $\mathcal{W}_{\vec{x},\vec{y}}$ would be an irreducible component of $\mathcal{Z}-\vec x$, but by our assumptions $\mathcal{Z}-\vec x$  is irreducible (since $\mathcal{Z}$ is). We also have $\mathcal{Z}-\vec x\ne \mathcal{W}_{\vec{x},\vec{y}}$, since $\mathcal{Z}-\vec x$ not invariant under translation along the direction of any nonzero vector (since $\mathcal{Z}$ does not satisfy (i)). So we indeed have $\mathcal{W}_{\vec{x},\vec{y}}\not \su \mathcal{Z}-\vec x$.
\end{proof}

Now we prove \cref{thm:quadratic-geometric}.

\begin{proof}[Proof of \cref{thm:quadratic-geometric}]
We proceed by induction on $d$. In the base case $d=0$, our quadric
$\mathcal{Z}$ consists of at most two points, so the theorem statement holds by \cref{cor:halasz} (taking  $\mathcal{W}$ in \cref{cor:halasz} to be $0$-dimensional, i.e., a single point). Assume now
that $d\ge1$ (and $r\ge d+1$), and that the theorem statement holds
for smaller values of $d$. Let $\vec{X}=\xi_{1}\vec{a}_{1}+\dots+\xi_{n}\vec{a}_{n}$. 

\textbf{Step 1: The reducible case.} First, it is easy to handle the case where the quadric $\mathcal{Z}\subsetneq \mathcal{W}$ is reducible, i.e., where the irreducible components of $\mathcal{Z}$ are two affine-linear subspaces of dimension $d$. Indeed, suppose
that $\mathcal{Z}=\mathcal{V}_{1}\cup\mathcal{V}_{2}$ for two $d$-dimensional
affine-linear subspaces $\mathcal{V}_{1},\mathcal{V}_{2}\subsetneq\mathcal{W}$.
By \cref{cor:halasz} (i.e., by the version of Hal\'asz' theorem in \cite{FJZ22}), we have
\[
\Pr[\vec{X}\in\mathcal{Z}]\le\Pr[\vec{X}\in\mathcal{V}_{1}]+\Pr[\vec{X}\in\mathcal{V}_{2}]\le \frac{2}{t^{(r-d)/2}},
\]
which implies the desired result.

\textbf{Step 2: The translation-invariant case.} It is also easy to handle the case where there is a direction $\vec{v}\ne\vec 0$
such that $\mathcal{Z}+\mathbb{R}\vec{v}=\mathcal{Z}$, because then
we can project our entire problem along the direction of $\vec{v}$ to obtain a lower-dimensional problem (noting that then we also have $\mathcal{W}+\mathbb{R}\vec{v}=\mathcal{W}$). Indeed,
consider a linear map $\phi:\RR^r\to \RR^{r-1}$ with kernel $\on{span}(\vec v)$, and observe that  $\phi(\mathcal{Z})$
is a quadric on the $(d-1)$-dimensional affine-linear subspace $\phi(\mathcal{W})$. 
Then,
we have 
\[
\Pr[\vec{X}\in\mathcal{Z}]=\Pr[\phi(\vec{X})\in\phi(\mathcal{Z})]\le\frac{2^{(d-1)(r-1)+1}}{t^{(r-d)/2}},
\]
by our induction hypothesis (noting that among $\phi(\vec{a}_1),\dots, \phi(\vec{a}_n)$ one can still form $t$ disjoint bases), and the desired result follows.

\textbf{Step 3: Decoupling.} Now, we can assume that $\mathcal{Z}\subseteq \mathcal{W}$ is an irreducible quadric and that there is
no direction $\vec{v}\ne\vec 0$ such that $\mathcal{Z}+\mathbb{R}\vec{v}=\mathcal{Z}$.
Let $\tilde{\mathcal{W}}\su \RR^r$
be  the $(d+1)$-dimensional linear subspace parallel to $\mathcal{W}$ (i.e., $\mathcal{W}=\tilde{\mathcal{W}}+\vec{w}$
for some $\vec{w}\in\RR^{r}$). For any $\vec{x}\ne\vec{y}$ with $\vec x-\vec y\in \tilde{\mc W}$, by \cref{lem:geometry-lemma} (with an affine-linear isomorphism $\RR^{d}\to \mc W-\vec x$),
 the intersection $(\mathcal{Z}-\vec{x})\cap(\mathcal{Z}-\vec{y})$
is a quadric on a $d$-dimensional affine-linear subspace $\mathcal{W}_{\vec{x},\vec{y}}\subsetneq \tilde{\mathcal{W}}+\vec w_{\vec{x},\vec{y}}$ for some $\vec w_{\vec{x},\vec{y}}\in \RR^r$.
The next step is to split our random variable into two parts and use
the decoupling lemma (\cref{lem:decoupling}) to relate our probability $\Pr[\vec{X}\in\mathcal{Z}]$
to probabilities that certain random variables lie in quadrics
of the form $(\mathcal{Z}-\vec{x})\cap(\mathcal{Z}-\vec{y})$ (these probabilities can then be bounded via the induction hypothesis).

Let $[n]=I\cup J$ be a partition of the index set $[n]$ into two subsets $I,J$ such that one can form $t/2$ disjoint bases from the vectors $\vec a_i$ for $i\in I$, and one can also form $t/2$ disjoint bases from the vectors $\vec a_j$ for $j\in J$. Let $\vec{X}_{I}=\sum_{i\in I}\xi_{i}\vec{a}_{i}$
and let $\vec{X}_{J}=\sum_{j\in J}\xi_{j}\vec{a}_{j}$
(so $\vec{X}=\vec{X}_I+\vec{X}_J$) and let $\vec{X}_I'$ be
an independent copy of the random variable $\vec{X}_I$. By decoupling (\cref{lem:decoupling}) we have
\begin{align}
\Pr[\vec{X}\in\mathcal{Z}]^{2} & \le\Pr[\vec{X}_I+\vec{X}_J\in\mathcal{Z}\text{ and }\vec{X}_I'+\vec{X}_J\in\mathcal{Z}]\nonumber \\
 & =\Pr[\vec{X}_J\in(\mathcal{Z}-\vec{X}_I)\cap(\mathcal{Z}-\vec{X}_I')\text{ and }\vec{X}_I\ne\vec{X}_I']+\Pr[\vec{X}_I+\vec{X}_J\in\mathcal{Z}\text{ and }\vec{X}_I'=\vec{X}_I].\label{eq:geom-decoupling}
\end{align}

\textbf{Step 4: Dealing with degenerate intersection.} We now study the second term in \cref{eq:geom-decoupling} (which can be
thought of as a lower-order ``error term'' corresponding to the possibility
that decoupling does not actually reduce the dimension of our problem).

For any outcome of $\vec{X}_I$, we have $\Pr[\vec{X}_I'=\vec{X}_I\,|\,\vec{X}_I]\le(t/2)^{-r/2}$
by Hal\'asz' theorem (\cref{cor:halasz}, taking  $\mathcal{W}$ to be a single point). So, 
\begin{equation}
\Pr[\vec{X}_I+\vec{X}_J\in\mathcal{Z}\text{ and }\vec{X}_I'=\vec{X}_I]\le \Pr[\vec{X}_I+\vec{X}_J\in\mathcal{Z}]\cdot (t/2)^{-r/2}=\Pr[\vec{X}\in\mathcal{Z}]\cdot (t/2)^{-r/2}.\label{eq:geom-error-term}
\end{equation}

\textbf{Step 5: Inductively bounding the main term.} Now we deal with the first term in \cref{eq:geom-decoupling}. Recalling that $\tilde{\mathcal{W}}\su \RR^r$ is the $(d+1)$-dimensional linear subspace parallel to $\mathcal{W}$, note that it is impossible to have
$\vec{X}_J\in(\mathcal{Z}-\vec{X}_I)\cap(\mathcal{Z}-\vec{X}_I')$
if $\vec{X}_I-\vec{X}_I'\not\in\tilde{\mathcal{W}}$ (since then $(\mathcal{Z}-\vec{X}_I)\cap(\mathcal{Z}-\vec{X}_I')\su (\mathcal{W}-\vec{X}_I)\cap(\mathcal{W}-\vec{X}_I')=\emptyset$). So, we
combine our induction hypothesis with a bound on the probability that
$\vec{X}_I-\vec{X}_I'\in\tilde{\mathcal{W}}$.

For all outcomes of $\vec{X}_I'$, by Hal\'asz'
theorem (\cref{cor:halasz}) we have
\[\Pr[\vec{X}_I-\vec{X}_I'\in\tilde{\mathcal{W}}\,|\,\vec{X}_I']=\Pr[\vec{X}_I\in\tilde{\mathcal{W}}+\vec{X}_I'\,|\,\vec{X}_I']\le(t/2)^{-(r-d-1)/2}.
\]
Also, for any outcomes of $\vec{X}_I,\vec{X}_I'$ such that $\vec{X}_I\ne \vec{X}_I'$ and $\vec{X}_I-\vec{X}_I'\in\tilde{\mathcal{W}}$, by the discussion in Step 3, the intersection $(\mathcal{Z}-\vec{X}_I)\cap(\mathcal{Z}-\vec{X}_I')$ is a quadric on a $d$-dimensional affine-linear subspace $\mathcal{W}_{\vec{X}_I,\vec{X}_I'}\su \RR^r$. So by our induction hypothesis we have
\[\Pr[\vec{X}_J\in(\mathcal{Z}-\vec{X}_I)\cap(\mathcal{Z}-\vec{X}_I')\,|\,\vec{X}_I,\vec{X}_I']\le\frac{2^{(d-1)r+1}}{(t/2)^{(r-d+1)/2}}.\]
Thus, we obtain
\begin{equation}\label{eq:geom-main-term}
    \Pr[\vec{X}_J\in(\mathcal{Z}-\vec{X}_I)\cap(\mathcal{Z}-\vec{X}_I')\text{ and }\vec{X}_I\ne\vec{X}_I']\le (t/2)^{-(r-d-1)/2}\cdot\frac{2^{(d-1)r+1}}{(t/2)^{(r-d+1)/2}}=\frac{2^{(d-1)r+1}}{(t/2)^{r-d}}\le\frac{2^{dr+1}}{t^{r-d}}.
\end{equation}

\textbf{Step 6: Concluding.} We can now deduce from (\ref{eq:geom-decoupling}), (\ref{eq:geom-error-term}) and (\ref{eq:geom-main-term})
that
\[\Pr[X\in\mathcal{Z}]^{2} \le\Pr[X\in\mathcal{Z}]\cdot(t/2)^{-r/2}+\frac{2^{dr+1}}{t^{r-d}}=\Pr[X\in\mathcal{Z}]\cdot \frac{2^{r/2}}{t^{r/2}}+\frac{2^{dr+1}}{t^{r-d}}.\]
So, by \cref{lem:simple-inequality} we have (also using that $d\ge 1$ and $r\ge d+1\ge 2$)
\[
\Pr[X\in\mathcal{Z}]\le \frac{2^{r/2}}{t^{r/2}}+\frac{2^{(dr+1)/2}}{t^{(r-d)/2}}\le\frac{2^{dr+1}}{t^{(r-d)/2}},
\]
as desired.
\end{proof}

\section{Proof strategy for the general case}
Now we turn to the general (not necessarily low-rank) case of \cref{thm:main-shiny}. Actually, we consider the following variation on \cref{thm:main-shiny}, with a slightly more technical assumption on $Q$ (namely, we need to assume that the \emph{quadratic part} of $Q$ ``robustly depends on many different variables''). This assumption is very similar to assumptions in some previous Littlewood--Offord-type theorems~\cite{MNV16,RV13}.

\begin{theorem}\label{thm:main}
Let $Q\in\RR[x_{1},\dots,x_{n}]$ be a multivariate quadratic polynomial with quadratic part $\vec x^\transpose A\vec x$ for some symmetric matrix $A\in \RR^{n\times n}$. Let $s\ge 1$ be an integer, and assume that for every subset $S\subseteq[n]$ with $|S|\ge n-s$, the submatrix $A[S\times S]$ has at least one nonzero entry outside its diagonal. Then for a sequence $\vec{\xi}\in\{-1,1\}^{n}$ of independent Rademacher random variables, we have
\[
\Pr[Q(\vec{\xi})=0]\le\frac{C'}{\sqrt{s}},
\]
for some absolute constant $C'$.
\end{theorem}

The reason that the assumption in \cref{thm:main} only pays attention to the nondiagonal entries of $A$ is that the diagonal entries correspond to square terms of the form $x_i^2$. If $x_i\in \{-1,1\}$ then $x_i^2$ is always equal to 1, so such square terms can be treated as constants (and therefore they ``do not really contribute'' to the quadratic part of $Q$).

In \cref{sec:arbitrary-distributions} we will show how to deduce \cref{thm:main-shiny} (and \cref{thm:general-distributions} for general distributions) from \cref{thm:main}. For most of the rest of the paper, we will focus on proving \cref{thm:main}.

At a high level, the strategy to prove \cref{thm:main}
is similar to the proof of \cref{thm:quadratic-geometric}: in order to estimate the probability $\Pr[Q(\vec{\xi})=0]$ for a given quadratic polynomial $Q\in \RR[x_1,\dots,x_n]$, we inductively estimate probabilities of the form $\Pr[Q(\vec{\xi})=0\text{ and }M\vec{\xi}=\vec{w}]$ for a given matrix $M\in \RR^{k\times n}$ and a given vector $\vec{w}\in \RR^k$ (i.e., we estimate the probability that $\vec{\xi}$ lies in a given quadric on a given affine-linear subspace). However, there are additional difficulties in comparison to the proof of \cref{thm:quadratic-geometric} in the previous section.

Most importantly, recall that in \cref{thm:quadratic-geometric} we worked with a random vector $\xi_1\vec a_1+\dots \xi_n\vec a_n$ which has ``a lot of anticoncentration in each direction'' (we assumed that there are many disjoint bases among the vectors $\vec a_1,\dots,\vec a_n$). This was only possible since the dimension of our space was much less than $n$: to be precise, that proof approach can only obtain bounds of the form $O(1/\sqrt n)$ when we have $\Omega(n)$ disjoint bases among the vectors $\vec a_1,\dots,\vec a_n$, which is only possible when we are working in $O(1)$-dimensional space.

In the proof of \cref{thm:main} we work directly with the random vector $\vec \xi\in \{-1,1\}^n$ in $n$-dimensional space. We cannot ensure good anticoncentration in all directions, so we need to restrict the affine-linear subspaces we can consider. Specifically, we prove bounds on $\Pr[Q(\vec{\xi})=0\text{ and }M\vec{\xi}=\vec{w}]$ only when the matrix $M$ satisfies a Hal\'asz-type robust rank condition (as in \cref{def:Halasz}). This means that we now need to maintain this robust rank condition as $M$ changes over the course of the induction.

Also, we encounter much more delicate nondegeneracy issues than in the proof of \cref{thm:quadratic-geometric}. In particular, even if $Q$ satisfies the nondegeneracy condition in \cref{thm:main} (robustly depending on many variables), it may become very degenerate when restricted to a subspace of the form $\{\vec x :M\vec x=\vec w\}$ for a matrix $M\in \RR^{k\times n}$ and a vector $\vec{w}\in \RR^k$. 
For example, consider the case where $n$ is divisible by 2 and $Q\in \mb R[x_1,\dots,x_n]$, $M\in \mb R^{1\times n}$ and $\vec w\in\mb R^1$ are defined by
\begin{equation}Q(\vec x)=(x_1+\dots+x_{n/2})(x_{n/2+1}+\dots+x_{n}),\quad\quad M=(1,\dots,1,0,\dots,0),\quad\quad \vec w=0\label{eq:example-perturbation}\end{equation}
(where the first $n/2$ entries of $M\in \mb R^{1\times n}$ are ``$1$'' and the last $n/2$ entries are ``$0$''). In this case $Q$ is always zero on the subspace $\{\vec x\in \mb R^n:M\vec x=\vec w\}$. So, in order to be able to obtain a sensible bound, we need to ensure that $Q$ satisfies a nondegeneracy condition \emph{with respect to $M$}.

In the rest of this section we describe the strategy of the proof of \cref{thm:main} in a bit more detail, stating several key lemmas and definitions along the way. First, we elaborate on the ``nondegeneracy with respect to $M$'' condition mentioned above. To formulate this condition (as well as other similar conditions appearing later in the proof), we define the notion of a \emph{$(T,U)$-perturbation}, as follows.

\begin{definition}\label{def:perturbation}
For matrices $A\in\RR^{n\times m}$ and $T\in \RR^{k\times m}$ and $U\in\RR^{k\times n}$, a \emph{$(T,U)$-perturbation of $A$} is a matrix $A'\in \RR^{n\times m}$ of the form $A'=A+LT+U^{\transpose}R$ for some matrices $L\in\RR^{n\times k}$
and $R\in\RR^{k\times m}$, i.e., some matrix that can be obtained from $A$ by adding linear combinations of rows of $T$ to its rows, and adding linear combinations of rows of $U$ to its columns.
\end{definition}

As the degenerate $k=0$ case of the above definition, note that if $T\in \mb R^{0\times m}$ and $U\in \mb R^{0\times n}$ are ``empty matrices'', then $A'\in \RR^{n\times m}$ is a $(T,U)$-perturbation of $A$ if and only if $A'=A$.

For $Q,M,\vec w$ as defined in the example in \cref{eq:example-perturbation}, we can write $Q(\vec x)=\vec x^\transpose A \vec x$ for a symmetric matrix $A\in \RR^{n\times n}$ (with a block structure where the bottom left block and the top right block, each of size $(n/2)\times (n/2)$, have all entries being $1/2$, and all entries outside these blocks are $0$). It turns out that this matrix $A$ is an $(M,M)$-perturbation of the zero matrix in $\RR^{n\times n}$. Roughly speaking, this is the reason for the degenerate behaviour of $Q$ on the subspace $\{\vec x :M\vec x=\vec w\}$. In general, the notion of an $(M,M)$-perturbation gives a condition under which two $n\times n$ matrices give rise to quadratic polynomials which are ``essentially the same'' on an affine-linear subspace of the form  $\{\vec x :M\vec x=\vec w\}$, as follows.

\begin{lemma}\label{lem:perturbation}
    Fix a matrix $M\in \RR^{k\times n}$, a vector $\vec w\in \RR^k$, and a quadratic polynomial $Q\in \RR[x_1,\dots,x_n]$ with quadratic part $\vec x^\transpose A\vec x$ (where $A\in \RR^{n\times n}$).
    
    Consider an $(M,M)$-perturbation $A'$ of $A$, and let $A^*\in \RR^{n\times n}$ be a matrix which agrees with $A'$ on its off-diagonal entries. Then, there is a quadratic polynomial $Q^*\in [x_1,\dots,x_n]$ with quadratic part $\vec x^\transpose A^*\vec x$ such that $Q^*(\vec{\xi})=Q(\vec{\xi})$ for all $\vec{\xi}\in \{-1,1\}^n$ with $M\vec{\xi}=\vec{w}$. In particular, for a sequence $\vec{\xi}\in\{-1,1\}^{n}$ of independent Rademacher random variables, we have
    \[\Pr[Q(\vec{\xi})=0\text{ and }M\vec{\xi}=\vec{w}]=\Pr[Q^*(\vec{\xi})=0\text{ and }M\vec{\xi}=\vec{w}].\]
\end{lemma}

\begin{proof}[Proof]
Let $A^*=A'+D=A+LM+M^{\transpose}R+D$ for some matrices $L\in\RR^{n\times k}$ and $R\in\RR^{k\times n}$ and a diagonal matrix $D\in \RR^{n\times n}$, and let $\lambda_1,\dots,\lambda_n\in \RR$ be the diagonal entries of $D$. Then
\[\vec x^\transpose A^*\vec x=\vec x^\transpose A\vec x+\vec x^\transpose LM\vec x+\vec x^\transpose M^{\transpose}R\vec x+\vec x^\transpose D\vec x=\vec x^\transpose A\vec x+(M\vec x)^\transpose L^\transpose \vec x+(M\vec x)^\transpose R \vec x+(\lambda_1x_1^2+\dots+\lambda_nx_n^2).\]
Thus, for every $\vec{\xi}\in \{-1,1\}^n$ with $M\vec{\xi}=\vec{w}$, we have
\[\vec \xi^\transpose A^*\vec \xi=\vec \xi^\transpose A\vec \xi+\vec w^\transpose L^\transpose \vec \xi+\vec w^\transpose R \vec x+(\lambda_1\xi_1^2+\dots+\lambda_n\xi_n^2)=\vec \xi^\transpose A\vec \xi+\vec w^\transpose (L^\transpose+R) \vec \xi+(\lambda_1+\dots+\lambda_n).\]
Hence, we can define the desired quadratic polynomial $Q^*\in [x_1,\dots,x_n]$ with quadratic part $\vec x^\transpose A^*\vec x$ by
\[Q^*(\vec x)=Q(\vec x)+\vec x^\transpose A^*\vec x-\vec x^\transpose A\vec x-\vec w^\transpose (L^\transpose+R) \vec x-(\lambda_1+\dots+\lambda_n),\]
and we indeed have $Q^*(\vec{\xi})=Q(\vec{\xi})$ for all $\vec{\xi}\in \{-1,1\}^n$ with $M\vec{\xi}=\vec{w}$.
\end{proof}

Now, recall that our strategy to prove \cref{thm:main} is to inductively upper-bound probabilities of the form $\Pr[Q(\vec{\xi})=0\text{ and }M\vec{\xi}=\vec w]$, assuming that $M$ satisfies a robust rank condition, and assuming a nondegeneracy condition on $Q$ with respect to $M$. It will be convenient to introduce some notation for the maximum possible such probability, as follows.

\begin{definition}\label{def:function-f}
For an integer $k\ge 0$ and real number $s\ge 0$, let us define 
\[f(k,s)=\sup_{(n,Q,M,\vec w)} \Pr[Q(\vec{\xi})=0\text{ and }M\vec{\xi}=\vec w],\]
where the supremum is taken over all quadruples $(n,Q,M,\vec w)$, where $n$ is a positive integer, $Q\in\RR[x_{1},\dots,x_{n}]$ is a quadratic polynomial, $M\in\mathcal{H}^{k\times n}(s)$ and $\vec w\in \RR^k$, such that the following condition holds:

\begin{compactitem}
    \item[($*$)] If we write the quadratic part of $Q(\vec{x})$ as $\vec{x}^{\transpose}A\vec{x}$ for a symmetric matrix $A\in \RR^{n\times n}$, then for every subset $S\subseteq[n]$ with $|S|\ge n-s$, and every $(M,M)$-perturbation $A'$ of $A$, the submatrix $A'[S\times S]$ has at least one nonzero entry outside the diagonal.
\end{compactitem}

For each such quadruple $(n,Q,M,\vec w)$, the probability above is taken with respect to a sequence of independent Rademacher random variables $\vec{\xi}\in\{-1,1\}^{n}$.
\end{definition}

Note that we always have $f(k,s)\le 1$, since $f(k,s)$ is defined as a supremum of certain probabilities (which are all at most $1$). Also note that for $0\le s'\le s$, we always have $f(k,s)\le f(k,s')$ (since the supremum in the definition of $f(k,s')$ is taken over a wider range of quadruples $(n,Q,M,\vec w)$ than for $f(k,s)$).

Note that for a polynomial $Q\in \RR[x_1,\dots,x_n]$ and an integer $s\ge 1$ as in \cref{thm:main}, condition ($*$) is satisfied for $k=0$, the empty matrix $M\in \RR^{0\times n}$ and the empty vector $\vec{w}\in \RR^0$. Indeed, writing the quadratic part of $Q$ as $\vec{x}^{\transpose}A\vec{x}$ for a symmetric matrix $A\in \RR^{n\times n}$, the only $(M,M)$ perturbation $A'$ of $A$ is $A'=A$, and by the assumption in \cref{thm:main} the matrix $A'[S\times S]=A[S\times S]$ has at least one nonzero entry outside the diagonal for every subset $S\subseteq[n]$ with $|S|\ge n-s$. 
Therefore we have (noting that the condition $M\vec \xi=\vec w\in \RR^0$ is vacuous)
\[\Pr[Q(\vec \xi)=0]=\Pr[Q(\vec{\xi})=0\text{ and }M\vec{\xi}=\vec w]\le f(0,s).\]
Thus, proving \cref{thm:main} amounts to showing that $f(0,s)\le C'/\sqrt{s}$ for some absolute constant $C'$.

Now, the following recursive upper bound on $f(k,s)$ is the main ingredient in our proof of \cref{thm:main}.

\begin{theorem}\label{thm:recursion}
    For any integer $k\ge 0$ and any real number $s>0$, we have
    \[f(k,s)\le \max\left\{s_*^{-(k+1)/2}, \quad s_*^{-(k+2)/2}+s_*^{-k/4}\cdot f(k+1,s_*)^{1/2}\right\},\]
    where $s_*=s/(k+2)^{500}$.
\end{theorem}

Using this recursive bound, we can obtain an upper bound for $f(k,s)$ inductively (the formula for this upper bound is rather complicated, so we postpone this calculation to \cref{sec:calculation-deduction}). With another straightforward (though somewhat technical) calculation, one can then show that in the $k=0$ case this upper bound implies \cref{thm:main}. This latter calculation can also be found in \cref{sec:calculation-deduction}.

To prove \cref{thm:recursion}, we need to upper-bound $\Pr[Q(\vec{\xi})=0\text{ and }M\vec{\xi}=\vec w]$ for
$Q,M,\vec w$
satisfying the conditions in \cref{def:function-f}, in terms of a probability of the same form but with slightly different parameters (most importantly, with ``$k+1$'' in place of ``$k$''). To do so, we use our inductive decoupling scheme outlined in \cref{subsec:decoupling-scheme}: 
we consider a partition $[n]=I\cup J$ of the index set into two parts $I$ and $J$, and, using the decoupling inequality in \cref{lem:decoupling}, we will obtain an upper bound on $\Pr[Q(\vec{\xi})=0\text{ and }M\vec{\xi}=\vec w]$ involving conditional probabilities of the form
\begin{equation}\Pr\Bigl[Q_{\vec \xi[I]}(\vec{\xi}[J])=0\text{ and }M_{\vec \xi[I],\pvec{\xi}[I]}\vec{\xi}[J]=\vec w_{\xi[I],\pvec\xi[I]}\,\Big|\,\xi[I],\pvec{\xi}[I]\Bigr],\label{eq:conditioning-preview}\end{equation}
where $M_{\vec \xi[I],\pvec{\xi}[I]}\in \RR^{(k+1)\times J}$ is a $(k+1)\times |J|$ matrix depending on $\vec \xi[I]$ and $\pvec{\xi}[I]$, and $Q_{\vec \xi[I]}\in \RR[x_j: j\in J]$ is a quadratic polynomial whose coefficients depend on $\vec \xi[I]$, and $\vec w_{\xi[I],\pvec\xi[I]}\in \RR^{k+1}$ is a vector depending on $\vec \xi[I]$ and $\pvec{\xi}[I]$.

In our proof of \cref{thm:recursion}, we wish to upper-bound conditional probabilities of the form in \cref{eq:conditioning-preview} by $f(k+1,s_*)$.
Recalling \cref{def:function-f}, this requires that $M_{\vec \xi[I],\pvec{\xi}[I]}\in \mathcal H^{(k+1)\times n}(s_*)$ (i.e, that $M_{\vec \xi[I],\pvec{\xi}[I]}$ robustly has rank at least $k+1$). So, as previously outlined, an important ingredient in the proof is a ``robust rank inheritance'' lemma, which implies that this robust rank condition for $M_{\vec \xi[I],\pvec{\xi}[I]}$ holds with sufficiently high probability.

Now, the way in which $M_{\vec \xi[I],\pvec{\xi}[I]}$ is derived from $M$ depends on $Q$: specifically, one can check that $M_{\vec \xi[I],\pvec{\xi}[I]}$ is obtained from $M[[k]\times J]$ by adding the vector $2A[J\times I](\vec{\xi}[I]-\pvec{\xi}[I])$ as an additional row, where $\vec x^\transpose A \vec x$ is the quadratic part of $Q(\vec x)$. The statement of our robust rank inheritance lemma requires an assumption on $Q$; to specify this assumption we need another definition.

\begin{definition}\label{def:M-definition}
    For integers $r\ge 1$ and $0\le k\le m\le n$ and $s\ge 0$, let $\mathcal{M}_r^{k,m,n}(s)\subseteq\RR^{k\times m}\times\RR^{k\times n}\times\RR^{n\times m}$ be the set of triples of matrices $(T,U,A)\in \RR^{k\times m}\times\RR^{k\times n}\times\RR^{n\times m}$ for which there exist disjoint subsets
$I_{1},\dots,I_{s}\subseteq[m]$ and disjoint subsets $J_{1},\dots,J_{s}\subseteq[n]$ of size $|I_1|=\dots=|I_s|=|J_1|=\dots=|J_s|=k+r$ such that
\begin{compactitem}
    \item[(a)] For $t=1,\dots,s$, the submatrix $T[[k]\times I_t]$ has rank $k$.
    \item[(b)] For $t=1,\dots,s$, the submatrix $U[[k]\times J_t]$ has rank $k$. 
    \item[(c)] For $t=1,\dots,s$, every $(T[[k]\times I_t],U[[k]\times J_t])$-perturbation of the matrix $A[J_t\times I_t]$ has rank at least $r$.
\end{compactitem}
\end{definition}

\begin{remark}\label{rem:M-implies-H}
If $(T,U,A)\in \mathcal{M}_r^{k,m,n}(s)$ for any $r\ge 1$, then we automatically have $T\in \mathcal{H}^{k\times m}(s)$ and $U\in \mathcal{H}^{k\times n}(s)$. Also note that the property $(T,U,A)\in \mathcal{M}_r^{k,m,n}(s)$ is invariant under rescaling any of the matrices $T$, $U$ and $A$.
\end{remark}

In our proof of \cref{thm:recursion}, we consider a partition $[n]=I\cup J$ as outlined above, and take the matrix ``$A$'' in \cref{def:M-definition} to be the matrix $A[J\times I]$ together with $T=M[[k]\times I]$ and $U=M[[k]\times J]$. In this case,  $(M[[k]\times I],M[[k]\times J],A[J\times I])\in \mathcal{M}_r^{k,I,J}(s)$ says (roughly speaking) that $A'[J\times I]$ robustly
has rank at least $r$ for any $(M,M)$-perturbation $A'$ of $A$.

Now, our robust rank inheritance lemma is as follows.

\begin{lemma}\label{key-lemma-corollary}
Let $s\ge 1$ and $0\le k\le m\le n$ be integers. Consider $(T,U,A)\in\mathcal{M}_2^{k,m,n}(s)$, and vectors $\vec{y}\in\RR^{k}$
and $\vec{b}\in\RR^{n}$. Let $\vec{\xi}\in\{-1,1\}^{m}$
be a sequence of independent Rademacher random variables, and write $U_{\vec \xi}'\in \RR^{(k+1)\times n}$ for the (random) matrix obtained by appending the vector $A\vec{\xi}-\vec{b}$ as an additional row to $U$. Then we have
\[
\Pr[T\vec{\xi}=\vec{y}\text{ and }U_{\vec \xi}'\notin \mathcal H^{(k+1)\times n}(s/6)]\le \left(\frac{s}{10^{61}(k+2)^{20}}\right)^{-(k+2)/2}.
\]
\end{lemma}

We will prove \cref{key-lemma-corollary} in \cref{sec:key-lemma}. We remark that on the right-hand side of the inequality, both the ``$+2$'' in
the exponent $(k+2)/2$ and the ``$+2$'' in the denominator are due to the ``2'' in the assumption $(T,U,A)\in\mathcal{M}_2^{k,m,n}(s)$;
in general we would get a ``$+r$'' in both of these places if we assumed
that $(T,U,A)\in\mathcal{M}_r^{k,m,n}(s)$. The exponent $(k+2)/2$ here leads to the exponent $(k+2)/2$ appearing in \cref{thm:recursion}, and the ``$+2$'' there is crucial in order to deduce \cref{thm:main} from \cref{thm:recursion} (just having ``$+1$'' there would not suffice).

In order to actually apply \cref{key-lemma-corollary} in the proof of \cref{thm:recursion}, we would like to show
that there is a partition $[n]=I \cup J$ such that $(M[[k]\times I],M[[k]\times J],A[J\times I])\in\mathcal{M}_2^{k,I,J}(s)$ for a suitable value of $s$.
Such a partition might not in general exist, because we are making no assumption that $A$ itself even has rank at least 2. However, if every $(M,M)$-perturbation of $A$ robustly has rank at least 2 even after changing the diagonal entries, we can find a suitable partition $[n]=I \cup J$ using the following lemma.

\begin{lemma}\label{lem:matrix-splitting}
Let $0\le k\le n$ be integers and let $s\ge 4k+8$ be a real number. Let $M\in\mathcal{H}^{k\times n}(s)$ and let $A\in \RR^{n\times n}$ be a symmetric matrix such that $\on{rank} A^*[S\times S]\ge 2$ for any subset $S\subseteq [n]$ of size $|S|\ge n- s$ and any matrix $A^*\in \RR^{n\times n}$ that agrees with some $(M,M)$-perturbation of $A$ in all off-diagonal entries. Then there is a partition $[n]= I\cup J$, with $|I|\le s$, such that
\[(M[[k]\times I],M[[k]\times J],A[J\times I])\in\mathcal{M}_2^{k,I,J}(\lfloor s/(4k+8)\rfloor).\]
\end{lemma}

We prove \cref{lem:matrix-splitting} in \cref{sec:splitting}. Roughly speaking, we are able to \emph{greedily} find the desired subsets $I_1,\dots,I_s,J_1,\dots,J_s$ in the definition of $\mathcal{M}_2^{k,m,n}(s)$ in \cref{def:M-definition}.

To summarise, the proof of \cref{thm:recursion} combines a number of ingredients. If $A$ does not satisfy the assumption in \cref{lem:matrix-splitting} (i.e., if some $(M,M)$-perturbation of $A$ does not robustly have rank at least 2 after changing the diagonal entries), then we complete the proof using \cref{thm:quadratic-geometric} (in the low-rank case, there are various types of degeneracies to consider, which can be gracefully handled with a geometric point of view). Otherwise, we apply \cref{lem:matrix-splitting} to find a suitable partition $I\cup J$, and we apply the decoupling inequality in \cref{lem:decoupling} to $(\vec \xi[I],\vec \xi[J])$. We manipulate the resulting expression in roughly the same way as in the proof of \cref{thm:quadratic-geometric}, and then bound relevant quantities using \cref{key-lemma-corollary} and \cref{cor:halasz-sub-version}. After proving \cref{thm:recursion}, we can obtain \cref{thm:main} with  somewhat technical calculations (presented in \cref{sec:calculation-deduction}).

In the next section we state and prove a consequence of \cref{thm:quadratic-geometric} which will be suitable to handle the low-rank case of \cref{thm:recursion}. In \cref{sec:key-lemma} we prove \cref{key-lemma-corollary} and in \cref{sec:splitting} we prove \cref{lem:matrix-splitting}. The proof of \cref{thm:recursion} then appears in \cref{sec:recursion}.

\section{The low-rank case}\label{sec:low-rank}

In this section, we show how to use \cref{thm:quadratic-geometric} to deduce the following proposition, which handles the low-rank case of \cref{thm:recursion}.

\begin{proposition}\label{prop:low-rank-case}
Consider integers $0\le k\le n$ and $s\ge 0$ and $r\ge 1$. Let $M\in\mathcal{H}^{k\times n}(s)$, let $\vec w\in \RR^k$ and let $Q\in\RR[x_{1},\dots,x_{n}]$ be a quadratic polynomial. Writing the quadratic part of $Q(\vec{x})$ as $\vec{x}^{\transpose}A\vec{x}$ for a symmetric matrix $A$, assume that $\on{rank} A\le r-1$ and that there is no subset $I\su [n]$ of size $|I|\ge n-s$ such that the matrix $A[I\times I]$ is an $(M[[k]\times I],M[[k]\times I])$-perturbation of the zero matrix in $\mathbb{R}^{I\times I}$.

Then for a sequence $\vec{\xi}\in\{-1,1\}^{n}$
of independent Rademacher random variables we have
\begin{equation}\label{eq:low-rank-case}
\Pr[Q(\vec \xi)=0\text{ and }M\vec{\xi}=\vec{w}]\le\left(\frac{s}{2^{3r^{2}}(k+r)^2}\right)^{-(k+1)/2}.
\end{equation}
\end{proposition}

In our proof of \cref{thm:recursion}, we will use \cref{prop:low-rank-case} only for $r=5$, but we still state it here for general $r$ (as the proof works for any $r$).

Deducing \cref{prop:low-rank-case} from \cref{thm:quadratic-geometric} mostly consists of translating from ``geometric'' to ``algebraic'' language.
To give a some idea of how this works: note that under the assumptions of \cref{prop:low-rank-case}, we can write \[Q(\vec x)=\lambda_{1}\cdot (g^{({1})}(\vec{x}))^2+\dots+\lambda_{r-1} \cdot (g^{({r-1})}(\vec{x}))^2+g^{(r)}(\vec x)+c\] for some $\lambda_1,\dots,\lambda_{r-1},c\in \RR$ and some linear forms $g^{({1})},\dots,g^{({r-1})},g^{(r)} \in \mathbb{R}[x_1,\dots,x_n]$. That is to say, we can write
\[Q(\vec x)=P(g^{(1)}(\vec{x}),\dots,g^{(r)}(\vec{x}))=P(x_1\vec a_1+\dots+x_n\vec a_n)\]
for the quadratic polynomial $P(\vec{y})=\lambda_1 y_1^2+\dots+\lambda_{r-1} y_{r-1}^2+y_r+c$, where for $j=1,\dots,n$, the vector $\vec a_j\in \mb R^r$ records the coefficients of $x_j$ in the linear forms $g^{({1})},\dots,g^{(r)}$. So, the event $\{Q(\vec \xi)=0\}$ can be interpreted as the event that $\xi_1\vec a_1+\dots+\xi_n\vec a_n$ falls in the vanishing locus of the polynomial $P$. The joint event $\{Q(\vec \xi)=0\text{ and }M\vec\xi=\vec w \}$ can be given a similar geometric interpretation: roughly speaking, we consider $k$ additional linear forms $g^{({r+1})},\dots,g^{({r+k})}\in \mathbb{R}[x_1,\dots,x_n]$ corresponding to the $k$ rows of $M$,
augmenting $\vec a_1,\dots,\vec a_n$ accordingly with $k$ additional entries each, and then we consider the event that the $(r+k)$-dimensional random vector $\xi_1\vec a_1+\dots+\xi_n\vec a_n$ falls into a certain quadric $\mc Z$ (still described by the same polynomial $P$) on the codimension-$k$ affine-linear subspace $\mc W\subseteq \RR^{r+k}$ given by $\mc W=
\{\vec y \in \RR^{r+k} : \vec y [\{r+1,\dots,r+k\}=\vec w\}$ (this subspace $\mc W\subseteq \RR^{r+k}$ corresponds to the system of equations $M\vec\xi=\vec w$).

This more-or-less explains how to reduce \cref{prop:low-rank-case} to the setting of \cref{thm:quadratic-geometric}. The only complication is that it may not be possible to find many disjoint bases of $\RR^{r+k}$ among the vectors $\vec a_1,\dots,\vec a_n\in \RR^{r+k}$, as is necessary to apply \cref{thm:quadratic-geometric} (in fact, the vectors $\vec a_1,\dots,\vec a_n\in \RR^{r+k}$ might not span $\RR^{r+k}$ at all). In this case, we need to ``drop to a subspace'' that is robustly spanned by most of the $\vec a_1,\dots,\vec a_n$ (i.e., we need to identify some large subset $I\subseteq [n]$, such that we can find many disjoint bases of $\on{span}(\vec a_i:i\in I)$ among the vectors $\vec a_i$ for $i\in I$). This, roughly speaking, corresponds to considering only some subset of the linear forms $g^{({1})},\dots,g^{(r+k)}$, such that the remaining linear forms are ``close to'' being linear combinations of these linear forms (i.e., each of the remaining linear forms agrees with such a linear combination in its coefficients of the variables $x_i$ for $i\in I$). The quadratic polynomial $P(\vec{y})$ considered above then needs to be changed appropriately in order to reflect these relations between the linear forms $g^{({1})},\dots,g^{(r+k)}$.

The following lemma encapsulates this translation discussed above, going from the setting of \cref{prop:low-rank-case} to the setting of \cref{thm:quadratic-geometric}.

\begin{lemma}\label{lem:low-rank-preparation}
Consider integers $0\le k\le n$ and $s\ge 0$ and $r\ge 1$. Let $M\in \mathcal{H}^{k\times n}(s)$ and let $Q\in \mathbb{R}[x_1\dots,x_n]$ be a quadratic polynomial. Writing the quadratic part of $Q(\vec{x})$ as $\vec{x}^{\transpose}A\vec{x}$ for a symmetric matrix $A\in \mathbb{R}^{n\times n}$, assume that $\on{rank} A\le r-1$ and that there is no subset $I\su [n]$ of size $|I|\ge n-s$ such that the matrix $A[I\times I]$ is an $(M[[k]\times I],M[[k]\times I])$-perturbation of the zero matrix in $\mathbb{R}^{I\times I}$.

Then, we can find a positive integer $\ell\le r$, a partition $[n]=I\cup S$ with $|S|\le s$, linear forms $g^{(1)},\dots,g^{({\ell+k})} \in \mathbb{R}[x_1,\dots,x_n]$, and a quadratic polynomial $P\in \mathbb{R}[y_1,\dots,y_{\ell+k}, (x_i)_{i\in S}]$ such that the following conditions hold.
\begin{compactitem}
    \item[(i)] $Q(x_1,\dots,x_n)=P(g^{(1)}(\vec{x}),\dots,g^{({\ell+k})}(\vec{x}), (x_i)_{i\in S})$.
    \item[(ii)]  For each $j=1,\dots,k$, the coefficient vector of $g^{({\ell+j})}$ is precisely the $j$-th row of $M$.
    \item[(iii)]  Writing $M'\in \mathbb{R}^{(\ell+k)\times n}$ for the $(\ell+k)\times n$ matrix whose $j$-th row is the coefficient vector of $g^{(j)}$ for $j=1,\dots,\ell+k$, we have $M'[[\ell+k]\times I]\in \mathcal{H}^{(\ell+k)\times I}(s/r)$.
    \item[(iv)]  There exist $j,j'\in [\ell]$ such that the coefficient of $y_jy_{j'}$ in $P$ is nonzero.
\end{compactitem}
\end{lemma}

We remark that the vectors $\vec a_i$ for $i\in I$ in the informal explanation above correspond to the columns of the matrix $M'[[\ell+k]\times I]$ in \cref{lem:low-rank-preparation}. Condition (iii) ensures that we can find many disjoint bases of $\mathbb{R}^{\ell+k}$ among these vectors $\vec a_i$ for $i\in I$ (recall that this is required to apply  \cref{thm:quadratic-geometric} to these vectors). Condition (iv) ensures that the polynomial $P$ does not vanish on our entire subspace $\mathcal{W}$ (corresponding to a system of equations of the form $M\vec{\xi}=\vec{w}$).

\begin{proof}
Let us consider the minimum number $\ell\in\{0,1,\dots,r\}$ such that there exists a partition $[n]=I\cup S$ with $|S|\le s\cdot (r-\ell)/r$, linear forms $g^{(1)},\dots,g^{({\ell+k})} \in \mathbb{R}[x_1,\dots,x_n]$, and a quadratic polynomial $P\in \mathbb{R}[y_1,\dots,y_{\ell+k}, (x_i)_{i\in S}]$, such that conditions (i) and (ii) hold (i.e., we have $Q(x_1,\dots,x_n)=P(g^{(1)}(\vec{x}),\dots,g^{({\ell+k})}(\vec{x}), (x_i)_{i\in S})$ and for each $j=1,\dots,k$, the coefficient vector of $g^{({\ell+j})}$ is precisely the $j$-th row of $M$).

First, in order to see that this minimum number $\ell$ is well-defined, let us check that $\ell=r$ satisfies the conditions. Indeed, given that $\operatorname{rank} A\le r-1$,  we can write \[Q(\vec x)=\lambda_{1}\cdot (g^{({1})}(\vec{x}))^2+\dots+\lambda_{r-1} \cdot (g^{({r-1})}(\vec{x}))^2+g^{(r)}(\vec x)+c\] for some $\lambda_1,\dots,\lambda_{r-1},c\in \RR$ and some linear forms $g^{({1})},\dots,g^{({r-1})},g^{(r)} \in \mathbb{R}[x_1,\dots,x_n]$. That is to say, defining $g^{({r+1})},\dots,g^{({r+k})} \in \mathbb{R}[x_1,\dots,x_n]$ to be the linear forms whose coefficient vectors are given by the rows of $M$ as in condition (ii), we can write
\[Q(\vec x)=P(g^{(1)}(\vec{x}),\dots,g^{(r+k)}(\vec{x})),\]
where $P(y_1,\dots,y_{r+k})=\lambda_1 y_1^2+\dots+\lambda_{r-1} y_{r-1}^2+y_r+c$ (formally this is a polynomial in $r+k$ variables, despite the fact that $y_{r+1},\dots,y_{r+k}$ do not actually appear in it). Note that (ii) holds by definition, and taking $I=[n]$ and $S=\emptyset$, condition (i) holds as well and we have $|S|=0\le s\cdot (r-r)/r$. Thus, $\ell=r$ has the required properties, and the minimum number $\ell$ above is well-defined.

Now let $\ell$ be this minimum number, and choose the partition $[n]=I\cup S$, the linear forms $g^{(1)},\dots,g^{({\ell+k})} \in \mathbb{R}[x_1,\dots,x_n]$, and the quadratic polynomial $P$ accordingly with the properties above. Note that $|S|\le s\cdot (r-\ell)/r\le s$, and consequently $|I|\ge n-s$. It remains to show that $\ell\ge 1$ and that conditions (iii) and (iv) hold.

Let us first show (iv). Suppose for contradiction that in the quadratic polynomial $P$, the coefficient of $y_jy_{j'}$ is zero for all $j,j'\in [\ell]$. Then the quadratic part of $Q(x_1,\dots,x_n)=P(g^{(1)}(\vec{x}),\dots,g^{({\ell+k})}(\vec{x}), (x_i)_{i\in S})$ can be written as a linear combination of terms of the form $g^{({j})}(\vec{x})x_h$ with $j\in \{\ell+1,\dots,\ell+k\}$ and $h\in [n]$, and terms of the form $x_hx_{i}$ with $h\in [n]$ and $i\in S$. By \cref{eq:product-linear-forms}, this leads to a representation of the symmetric matrix $A$ as a linear combination of matrices of the form $\vec{v}_j\vec{e}_h^{\,\transpose}$, $\vec{e}_h\vec{v}_j^{\,\transpose}$, $\vec{e}_i\vec{e}_{h}^{\,\transpose}$ and $\vec{e}_h\vec{e}_{i}^{\,\transpose}$ for $j\in \{\ell+1,\dots,\ell+k\}$, $h\in [n]$ and $i\in S$, where $\vec{e}_1,\dots,\vec{e}_n$ are the standard basis vectors in $\RR^n$, and $\vec{v}_{\ell+1},\dots,\vec{v}_{\ell+k}\in \mathbb{R}^n$ denote the coefficient vectors of $g^{({\ell+1})},\dots,g^{({\ell+k})}$ (i.e., the row vectors of $M$, by condition (ii)). 
Hence the matrix $A[I\times I]$ is a linear combination of matrices of the form $\vec{v}_j[I]\vec{e}_h^{\,\transpose}[I]$ and $\vec{e}_h[I]\vec{v}_j^{\,\transpose}[I]$ 
for $j\in \{\ell+1,\dots,\ell+k\}$ and $h\in I$. Recalling that $\vec{v}_{\ell+1},\dots,\vec{v}_{\ell+k}$ are the row vectors of the matrix $M$, this means that $A[I\times I]$ is a $(M[[k]\times I],M[[k]\times I])$-perturbation of the zero matrix, contradicting our assumption. So indeed, for some monomial $y_jy_{j'}$ with $j,j'\in [\ell]$, the coefficient of $y_jy_{j'}$ in $P$ must be nonzero. This also automatically implies that $\ell\ge 1$.

It now remains to show (iii), i.e., to show that for the matrix $M'\in \mathbb{R}^{(\ell+k)\times n}$ whose rows are given by the coefficient vectors of $g^{(1)},\dots,g^{({\ell+k})}$ we have $M'[[\ell+k]\times I]\in \mathcal{H}^{(\ell+k)\times I}(s/r)$. Assume the contrary; then there exists a nontrivial linear combination of the rows of $M'[[\ell+k]\times I]$ yielding a vector with at most $s/r$ nonzero entries. By (ii), the rows of the matrix $M'[\{\ell+1,\dots,\ell+k\}\times [n]]$ agree with the rows of the matrix $M$, and we therefore have $M'[\{\ell+1,\dots,\ell+k\}\times [n]]\in \mathcal{H}^{k\times n}(s)$ (recalling that $M\in \mathcal{H}^{k\times n}(s)$). Using $|S|\le s\cdot (r-\ell)/r\le s\cdot (r-1)/r$, we can conclude that  $M'[\{\ell+1,\dots,\ell+k\}\times I]\in \mathcal{H}^{k\times I}(s-|S|)\su \mathcal{H}^{k\times I}(s/r)$, so our linear combination cannot only involve rows with indices in $\{\ell+1,\dots,\ell+k\}$. Thus, we may assume without loss of generality that it involves the first row of $M'[[\ell+k]\times I]$, meaning that this first row differs in at most $s/r$ entries from some linear combination of the other rows of $M'[[\ell+k]\times I]$. In other words, we can write $g^{({1})}$ as a linear combination of $g^{(2)},\dots,g^{({\ell+k})}$ and $x_i$ for $i\in S\cup S'$, for some subset $S'\su [n]$ of size $|S'|\le s/r$. But this means that $Q(x_1,\dots,x_n)=P(g^{(1)}(\vec{x}),\dots,g^{({\ell+k})}(\vec{x}), (x_i)_{i\in S})$ can be written as $P^*(g^{(2)}(\vec{x}),\dots,g^{({\ell+k})}(\vec{x}), (x_i)_{i\in S\cup S'})$ for some quadratic polynomial $P^*\in \mathbb{R}[y_2,\dots,y_{\ell+k}, (x_i)_{i\in S\cup S'}]$. Since $|S\cup S'|\le |S|+|S'|\le s\cdot (r-\ell)/r+s/r=s\cdot (r-\ell+1)/r$, this contradicts the minimality of $\ell$ (we now have a suitable representation of $Q$ in terms of the $(\ell-1)+k$  linear forms $g^{(2)},\dots,g^{({\ell+k})}$).
So we must have $M'[[\ell+k]\times I]\in \mathcal{H}^{(\ell+k)\times I}(s/r)$, as desired.
\end{proof}

Using \cref{lem:low-rank-preparation} and \cref{thm:quadratic-geometric}, we now prove \cref{prop:low-rank-case}.

\begin{proof}[Proof of \cref{prop:low-rank-case}]
First, note that we may assume that $s\ge 2^{3r^2}(k+r)^2$, since otherwise \cref{eq:low-rank-case} is trivially true. Now, apply \cref{lem:low-rank-preparation} to obtain a positive integer $\ell\le r$, a partition $[n]=I\cup S$, linear forms $g^{(1)},\dots,g^{(\ell+k)}\in \mathbb{R}[x_1,\dots,x_n]$ and a quadratic polynomial $P\in \mathbb{R}[y_1,\dots,y_{\ell+k}, (x_i)_{i\in S}]$ satisfying conditions (i) to (iv) in the statement of the lemma.

Our plan is to show that the bound in \cref{eq:low-rank-case} holds even if we condition on an arbitrary outcome of $\vec \xi[S]$ (leaving only the randomness in $\vec \xi[I]$).
For any outcome of $\vec \xi[S]$, when plugging in $x_i=\vec{\xi}[i]$ for $i\in S$ into the polynomial $Q(\vec x)$ and the linear forms $g^{(1)}(\vec x),\dots,g^{({\ell+k})}(\vec x)$, we obtain a polynomial $Q_{\vec \xi[S]}(\vec x[I])$ and linear functions $g^{(1)}_{*}(\vec x[I])+c^{(1)}_{\vec \xi[S]},\dots,g^{({\ell+k})}_{*}(\vec x[I])+c^{(\ell+k)}_{\vec \xi[S]}$ (where $g^{(1)}_{*}(\vec x[I]),\dots,g^{({\ell+k})}_{*}(\vec x[I])$ are the linear forms obtained from $g^{(1)}(\vec x),\dots,g^{({\ell+k})}(\vec x)$ by omitting all terms with variables $x_i$ for $i\in S$, and where $c^{(1)}_{\vec \xi[S]},\dots,c^{({\ell+k})}_{\vec \xi[S]}\in \RR$ are real numbers that may depend on $\vec \xi[S]$) and by (i) we obtain that $Q_{\vec \xi[S]}(\vec{x}[I])=P\bigl(g^{(1)}_{*}(\vec x[I])+c^{(1)}_{\vec \xi[S]},\dots,g^{({\ell+k})}_{*}(\vec x[I])+c^{(\ell+k)}_{\vec \xi[S]}, (\xi[i])_{i\in S}\bigr)$. For any outcome of $\vec \xi[S]$, we can furthermore write $P\bigl(y_1+c^{(1)}_{\vec \xi[S]},\dots,y_{\ell+k}+c^{(\ell+k)}_{\vec \xi[S]},(\xi[i])_{i\in S}\bigr)=P_{\vec \xi[S]}(y_1,\dots,y_{\ell+k})$ for a polynomial $P_{\vec \xi[S]}\in \RR[y_1,\dots,y_{\ell+k}]$ whose coefficients may depend on $\vec \xi[S]$. Then we always have $Q_{\vec \xi[S]}(\vec{x}[I])=P_{\vec \xi[S]}\bigl(g^{(1)}_{\vec \xi[S]}(\vec{x}[I]]),\dots,g^{({\ell+k})}_{\vec \xi[S]}(\vec{x}[I])\bigr)$. Recall that the coefficient vectors of the linear forms $g^{(1)}_{*}(\vec x[I]),\dots,g^{({\ell+k})}_{*}(\vec x[I])$ are obtained by restricting the coefficient vectors of $g^{(1)}(\vec x),\dots,g^{({\ell+k})}(\vec x)$ to the index set $I$. In particular, by condition (ii), for $j=1,\dots,k$ the coefficient vector of $g^{({\ell+j})}_{*}(\vec x[I])$ is the restriction $M[[j]\times I]$ of the $j$-th row of $M$ to $I$. We then have
\begin{align*}
&\Pr\Bigl[Q(\vec \xi)=0\text{ and }M\vec{\xi}=\vec{w}\,\Big|\,\vec \xi[S]\Bigr]\\
&\qquad =\Pr\Bigl[Q_{\vec \xi[S]}(\vec \xi[I])=0\text{ and }M[[k]\times I]\vec{\xi}[I]=\vec{w}_{\vec \xi[S]}\,\Big|\,\vec \xi[S]\Bigr]\\
&\qquad =\Pr\Bigl[P_{\vec \xi[S]}(g^{(1)}_{*}(\vec{\xi}[I]),\dots,g^{({\ell+k})}_{*}(\vec{\xi}[I]))=0\text{ and }g^{({\ell+j})}_{*}(\vec{\xi}[I])=\vec{w}_{\vec \xi[S]}[j]\text{ for }j=1,\dots,k\,\Big|\,\vec \xi[S]\Bigr]
\end{align*}
for any outcome of $\vec \xi[S]$, where $\vec{w}_{\vec \xi[S]}=\vec{w}-M[[k]\times S]\vec{\xi}[S]\in \RR^k$.
We wish to show that for all outcomes of $\vec \xi[S]$, the above conditional probability is bounded by the right-hand side of \cref{eq:low-rank-case}.

Recall that $P\in \mathbb{R}[y_1,\dots,y_{\ell+k}, (x_i)_{i\in S}]$ is a quadratic polynomial, and by condition (iv), for some $j,j'\in [\ell]$ the coefficient of $y_j y_{j'}$ in $P$ is nonzero. This means that the coefficient of $y_jy_{j'}$ in $P_{\vec \xi[S]}(y_1,\dots,y_{\ell+k})\in \mathbb{R}[y_1,\dots,y_{\ell+k}]$ is still nonzero, for any outcome of $\vec \xi[S]$.

Since the coefficient vectors of the linear forms $g^{(1)}(\vec x),\dots,g^{({\ell+k})}(\vec x)\in \RR[x_1,\dots,x_n]$ are the rows of the matrix $M'$ in condition (iii), the coefficient vectors of $g^{(1)}_{*}(\vec x[I]),\dots,g^{({\ell+k})}_{*}(\vec x[I])\in \RR[x_i: i\in I]$ are the rows of the matrix $M'[[\ell+k]\times I]$. Now, define the vectors $\vec{a}_i\in \mathbb{R}^{\ell+k}$, for $i\in I$, to be the columns of the matrix $M'[[\ell+k]\times I]$. Then for any outcome of $\vec{\xi}[I]$, the vector $\bigl(g^{(1)}_{*}(\vec{\xi}[I]),\dots,g^{({\ell+k})}_{*}(\vec{\xi}[I])\bigr)$ agrees with $\sum_{i\in I} \vec{\xi}[i]\vec{a}_i$.

Furthermore, as $M'[[\ell+k]\times I]\in \mathcal{H}^{(\ell+k)\times I}(s/r)$, by \cref{lem-halasz-equivalent}(ii) the matrix $M'[[\ell+k]\times I]$ must contain $\lceil s/(r(\ell+k))\rceil\ge \lceil s/(k+r)^2\rceil$ disjoint nonsingular $(\ell+k)\times (\ell+k)$ submatrices, so among the vectors $\vec{a}_i$ for $i\in I$ we can form $\lceil s/(k+r)^2\rceil$ disjoint bases of $\mathbb{R}^{\ell+k}$. Consider the largest integer $t\le \lceil s/(k+r)^2\rceil$ which is divisible by $2^{\ell-1}$, and note that $t\ge \lceil s/(k+r)^2\rceil/2\ge s/(2(k+r)^2)$, since $\lceil s/(k+r)^2\rceil\ge 2^{3r^2}\ge 2^{\ell-1}$.

Now, for any outcome of $\vec \xi[S]$, we obtain
\begin{align*}
&\Pr\Bigl[Q(\vec \xi)=0\text{ and }M\vec{\xi}=\vec{w}\,\Big|\,\vec \xi[S]\Bigr]\\
&\qquad=
\Pr\Biggl[P_{\vec \xi[S]}\biggl(\sum_{i\in I}\vec \xi[i]\vec{a}_i\biggr)=0\text{ and }\sum_{i\in I}\vec \xi[i]\vec{a}_i[\ell+j]=\vec{w}_{\vec \xi[S]}[j]\text{ for }j=1,\dots,k\,\Bigg|\,\vec \xi[S]\Biggr]\\
&\qquad=\Pr\Biggl[P_{\vec \xi[S]}\biggl(\sum_{i\in I}\vec \xi[i]\vec{a}_i\biggr)=0\text{ and }\sum_{i\in I}\vec \xi[i]\vec{a}_i\in \mathcal{W}_{\vec \xi[S]}\,\Big|\,\vec \xi[S]\Biggr]=\Pr\Biggl[\sum_{i\in I}\vec \xi[i]\vec{a}_i\in \mathcal{Z}_{\vec \xi[S]}\,\Bigg|\,\vec \xi[S]\Biggr],
\end{align*}
where $\mathcal{W}_{\vec \xi[S]}\su \RR^{\ell+k}$ is the $\ell$-dimensional affine-linear subspace consisting of the points $\vec{y}\in \RR^{\ell+k}$ with $\vec{y}[\ell+j]=\vec{w}_{\vec \xi[S]}[j]$ for $j=1,\dots,k$, and $\mathcal{Z}_{\vec \xi[S]}\su \mathcal{W}_{\vec \xi[S]}$ is the subset of $\mathcal{W}_{\vec \xi[S]}$ given by $\mathcal{Z}_{\vec \xi[S]}=\{\vec{y}\in \mathcal{W}_{\vec \xi[S]}: P_{\vec \xi[S]}(\vec{y})=0\}$. 

We claim that for any outcome of $\vec \xi[S]$, we have $\mathcal{Z}_{\vec \xi[S]}\subsetneq \mathcal{W}_{\vec \xi[S]}$, i.e., $\mathcal{Z}_{\vec \xi[S]}$ is a quadric on $\mathcal{W}_{\vec \xi[S]}$. Indeed, if we had $\mathcal{Z}_{\vec \xi[S]}= \mathcal{W}_{\vec \xi[S]}$, then the polynomial $P_{\vec \xi[S]}$ would be identically zero on the entire subspace $\mathcal{W}_{\vec \xi[S]}\su \RR^{\ell+k}$. Note that on the space $\mathcal{W}_{\vec \xi[S]}$, we can identify $P_{\vec \xi[S]}$ with some quadratic polynomial in the variables $y_{1},\dots,y_{\ell}$ (obtained by substituting $y_{\ell+j}=\vec{w}_{\vec \xi[S]}[j]$ for $j=1,\dots,k$). This polynomial has a nonzero coefficient for some monomial of the form $y_jy_{j'}$ with $j,j'\in [\ell]$, since $P_{\vec \xi[S]}$ also has a nonzero coefficient for such a monomial. Hence the polynomial $P_{\vec \xi[S]}$ cannot vanish on the entire subspace $\mathcal{W}_{\vec \xi[S]}$, so indeed $\mathcal{Z}_{\vec \xi[S]}\subsetneq \mathcal{W}_{\vec \xi[S]}$.

Recalling that among the vectors $\vec{a}_i$ for $i\in I$ we can form $t$ disjoint bases of $\mathbb{R}^{\ell+k}$, we can now apply \cref{thm:quadratic-geometric} (after conditioning on any outcome of $\vec \xi[S]$) and obtain the bound
\begin{align*}
&\Pr\Bigl[Q(\vec \xi)=0\text{ and }M\vec{\xi}=\vec{w}\,\Big|\, \vec \xi[S]\Bigr]=\Pr\Biggl[\sum_{i\in I}\vec \xi[i]\vec{a}_i\in \mathcal{Z}_{\vec \xi[S]}\,\Bigg|\, \vec \xi[S]\Biggr]\\
&\qquad\qquad\qquad\qquad\le \frac{2^{(\ell-1)(\ell+k)+1}}{t^{((\ell+k)-(\ell-1))/2}}\le \frac{2^{(\ell-1)\ell(k+1)+1}}{t^{(k+1)/2}}\le \frac{2^{\ell^2(k+1)}}{(s/(2(k+r)^2))^{(k+1)/2}}\\
&\qquad\qquad\qquad\qquad\le \left(\frac{2^{3\ell^2}(k+r)^2}{s}\right)^{(k+1)/2}\!\!= \left(\frac{s}{2^{3r^2}(k+r)^2}\right)^{-(k+1)/2}
\!,\end{align*}
as desired.
\end{proof}

\section{Robust rank inheritance}\label{sec:key-lemma}

In this section, we prove \cref{key-lemma-corollary} (the robust rank inheritance lemma). Actually, we will deduce \cref{key-lemma-corollary} from the following somewhat simpler statement with just two matrices $T,A$ instead of three matrices $T,U,A$ (to deduce \cref{key-lemma-corollary}, we will take $r=2$ in the statement below).

\begin{lemma}\label{key-lemma}
Let $r\ge 1$ and $s\ge 1$ and $0\le k\le m\le n$ be integers. Consider matrices $T\in \RR^{k\times m}$ and $A\in \RR^{n\times m}$, and vectors $\vec w\in\RR^{k}$ and $\vec v\in\RR^{n}$. Assume that there exist disjoint subsets $I_1,\dots,I_s\su [m]$ and disjoint subsets $J_1,\dots,J_s\su [n]$ of size $|I_1|=\dots=|I_s|=|J_1|=\dots=|J_s|=k+r$ such that:
\begin{compactitem}
    \item[(i)] For $t=1,\dots,s$, the submatrix $T[[k]\times I_t]$ has rank $k$. 
    \item[(ii)] For $t=1,\dots,s$, every $(T[[k]\times I_t],0)$-perturbation of the matrix $A[J_t\times I_t]$ has rank at least $r$.
\end{compactitem}
Then, for a sequence of independent Rademacher random variables $\vec{\xi}=(\xi_{1},\dots,\xi_{m})\in\{-1,1\}^{m}$,
we have
\begin{equation}\label{eq:key-lemma-equation}
\Pr[T\vec{\xi}=\vec{w}\text{ and }A\vec{\xi}-\vec{v}\text{ has at most }s/6\text{ nonzero coordinates}]\le\left(\frac{s}{10^{60}(k+r)^{20}}\right)^{-(k+r)/2}.
\end{equation}
\end{lemma}

Note that if $k=0$, then condition (i) is vacuous, and in condition (ii), there are no nontrivial $(T[[k]\times I_t],0)$-perturbations. This case of \cref{key-lemma} implies \cref{lem:baby-key} (the ``Hamming norm'' anticoncentration inequality mentioned in \cref{sec:outline}); the formal deduction can be found at the end of this section.

Also, note that assumptions (i) and (ii) in particular imply that for all $t=1,\dots,s$ we have
\begin{equation}\label{eq:assumption-key-lemma-consequence}
\on{rank}\begin{pmatrix}A[J_t\times I_t]\\ T[[k]\times I_t]\end{pmatrix}\ge k+r.
\end{equation}
Indeed, if this rank were at most $k+r-1$, then one would be able to form a basis of the row span of the above matrix using the $k$ rows of $T[[k]\times I_t]$  and up to $r-1$ rows of $A[J_t\times I_t]$ (recall that by assumption (i) the rows of the matrix $T[[k]\times I_t]$ are linearly independent). But then it would be possible to add linear combinations of the rows of $T[[k]\times I_t]$ to the rows of $A[J_t\times I_t]$ to obtain a matrix of rank at most $r-1$, contradicting assumption (ii).

In order to prove \cref{key-lemma} we will employ a \emph{witness-counting} approach (as outlined in \cref{subsec:witness-count}). We need to bound the probability that $T\vec{\xi}=\vec{w}$ and $A\vec{\xi}-\vec{v}$ has more than $n-s/6$ zero coordinates. Note that we can use Hal\'asz' inequality (\cref{thm-Halasz-FJZ}) to upper-bound the probability
that $T\vec{\xi}=\vec{w}$ and a particular set of coordinates of
$A\vec{\xi}-\vec{v}$ are zero. However, it is far too wasteful to simply
take a union bound over all subsets of $n-s/6$ coordinates. Instead, we consider certain types of ``witness'' sequences $(h_1,\dots,h_z)\in [n]^z$ of indices where the coordinates of $A\vec{\xi}-\vec{v}$ are zero. If $A\vec{\xi}-\vec{v}$ has many zero coordinates, there must be many such ``witness'' sequences, but on the other hand we can bound the expected number of such ``witness'' sequences by considering the probability that $(A\vec{\xi}-\vec{v})[\{h_1,\dots,h_z\}]=\vec 0$ (and $T\vec{\xi}=\vec{w}$) for a given such sequence $(h_1,\dots,h_z)$.

The following lemma states that under the assumptions of \cref{key-lemma}, one can find a large submatrix of $A$ which ``has its
rank robustly'' (the assumptions of the assumptions of \cref{key-lemma} guarantee that $A$ robustly has rank \emph{at least}
$r$, but the rank of $A$  may actually be much larger than $r$, and that larger rank may be ``fragile''). More precisely, we can find such a submatrix even within any large specified set $H$ of rows.

\begin{lemma}\label{lem:drop-robust}
Let $r,k,s,m,n\in \mathbb{Z}$, the matrices $T\in \RR^{k\times m}$ and $A\in \RR^{n\times m}$, and the sets $J_1,\dots,J_s\su [n]$ be as in \cref{key-lemma}. Then for any subset $H\su J_1\cup\dots\cup J_s$ of size $|H|\ge (k+r)s-(2/3)s$, there are subsets $J\subseteq H$ and $I\subseteq [m]$ with $|J|\ge |H|-s/6$ and $|I|\ge m-s/6$ and an integer $z\ge r$ such that
\begin{equation}\label{eq:drop-robust-lemma}
\on{rank}\begin{pmatrix}A[J\times I]\\ T[[k]\times I]\end{pmatrix}=\on{rank}\begin{pmatrix}A[J'\times I']\\ T[[k]\times I']\end{pmatrix}=k+z
\end{equation}
for all subsets $J'\su J$ and $I'\su I$ of sizes $|J'|\ge |J|-s/(12z^2)$ and $|I'|\ge |I|-s/(12z^2)$.
\end{lemma}
\begin{proof}
For every integer $z\ge 0$, define $f(z)=\sum_{y=z+1}^{\infty} s/(12y^2)$ and note that $f(z)\le \sum_{y=1}^{\infty} s/(12y^2)< s/6$ for all $z\ge 0$. Now consider the minimum integer $z\ge 0$ such that there exist subsets $J\subseteq H$ and $I\subseteq[m]$ with $|J|\ge |H|-f(z)$ and $|I|\ge m-f(z)$ and
\[\on{rank}\begin{pmatrix}A[J\times I]\\ T[[k]\times I]\end{pmatrix}\le k+z\]
(such an integer $z$ exists, for example $z=m$ satisfies the condition when taking $J= H$ and $I=[m]$). We claim that $z\ge r$. Indeed, note that $|J|\ge |H|-f(z)> (k+r)s-(2/3)s-s/6=|J_1\cup\dots\cup J_s|-(5s/6)$ and $|I|\ge m-f(z)> m-s/6$, so there are strictly fewer than $s$ indices $t\in \{1,\dots,s\}$ such that $I_t\setminus I\ne \emptyset$ or $J_t\setminus J\ne \emptyset$. Hence there must be some index $t\in \{1,\dots,s\}$ with $I_t\su I$ and $J_t\su J$, and we have
\[k+z\ge \on{rank}\begin{pmatrix}A[J\times I]\\ T[[k]\times I]\end{pmatrix}\ge \on{rank}\begin{pmatrix}A[J_t\times I_t]\\ T[[k]\times I_t]\end{pmatrix}\ge k+r,\]
where the last step follows from the assumptions in \cref{key-lemma}, see \cref{eq:assumption-key-lemma-consequence}. So we indeed have $z\ge r\ge 1$.

Finally, for any subsets $I'\su I$ and $J'\su J$ of sizes $|I'|\ge |I|-s/(12z^2)$ and $|J'|\ge |J|-s/(12z^2)$, we have $|I'|\ge m-f(z)-s/(12z^2)=m-f(z-1)$ and $|J'|\ge |H|-f(z)-s/(12z^2)=|H|-f(z-1)$, so by the choice of $z$ we must have
\[k+z\ge \on{rank}\begin{pmatrix}A[J\times I]\\ T[[k]\times I]\end{pmatrix}\ge \on{rank}\begin{pmatrix}A[J'\times I']\\ T[[k]\times I']\end{pmatrix}\ge k+(z-1)+1=k+z,\]
implying \cref{eq:drop-robust-lemma}.
\end{proof}

We can use the robust-rank submatrix guaranteed by \cref{cor:drop-robust} to find many sequences $(h_1,\dots,h_z)$ to be used as ``witnesses'', as follows.

\begin{corollary}\label{cor:drop-robust}
Let $r,k,s,m,n\in \mathbb{Z}$, the matrices $T\in \RR^{k\times m}$ and $A\in \RR^{n\times m}$, and the sets $J_1,\dots,J_s\su [n]$ be as in \cref{key-lemma}. Then for any subset $H\su J_1\cup\dots\cup J_s$ of size $|H|\ge (k+r)s-(2/3)s$, there is some integer $z\ge r$ such that there are at least $s^z/(12z^2)^z$ sequences $(h_1,\dots,h_z)\in H^z$ with
\begin{equation}\label{eq:Halasz-good-z-tuple}
\begin{pmatrix}A[\{h_{1},\dots,h_{z}\}\times[m]]\\
T
\end{pmatrix}\in \mc H^{(k+z)\times m}(s/(12z^2)).
\end{equation}
\end{corollary}
\begin{proof}
    Let $J\subseteq H$ and $I\subseteq [m]$ and $z\ge r$ be as in \cref{lem:drop-robust}. We first claim that there are at least $s^z/(12z^2)^z$ sequences $(h_1,\dots,h_z)\in J^z\su H^z$ such that
    \begin{equation}\label{eq:Halasz-good-z-tuple-with-I}
    \on{row-span}\begin{pmatrix}A[\{h_{1},\dots,h_{z}\}\times I]\\
    T[[k]\times I]
    \end{pmatrix}
    =\on{row-span}\begin{pmatrix}A[J\times I]\\
    T[[k]\times I]
    \end{pmatrix}
    \end{equation}
    Indeed, the matrix on the right-hand side has rank $k+z$ by \cref{eq:drop-robust-lemma}, and the rows of $T[[k]\times I]$ are linearly independent by assumption (i) in \cref{key-lemma} (since $|I|\ge m-s/6$, there needs to be at least one index $t\in \{1,\dots,s\}$ with $I_t\su I$, meaning that $\on{rank} T[[k]\times I]=T[[k]\times I_t]=k$). So we can form a basis of the row span of the matrix on the right-hand side by taking the rows of $T[[k]\times I]$ and adding, one by one, $z$ different rows of $A[J\times I]$. By \cref{eq:drop-robust-lemma}, at every step we have at least $s/(12z^2)$ choices for a new row to add to our basis (indeed, the index set $J'\su J$ of all rows in the span of the already selected basis elements must have size $|J'|<|J|-s/(12z^2)$, since this span has dimension less than $k+z$ and so otherwise we would have a contradiction to \cref{eq:drop-robust-lemma} with $I'=I$). Thus, there are indeed at least $s^z/(12z^2)^z$ choices for the sequence $(h_1,\dots,h_z)\in J^z$ of indices of the rows of $A[J\times I]$ selected in this process. For each such sequence \cref{eq:Halasz-good-z-tuple-with-I} holds.

    It remains to show that every sequence $(h_1,\dots,h_z)\in J^z$ satisfying \cref{eq:Halasz-good-z-tuple-with-I} also satisfies \cref{eq:Halasz-good-z-tuple}. So, consider any $(h_1,\dots,h_z)\in J^z$ satisfying \cref{eq:Halasz-good-z-tuple-with-I}. To show \cref{eq:Halasz-good-z-tuple}, it suffices to show that
    \begin{equation}\label{eq:Halasz-good-z-tuple-with-I-2}
    \begin{pmatrix}A[\{h_{1},\dots,h_{z}\}\times I]\\
    T[[k]\times I]
    \end{pmatrix}\in \mc H^{(k+z)\times I}(s/(12z^2)).
    \end{equation}
    Suppose the contrary; then there exists a subset $I'\su I$ of size $|I'|\ge |I|-s/(12z^2)$ such that
    \[\on{rank} \begin{pmatrix}A[\{h_{1},\dots,h_{z}\}\times I']\\
    T[[k]\times I']
    \end{pmatrix}<k+z.\]
    But now by \cref{eq:Halasz-good-z-tuple-with-I} we have
    \[\on{row-span}\begin{pmatrix}A[J\times I']\\
    T[[k]\times I']
    \end{pmatrix}
    =\on{row-span}\begin{pmatrix}A[\{h_{1},\dots,h_{z}\}\times I']\\
    T[[k]\times I']
    \end{pmatrix}\]
    and hence
    \[\on{rank} \begin{pmatrix}A[J\times I']\\
    T[[k]\times I']
    \end{pmatrix}
    =\on{rank} \begin{pmatrix}A[\{h_{1},\dots,h_{z}\}\times I']\\
    T[[k]\times I']
    \end{pmatrix}<k+z,\]
    contradicting \cref{eq:drop-robust-lemma}. So we indeed have \cref{eq:Halasz-good-z-tuple-with-I-2}, as desired.
\end{proof}

Now we prove \cref{key-lemma}. We need two slightly different implementations of our witness-counting strategy described above, with different conditions for the ``witness'' sequences $(h_1,\dots,h_z)$ we are counting. In the case where $A$ ``robustly has very high rank'', we can get away with rather crude arguments using Odlyzko's lemma. In the complementary case,
we need to use \cref{cor:drop-robust} and Hal\'asz' inequality (we cannot simply use \cref{cor:drop-robust} and Hal\'asz' inequality in both cases, because the conclusions of \cref{lem:drop-robust} and \cref{cor:drop-robust} are very weak when $z$ is very large).

\begin{proof}[Proof of \cref{key-lemma}]
Let\footnote{Recall that $\log s$ denotes the base-2 logarithm of $s$.} $L=\lceil 2(k+r)\log s\rceil$, let $J_*=J_1\cup\dots\cup J_s\su [n]$ and note that $|J_*|=(k+r)s$. Furthermore note that \eqref{eq:key-lemma-equation} trivially holds if $s\le 10^{60}(k+r)^{20}$, so we may assume that $s>10^{60}(k+r)^{20}\ge 10^{60}$ and hence $\log s\le s^{1/20}$. Then we in particular have
\begin{equation}\label{eq:bound-for-L}
L=\lceil 2(k+r)\log s\rceil\le 4(k+r)\cdot \log s\le \frac{s^{1/20}}{250}\cdot s^{1/20}=\frac{s^{1/10}}{250}.
\end{equation}

\textbf{Step 1: The robust-high-rank case.} First, suppose that for every subset $J'\subseteq J_*$ with $|J'|\ge |J_*|-s/2$, we have $\rank A[J'\times [m]]\ge L$.

Then, for every subset $S\su J_*$ of size $|S|<L$, the row span of $A[S\times [m]]$ has dimension less than $L$ and so it can contain at most $|J_*|-s/2$ rows of the matrix $A[J_*\times [m]]$. This means that there are at least $s/2$ indices $h\in J_*$ such that $\on{rank} A[(S\cup \{h\})\times [m]]=\on{rank} A[S\times [m]]+1$. Let $H_S\su J_*$ denote the set of the $\lceil s/2\rceil$ smallest such indices $h\in J_*$.

For our double-counting argument, we will consider sequences $(h_{1},\dots,h_{L})\in J_*^L$ with $h_{t}\in H_{\{h_1,\dots,h_{t-1}\}}$ for $t=1,\dots,L$. For every such sequence, we have $\on{rank} A[\{h_{1},\dots,h_{t}\}\times [m]]=\on{rank} A[\{h_{1},\dots,h_{t-1}\}\times [m]]+1$ for $t=1,\dots,L$ and hence  $\on{rank} A[\{h_{1},\dots,h_{L}\}\times [m]]=L$. Furthermore, note that the total number of such sequences is exactly $\lceil s/2\rceil^L$ (since after choosing $h_1,\dots,h_{t-1}$ we have exactly $|H_{\{h_1,\dots,h_{t-1}\}}|=\lceil s/2\rceil$ choices for $h_t$).

We claim that for every outcome of $\vec{\xi}\in \{-1,1\}^m$ such that $A\vec{\xi}-\vec{v}$ has at most $s/6$ nonzero coordinates, there are at least $(s/3)^L$ sequences $(h_{1},\dots,h_{L})\in J_*^L$ with $h_{t}\in H_{\{h_1,\dots,h_{t-1}\}}$ for $t=1,\dots,L$, such that $(A\vec{\xi}-\vec{v})[\{h_1,\dots,h_z\}]=\vec 0$.  Indeed, choosing $h_{1},\dots,h_{L}$ one at a time, at every step we need to choose $h_{t}\in H_{\{h_1,\dots,h_{t-1}\}}$ with $(A\vec{\xi}-\vec{v})[h_{t}]=0$. At most $s/6$ of the $\lceil s/2\rceil$ elements of $H_{\{h_1,\dots,h_{t-1}\}}$ fail this condition, so we indeed have at least $\lceil s/2\rceil-s/6\ge s/3$ choices at every step. Hence are indeed at least $(s/3)^L$ such sequences $(h_{1},\dots,h_{L})\in J_*^L$.

Thus, the expected number of sequences $(h_{1},\dots,h_{L})\in J_*^L$ with $h_{t}\in H_{\{h_1,\dots,h_{t-1}\}}$ for $t=1,\dots,L$ such that $(A\vec{\xi}-\vec{v})[\{h_1,\dots,h_z\}]=\vec 0$ is at least $\Pr[A\vec{\xi}-\vec{v}\text{ has at most }s/6\text{ nonzero coordinates}]\cdot (s/3)^L$. On the other hand, for every given sequence $(h_{1},\dots,h_{L})\in J_*^L$ with $h_{t}\in H_{\{h_1,\dots,h_{t-1}\}}$ for $t=1,\dots,L$, having $(A\vec{\xi}-\vec{v})[\{h_1,\dots,h_z\}]=\vec 0$ is equivalent to $A[\{h_{1},\dots,h_{L}\}\times [m]]\vec{\xi}=\vec{v}[\{h_{1},\dots,h_{L}\}]$ and by Odlyzko's lemma (\cref{lem:odlyzko}), this happens with probability at most $2^{-L}$ (recalling that $\on{rank} A[\{h_{1},\dots,h_{L}\}\times [m]]=L$). Hence the expected number of such sequences $(h_{1},\dots,h_{L})$ is at most $\lceil s/2\rceil^L\cdot 2^{-L}$. Thus, we obtain
\[\Pr[A\vec{\xi}-\vec{v}\text{ has at most }s/6\text{ nonzero coordinates}]\le \frac{\lceil s/2\rceil^L\cdot 2^{-L}}{(s/3)^L}\le \frac{(5s/9)^L\cdot 2^{-L}}{(s/3)^L}=(5/6)^L\le s^{-(k+r)/2},\]
recalling that $L=\lceil 2(k+r)\log s\rceil$ and using that $(5/6)^4<1/2$. This in particular implies \cref{eq:key-lemma-equation}.

\textbf{Step 2: Covering events for the low-rank case.} Now, we may assume that for some subset $J'\subseteq J_*$ with $|J'|\ge |J_*|-s/2$, we have $\rank A[J'\times [m]]\le L$.

For every outcome of $\vec{\xi}\in \{-1,1\}^m$, let $H_{\vec{\xi}}\su J'$ be the set of indices $h\in J'$ with $(A\vec{\xi}-\vec{v})[h]=0$. Note that whenever $A\vec{\xi}-\vec{v}$ has at most $s/6$ nonzero coordinates, we have $|H_{\vec{\xi}}|\ge |J'|-s/6\ge |J_*|-s/2-s/6=(k+r)s-(2/3)s$, and so by \cref{cor:drop-robust} there is some integer $z\ge r$ such that there are at least $s^z/(12z^2)^z$ sequences $(h_1,\dots,h_z)\in H_{\vec{\xi}}^z$ satisfying \cref{eq:Halasz-good-z-tuple}. Note that for each such sequence, by \cref{eq:Halasz-good-z-tuple} we must in particular have
\[k+z=\on{rank} \begin{pmatrix}A[\{h_{1},\dots,h_{z}\}\times[m]]\\
T
\end{pmatrix}\le k+\on{rank} A[\{h_{1},\dots,h_{z}\}\times[m]]\le k+\on{rank} A[J'\times[m]]\le k+L,\]
and consequently $z\le L$.

Thus, for every outcome of $\vec{\xi}\in \{-1,1\}^m$ such that $A\vec{\xi}-\vec{v}$ has at most $s/6$ nonzero coordinates,  there is some integer $z\in \{r,r+1,\dots,L\}$ such that there are at least $s^z/(12z^2)^z$ sequences $(h_1,\dots,h_z)\in J_*^z$ satisfying \cref{eq:Halasz-good-z-tuple} and $(A\vec{\xi}-\vec{v})[\{h_1,\dots,h_z\}]=\vec 0$. For every $z\in \{r,r+1,\dots,L\}$, let $\mathcal{E}_z$ be the event that $T\vec{\xi}=\vec{w}$ and that there are at least $s^z/(12z^2)^z$ sequences $(h_1,\dots,h_z)\in J_*^z$ satisfying \cref{eq:Halasz-good-z-tuple} and $(A\vec{\xi}-\vec{v})[\{h_1,\dots,h_z\}]=\vec 0$. Then
\begin{equation}\label{eq:sum-union-bound}
    \Pr[T\vec{\xi}=\vec{w}\text{ and }A\vec{\xi}-\vec{v}\text{ has at most }s/6\text{ nonzero coordinates}]\le \sum_{z=r}^{L} \Pr[\mathcal{E}_z]
\end{equation}
and it remains to bound the probability of the events $\mathcal{E}_z$ for $z=r,\dots,L$.

\textbf{Step 3: Double-counting. } Let $z\in \{r,r+1,\dots,L\}$ (and note that then in particular $z\ge r\ge 1$).

For every sequence $(h_1,\dots,h_z)\in J_*^z$ satisfying \cref{eq:Halasz-good-z-tuple}, we have the equivalence
\[T\vec{\xi}=\vec{w}\text{ and }(A\vec{\xi}-\vec{v})[\{h_1,\dots,h_z\}]=\vec 0
\quad \iff \quad 
\begin{pmatrix}A[\{h_{1},\dots,h_{z}\}\times[m]]\\
T
\end{pmatrix}
\vec{\xi}=
\begin{pmatrix}\vec{v}[\{h_{1},\dots,h_{z}\}]\\
\vec{w}
\end{pmatrix}.
\]
So, by the version of Hal\'asz' inequality in \cref{cor:halasz-sub-version}, we have
\[\Pr\bigl[T\vec{\xi}=\vec{w}\text{ and }(A\vec{\xi}-\vec{v})[\{h_1,\dots,h_z\}]=\vec 0\bigr]\le\left(\frac{s}{12z^{2}(z+k)} \right)^{-(k+z)/2}\le (4(k+z)^2)^{k+z}\cdot s^{-(k+z)/2}\label{eq:preserve-halasz-bound}
\]
for every sequence $(h_1,\dots,h_z)\in J_*^z$ satisfying \cref{eq:Halasz-good-z-tuple}.

Hence, the expected number of sequences $(h_1,\dots,h_z)\in J_*^z$ satisfying \cref{eq:Halasz-good-z-tuple} such that $(A\vec{\xi}-\vec{v})[\{h_1,\dots,h_z\}]=\vec 0$ and $T\vec{\xi}=\vec{w}$ hold, is at most $|J_*|^z\cdot (4(k+z)^2)^{k+z}\cdot s^{-(k+z)/2}$. On the other hand, whenever the event $\mathcal{E}_z$ occurs, there are at least $s^z/(12z^2)^z$ such sequences.
We deduce that
\[
    \Pr[\mathcal{E}_z]\le \frac{|J_*|^z\cdot (4(k+z)^2)^{k+z}\cdot s^{-(k+z)/2}}{s^z/(12z^2)^z}\le \frac{((k+r)s)^z (48(k+z)^4)^{k+z}}{s^z\cdot s^{(k+z)/2}}\le \frac{(48(k+z)^5)^{k+z}}{s^{(k+z)/2}}
\]
for every $z\in \{r,\dots,L\}$.

\textbf{Step 4: Summing up.} We now obtain
\[\sum_{z=r}^{L} \Pr[\mathcal{E}_z]\le  \sum_{z=r}^{L}\frac{(48(k+z)^5)^{k+z}}{s^{(k+z)/2}}= \sum_{z=r}^{L}a(z),\quad \text{ where }a(z)=\frac{(48(k+z)^5)^{k+z}}{s^{(k+z)/2}}\text{ for }z=r,\dots,L.\]
For $z=r,\dots,L-1$ we compute (recalling that $k\le \lceil 2(k+r)\log s\rceil=L\le s^{1/10}/250$ by \cref{eq:bound-for-L} and hence $k+z+1\le k+L\le 2L\le s^{1/10}/100$)
\[\frac{a(z+1)}{a(z)}=\frac{48(k+z+1)^5}{s^{1/2}}\cdot \left(\frac{k+z+1}{k+z}\right)^{5(k+z)}
	\le \frac{48(s^{1/10}/100)^5}{s^{1/2}}\cdot e^5\le \frac{48\cdot 3^5}{100^5}\le \frac{1}{2}.\]
So, we have $a(z)\le 2^{-(z-r)}a(r)$ for $z=r,\dots,L$ (so $a(z)$ is bounded by a geometric series with common ratio $1/2$), and consequently
\[\sum_{z=r}^{L} \Pr[\mathcal{E}_z]\le  \sum_{z=r}^{L}a(z)\le \sum_{z=r}^{L}2^{-(z-r)}a(r)\le 2\cdot a(r)=\frac{2\cdot (48(k+r)^5)^{k+r}}{s^{(k+r)/2}}\le \left(\frac{s}{10^{60}(k+r)^{20}}\right)^{-(k+r)/2}.\]
Together with \cref{eq:sum-union-bound}, this implies \cref{eq:key-lemma-equation}.
\end{proof}

Next, we deduce \cref{key-lemma-corollary} from \cref{key-lemma}.

\begin{proof}[Proof of \cref{key-lemma-corollary}]Recalling that $(T,U,A)\in \mathcal{M}^{k,m,n}_2(s)$, let $I_1,\dots,I_s,J_1,\dots,J_s$ be index sets as in the definition of $\mathcal{M}^{k,m,n}_2(s)$ in \cref{def:M-definition} (recall in particular that $I_1,\dots,I_s\su [m]$ are disjoint sets and $J_1,\dots,J_s\su [n]$ are disjoint sets).

Note that $U\in \mathcal{H}^{k\times n}(s)$, due to condition (b) in  \cref{def:M-definition}. Write $\mathcal{W}_{U}\subseteq\RR^{n}$ for the row span of $U$, and note that we have $U_{\vec \xi}'\notin \mathcal H^{(k+1)\times n}(s/6)$ if and only if $A\vec \xi-\vec b$ agrees in at least $n-s/6$ coordinates with a vector $\vec v\in \mc W_U$. Hence
\begin{align*}
 & \Pr\Bigl[T\vec{\xi}=\vec{y}\text{ and }U_{\vec \xi}'\notin \mathcal H^{(k+1)\times n}(s/6)\Bigr]\\
& \quad\!\!=\Pr\Bigl[T\vec{\xi}=\vec{y}\text{ and there is }S\subseteq[n]\text{ with }|S|\ge n-s/6\text{ and }(A\vec{\xi}-\vec{b})[S]=\vec{v}[S]\text{ for some }\vec{v}\in\mathcal{W}_{U}\Bigr]\\
 &\quad\!\!\le\frac{2}{s}\sum_{t=1}^{s}\Pr\Bigl[T\vec{\xi}=\vec{y}\text{ and there is }J_t\su S\subseteq[n]\text{ with }|S|\ge n-s/6\text{ and }(A\vec{\xi}-\vec{b})[S]=\vec{v}[S]\text{ for some }\vec{v}\in\mathcal{W}_{U}\Bigr],
\end{align*}
where the inequality follows by observing that every set $S\subseteq[n]$ of size $|S|\ge n-s/6$ satisfies $J_t\su S$ for at least $s-s/6\ge s/2$ indices $t\in \{1,\dots,s\}$. So it suffices to show that for every fixed $t=1,\dots,s$ we have
\begin{align}
&\Pr\Bigl[T\vec{\xi}=\vec{y}\text{ and there is }J_t\su S\subseteq[n]\text{ with }|S|\ge n-s/6\text{ and }(A\vec{\xi}-\vec{b})[S]=\vec{v}[S]\text{ for some }\vec{v}\in\mathcal{W}_{U}\Bigr]\notag\\
&\qquad \le \left(\frac{s}{10^{60}(k+2)^{20}}\right)^{-(k+2)/2}.
\label{eq:for-each-h}
\end{align}
By condition (b) in \cref{def:M-definition} we can choose a subset $J_{t}'\subseteq J_{t}$ of size $|J_t'|=k$ such that $U[[k]\times J_{t}']$ is nonsingular . Without loss of generality we may assume that $A[J_{t}'\times[m]]$
is the all-zero matrix, since the event on the left-hand side of \cref{eq:for-each-h} does not change if we
add linear combinations of the rows of $U$ to the columns of $A$.
But this assumption means that $(A\vec{\xi}-\vec{b})[J_{t}']=-\vec{b}[J_{t}']$,
and recalling that $U[[k]\times J_{t}']$ is nonsingular, there is only
exactly one vector $\vec{v}_{t}\in\mathcal{W}_{U}$ with $\vec{v}_{t}[J_{t}']=-\vec{b}[J_{t}']$. Hence the event on the left-hand side of \cref{eq:for-each-h} can only happen with $\vec{v}=\vec{v}_{t}$, and we obtain
\begin{align*}
&\Pr\Bigl[T\vec{\xi}=\vec{y}\text{ and there is }J_t\su S\subseteq[n]\text{ with }|S|\ge n-s/6\text{ and }(A\vec{\xi}-\vec{b})[S]=\vec{v}[S]\text{ for some }\vec{v}\in\mathcal{W}_{U}\Bigr]\\
 & \qquad=\Pr\Bigl[T\vec{\xi}=\vec{y}\text{ and there is }J_t\su S\subseteq[n]\text{ with }|S|\ge n-s/6\text{ such that }(A\vec{\xi}-\vec{b})[S]=\vec{v}_{t}[S]\Bigr]\\
 & \qquad\le \Pr[T\vec{\xi}=\vec{y}\text{ and }A\vec{\xi}-\vec{b}-\vec{v}_{t}\text{ has at least }n-s/6\text{ zero coordinates}]\\
 & \qquad=\Pr[T\vec{\xi}=\vec{y}\text{ and }A\vec{\xi}-\vec{b}-\vec{v}_{t}\text{ has at most }s/6\text{ nonzero coordinates}].
\end{align*}
The desired bound \cref{eq:for-each-h} now follows from \cref{key-lemma} with $r=2$ (note that conditions (i) and (ii) in \cref{key-lemma} hold by conditions (a) and (c) in \cref{def:M-definition}, respectively).
\end{proof}

We end this section by deducing \cref{lem:baby-key} (the ``Hamming norm'' anticoncentration inequality stated in the outline in \cref{sec:outline}) from \cref{key-lemma}.

\begin{proof}[Proof of \cref{lem:baby-key}]
    Let $C_r=(10r)^{30r}$ and $c_r=1/(6r)$. As in the statement of the lemma, let $A\in \RR ^{m\times n}$ be a matrix which has rank at least $r$ after deletion of any $t$ rows and $t$ columns.
    
    We claim that for $s=\lceil t/r\rceil$ we can find disjoint subsets $I_1,\dots,I_s\su [n]$ and disjoint subsets $J_1,\dots,J_s\su [m]$ of size $|I_1|=\dots=|I_s|=|J_1|=\dots=|J_s|=r$ such that $\rank A[J_\ell \times I_\ell]=r$ for $\ell=1,\dots,s$. Indeed, we can find such subsets greedily: after having found $I_1,\dots,I_\ell\su [n]$ and $J_1,\dots,J_\ell\su [m]$ for some $\ell<t/r$, we can delete all columns with indices in $I_1\cup \dots \cup I_\ell$ and all rows with indices in $J_1\cup \dots \cup J_\ell$ from $A$ and the resulting matrix must have rank at least $r$, so there must be subsets $I_{\ell+1}\su  [n]\setminus (I_1\cup \dots \cup I_\ell)$ and $J_{\ell+1}\su  [m]\setminus (J_1\cup \dots \cup J_\ell)$ with $\rank A[J_\ell \times I_\ell]=r$ and $|J_{\ell+1}|=|I_{\ell+1}|=r$.

    Having found such subsets $I_1,\dots,I_s\su [n]$ and $J_1,\dots,J_s\su [m]$, we can now apply \cref{key-lemma} with $k=0$ and $T\in \RR ^{0\times n}$ being the empty matrix and $\vec w\in \RR ^{0}$ being the empty vector. Note that condition (i) is the vacuously true, and condition (ii) is true because for every $\ell=1,\dots,s$ the only $(T[[k]\times I_t],0)$-perturbation of the matrix $A[J_t\times I_t]$ is the matrix $A[J_t\times I_t]$ itself, which has rank $r$. Thus, for a sequence $\vec{\xi}=(\xi_{1},\dots,\xi_{n})\in\{-1,1\}^{n}$ of independent Rademacher random variables, and for any vector $\vec v\in \RR^m$, we have
\begin{align*}
\Pr[A\vec{\xi}\text{ differs from }\vec{v}\text{ in fewer than }t/(6r)\text{ coordinates}]&\le \Pr[A\vec{\xi}-\vec{v}\text{ has at most }s/6\text{ nonzero coordinates}]\\
&\le\left(\frac{s}{10^{60}r^{20}}\right)^{-r/2}\le \left(\frac{t}{10^{60}r^{21}}\right)^{-r/2}\le (10r)^{30r}\cdot t^{-r/2}
\end{align*}
(noting that $T\vec{\xi}=\vec{w}$ holds vacuously for all $\vec{\xi}\in\{-1,1\}^{n}$).
\end{proof}

\section{Splitting the index set}\label{sec:splitting}

In this section, we prove \cref{lem:matrix-splitting}, splitting our index set $[n]$ into disjoint subsets $I$ and $J$ such that the conditions in \cref{def:M-definition} are satisfied. This amounts to finding disjoint subsets $I_1,\dots,I_s,J_1,\dots,J_s\su [n]$ satisfying certain robust-rank-2 conditions. Most of the work is in finding a single pair of subsets $I_1,J_1$ satisfying the desired property; we will then be able to find our subsets in a greedy fashion, in a similar way to the proof of \cref{lem-halasz-equivalent}(ii).

\begin{proof}[Proof of \cref{lem:matrix-splitting}]
Let $0\le k\le n$, $s\ge 4k+8$, $M\in\mathcal{H}^{k\times n}(s)$ and $A\in \RR^{n\times n}$ be as in the statement of the lemma. Recall that $A$ is a symmetric matrix and that we are assuming $\on{rank} A^*[S\times S]\ge 2$ for any subset $S\subseteq [n]$ of size $|S|\ge n- s$ and any matrix $A^*\in \RR^{n\times n}$ that agrees with some $(M,M)$-perturbation of $A$ in all off-diagonal entries. 

\textbf{Step 1: Setup.} For $\ell=\lfloor s/(4k+8)\rfloor$, we wish to find disjoint subsets $I_1,\dots, I_\ell, J_1,\dots,J_\ell\su [n]$ of size $|I_1|=\dots=|I_\ell|=|J_1|=\dots=|J_\ell|=k+2$ such that the following three conditions hold for $t=1,\dots, \ell$:
\begin{compactitem}
    \item[(a)] The submatrix $M[[k]\times I_t]$ has rank $k$.
    \item[(b)] The submatrix $M[[k]\times J_t]$ has rank $k$. 
    \item[(c)] Every $(M[[k]\times I_t],M[[k]\times J_t])$-perturbation of the matrix $A[J_t\times I_t]$ has rank at least $2$.
\end{compactitem}
This suffices, since we can then take $I=I_1\cup \dots\cup I_\ell$ and $J=[n]\setminus I$ to obtain a partition $[n]=I\cup J$ satisfying the conclusion of the lemma (the conditions (a)--(c) in Definition \ref{def:M-definition} precisely correspond to conditions (a) to (c) above, and we have $k\le |I|\le |J|$ since $|I|=\ell(k+2)$ and $|J|\ge |J_1\cup \dots\cup J_\ell|=\ell(k+2)$).

To find such disjoint subsets $I_1,\dots, I_\ell, J_1,\dots,J_\ell\su [n]$, it suffices to show that for any subset $R\subseteq [n]$ with $|R|> n-s/2$, we can find disjoint subsets $I_1,J_1\subseteq R$ of size $|I_1|=|J_1|=k+2$ such that conditions (a)--(c) above hold for $t=1$. Indeed, then we can then construct the desired subsets $I_1,\dots, I_\ell, J_1,\dots,J_\ell\su [n]$ greedily (at every step choosing $I_t,J_t\subseteq [n]\setminus (I_1\cup\dots\cup I_{t-1}\cup J_1\cup\dots\cup J_{t-1}$). Showing this will be our objective for the rest of the proof, so let us fix a subset $R\subseteq [n]$ with $|R|> n-s/2$.

Since $M\in\mathcal{H}^{k\times n}(s)$, we can find disjoint subsets $I_1',J_1'\su R$ of size $|I_1'|=|J_1'|=k$ with $\rank M[[k]\times I_1']=\rank M[[k]\times J_1']=k$. We will obtain $I_1,J_1$ by augmenting $I_1'$ and $J_1'$ with two additional indices each (we will find suitable indices for this using our assumption on $A$).

\textbf{Step 2: Augmenting the index sets.} Since the $k\times k$ matrices $M[[k]\times I_1']$ and $M[[k]\times J_1']$ are both nonsingular, we can add linear combinations of the rows of $M$ to the rows and columns of $A$ to obtain an $(M,M)$-perturbation $A'$ of $A$ such that all entries of $A'[J_1'\times [n]]$ and all entries of $A'[[n]\times I_1']$ are zero. Define  $R'=R\setminus (I_1'\cup J_1')$ and note that $|R'|\ge n-s/2-2k\ge n-s+4$. By our assumption on $A$, the submatrix $A'[R'\times R']$ must have some nonzero off-diagonal entry $a'_{j,i}$ with distinct $i,j\in R'$ (if all off-diagonal entries were zero, by modifying the diagonal entries of $A'$ we would be able to find a matrix $A^*\in \RR^{n\times n}$ agreeing with the $(M,M)$-perturbation $A'$ of $A$ in all off-diagonal entries such that $\rank A^*[R'\times R']=0$, contradicting our assumption on $A$).

Now, let $A^*$ be the matrix obtained from $A'$ by adjusting the diagonal entries $a'_{h,h}$ for $h\in R'\setminus \{i,j\}$ in such a way that each $2\times 2$ submatrix of the form $A^*[\{j,h\}\times \{i,h\}]$ (for any $h\in R'\setminus \{i,j\}$) is singular. Since $A^*$ agrees with the $(M,M)$-perturbation $A'$ of $A$ in all off-diagonal entries, we have $\rank A^*[S\times S]\ge 2$ for every subset $S\su [n]$ of size $|S|\ge n-s$, and in particular for the subset $S=R'\setminus \{i,j\}$. Hence we obtain $\rank A^*[(R'\setminus \{i\})\times  (R'\setminus \{j\})]\ge 2$, and also note that the $(j,i)$-entry of $A^*$ is $a^*_{j,i}=a'_{j,i}\ne 0$.

Now, if a matrix has rank at least 2, then for any nonzero row (respectively, column), we can find a second linearly independent row (respectively, column). Applying this fact to the matrix $A^*[(R'\setminus \{i\})\times  (R'\setminus \{j\})]$ and the row with index $j$, we can find an index $j'\in R'\setminus \{i,j\}$ such that $\rank A^*[\{j,j'\}\times  (R'\setminus \{j\})]=2$ (i.e., such that the row with index $j'$ is linearly independent from the row with index $j$). Applying the fact again to the matrix $A^*[\{j,j'\}\times  (R'\setminus \{j\})]$ and the column with index $i$, we can find an index $i'\in R'\setminus \{j,i\}$ such that $\rank A^*[\{j,j'\}\times  \{i,i'\}]=2$. Now, we must have $j'\ne i'$ (since $\rank A^*[\{j,h\}\times \{i,h\}]\le 1$ for all $h\in R'\setminus \{j,i\}$). So, $\{j,j'\}$ and $\{i,i'\}$ are disjoint subsets of $R'=R\setminus (I_1'\cup J_1')$, and we have $\rank A'[\{j,j'\}\times  \{i,i'\}]=\rank A^*[\{j,j'\}\times  \{i,i'\}]=2$. Defining $I_1=I_1'\cup \{i,i'\}$ and $J_1=J_1'\cup \{i,i'\}$, we obtain disjoint subsets of $R$ of size $k+2$.

\textbf{Step 3: Proving the full-rank condition.} We now need to show that the subsets $I_1,J_1\su R\su [n]$ satisfy conditions (a)--(c) above for $t=1$. This is equivalent to showing that the $(2k+2)\times (2k+2)$ matrix
\[
\begin{pmatrix}A[J_{1}\times I_{1}] & M[[k]\times J_{1}]^{\transpose}\\
M[[k]\times I_{1}] & 0
\end{pmatrix}
\]
has full rank $2k+2$. But note that the above matrix can be reduced to 
\[
\begin{pmatrix}0 & 0 & M[[k]\times J_{1}']^{\transpose}\\
0 & A'[\{j,j'\}\times \{i,i'\}] & M[[k]\times \{j,j'\}]^{\transpose}\\
M[[k]\times I_{1}'] & M[[k]\times \{i,i'\}] & 0
\end{pmatrix}
\]
by elementary row and column operations (precisely the row and column operations that
were used to obtain $A'$ from $A$ in the previous step of the proof).
Recalling that the matrices $M[[k]\times J_{1}']$, $M[[k]\times I_{1}']$
and $A'[\{j,j'\}\times \{i,i'\}]$ are nonsingular, the desired
result follows.
\end{proof}

\section{Proof of the recursive bound in \texorpdfstring{\cref{thm:recursion}}{Theorem \ref{thm:recursion}}}\label{sec:recursion}

In this section, we finally prove \cref{thm:recursion}, using the results established in the previous three sections.

\begin{proof}[Proof of \cref{thm:recursion}] As in the statement of the theorem, let $k\ge 0$ be an integer, $s> 0$ be a real number and define $s_*=s/(k+2)^{500}$. To show the desired bound on $f(k,s)$, we need to show that for any quadruple $(n,Q,M,\vec w)$ as in \cref{def:function-f}, we have
\begin{equation}\label{eq:recursion-to-show}
\Pr[Q(\vec{\xi})=0\text{ and }M\vec{\xi}=\vec w]\le \max\left\{s_*^{-(k+1)/2}, \quad s_*^{-(k+2)/2}+s_*^{-k/4}\cdot f(k+1,s_*)^{1/2}\right\},
\end{equation}
where the probability is taken with respect to a sequence of independent Rademacher random variables $\vec{\xi}\in\{-1,1\}^{n}$. This is clearly true if $s_*\le 1$, so we may assume without loss of generality that $s_*>1$ and hence $s>(k+2)^{500}\ge 2^{500}$.

So let $n$ be let a positive integer, $Q\in\RR[x_{1},\dots,x_{n}]$ be a quadratic polynomial, let $M\in\mathcal{H}^{k\times n}(s)$ and $\vec w\in \RR^k$.  Let us write the quadratic part of $Q(\vec{x})$ as $\vec{x}^{\transpose}A\vec{x}$ for a symmetric matrix $A\in \RR^{n\times n}$, and assume that for every subset $S\subseteq[n]$ with $|S|\ge n-s$, and every $(M,M)$-perturbation $A'$ of $A$, the submatrix $A'[S\times S]$ has at least one nonzero entry outside the diagonal (this is condition ($*$) in \cref{def:function-f}).

As the first step of the proof, we treat the case where the matrix $A$ does not robustly have rank at least $2$. In the remaining steps of the proof we can then assume that the matrix $A$ does have rank at least $2$ robustly.

\textbf{Step 1: The low-rank case.} Suppose that there is a matrix $B \in \RR^{n\times n}$ which can be obtained from an $(M,M)$-perturbation of $A$ by changing its diagonal entries, and a set $S$ of size $|S|\ge n-2s/3$, such that $\on{rank} B[S\times S]\le 2$. Note that then $B^\transpose$ can also be obtained from an $(M,M)$-perturbation of $A^\transpose=A$ by changing its diagonal entries, and satisfies $\on{rank} B^\transpose[S\times S]\le 2$. Consider the symmetric matrix $A^*=\frac{1}{2}(B+B^\transpose)$; note that $A^*$ can be obtained from an $(M,M)$-perturbation of $\frac{1}{2}(A+A)=A$ by changing its diagonal entries, and we have $\on{rank} A^*[S\times S]\le 4$.

Let $Q^*\in \RR[x_1,\dots,x_n]$ be a quadratic polynomial with quadratic part $\vec x^\transpose A^*\vec x$ such that $Q(\vec \xi)=Q^*(\vec \xi)$ for all $\vec{\xi}\in \{-1,1\}^n$ with $M\vec \xi=\vec w$ (such a polynomial $Q^*$ exists by \cref{lem:perturbation}).

Now, let $T=[n]\setminus S$, and note that it suffices to prove the probability bound in \cref{eq:recursion-to-show} conditioned on every possible outcome for $\vec \xi[T]$. For any given outcome of $\vec \xi[T]$, we can write $Q^*(\vec \xi)$ as $Q^*_{\vec \xi[T]}(\vec \xi[S])$, for some quadratic polynomial $Q^*_{\vec \xi[T]}$ in the variables $\vec{\xi}[i]$ for $i\in S$, with quadratic part $\vec x[S]^\transpose A^*[S\times S]\vec x[S]$ (and whose linear and constant coefficients depend on $\vec \xi[T]$). Simply put, this polynomial is obtained from $Q^*$ by plugging in the given values of $\vec{\xi}[i]$ for $i\in T$. Recall that we always have $Q(\vec \xi)=Q^*(\vec \xi)=Q^*_{\vec \xi[T]}(\vec \xi[S])$ (for any $\vec \xi\in \{-1,1\}^n$ with $M\vec \xi=\vec w$).

Now, we claim that for any subset $S'\su S$ of size $|S'|\ge |S|-s/3\ge n-s$, the matrix $A^*[S'\times S']$ cannot be a $(M[[k]\times S'],M[[k]\times S'])$-perturbation of the zero matrix in $\RR^{S'\times S'}$. Indeed, if $A^*[S'\times S']$ was a $(M[[k]\times S'],M[[k]\times S'])$-perturbation of the zero matrix, then from the matrix $A[S'\times S']$ one could obtain the zero matrix by taking an $(M[[k]\times S'],M[[k]\times S'])$-perturbation and changing its diagonal entries. But this means that there is some $(M,M)$-perturbation $A'$ of $A$ such $A'[S'\times S']$ agrees with the zero matrix in $\RR^{S'\times S'}$ in all off-diagonal entries. This means that $A'[S'\times S']$ does not have any nonzero entries outside its diagonal, contradicting our assumption on $A$ made above (coming from condition ($*$) \cref{def:function-f}). This means that for any subset $S'\su S$ of size $|S'|\ge |S|-s/3$, the matrix $A^*[S'\times S']$ is indeed not a $(M[[k]\times S'],M[[k]\times S'])$-perturbation of the zero matrix.

Furthermore $M[[k]\times S]\in\mathcal{H}^{k\times S}(s/3)$ and $\on{rank} A^*[S\times S]\le 4$, so we can apply \cref{prop:low-rank-case} with $r=5$ to the matrix $M[[k]\times S]$, the vector $\vec{w}-M[[k]\times T]\vec \xi[T]\in \RR^k$, and the quadratic polynomial $Q^*_{\vec \xi[T]}$ in the variables $\vec{\xi}[i]$ for $i\in S$. The conclusion of the proposition then gives
\begin{align*}
\Pr\Bigl[Q(\vec \xi)=0\text{ and }M\vec \xi=\vec w\,\Big|\,\vec \xi[T]\Bigr]&=\Pr\Bigl[Q_{\vec{\xi}[T]}^*(\vec \xi[S])=0\text{ and }M[[k]\times S]\vec \xi[S]=\vec{w}-M[[k]\times T]\vec \xi[T]\,\Big|\,\vec \xi[T]\Bigr]\\
&\le \left(\frac{\lfloor s/3\rfloor}{2^{75}(k+5)^2}\right)^{-(k+1)/2}\le s_*^{-(k+1)/2}.
\end{align*}
for every possible outcome of $\vec \xi[T]$. Hence $\Pr[Q(\vec \xi)=0\text{ and }M\vec \xi=\vec w]\le s_*^{-(k+1)/2}$, which in particular proves \cref{eq:recursion-to-show}.

\textbf{Step 2: Decoupling.} From now on, we may assume that $\on{rank} B[S\times S]\ge 3$ for any matrix $B \in \RR^{n\times n}$ which can be obtained from an $(M,M)$-perturbation of $A$ by changing its diagonal entries, and any set $S$ of size $|S|\ge n-2s/3$. This means in particular that the assumption in \cref{lem:matrix-splitting} is satisfied for $s/2\ge 4k+8$, and therefore \cref{lem:matrix-splitting} gives us a partition $[n]=I\cup J$ with $|I|\le s/2$, such that
\[\Bigl(M[[k]\times I],\,M[[k]\times J],\,A[J\times I]\Bigr)\in\mathcal{M}_2^{k,I,J}\Bigl(\lfloor s/(8k+16)\rfloor \Bigr).\]
Recalling the second part of \cref{rem:M-implies-H}, this implies 
\begin{equation}
    \Bigl(M[[k]\times I],\,M[[k]\times J],\,-2A[J\times I]\Bigr)\in \mathcal{M}_2^{k,I,J}\Bigl(\lfloor s/(8k+16)\rfloor \Bigr)\su \mathcal{M}_2^{k,I,J}\Bigl(6\cdot \lfloor s/(48k+96)\rfloor\Bigr)
    \label{eq:splitting-condition}
\end{equation}
(which in particular means that $6\cdot \lfloor s/(48k+96)\rfloor\le |I|\le |J|$).

Let $\pvec{\xi}[I]$ be an independent copy of $\vec{\xi}[I]$ (i.e., consider independent random variables $\pvec{\xi}[i]\in \{-1,1\}$ for $i\in I$, independent from all entries of $\vec \xi$). Now, let us extend $\pvec{\xi}[I]$ to a vector $\pvec{\xi}\in \RR^n$ by defining $\pvec{\xi}[j]=\vec{\xi}[j]$ for all $j\in J$. Note that then $\pvec{\xi}$ consists of the data $(\pvec{\xi}[I],\vec{\xi}[J])$, while $\vec \xi$ consists of the data $(\vec{\xi}[I],\vec{\xi}[J])$. By decoupling (\cref{lem:decoupling})
we have
\begin{align}
&\Pr[Q(\vec{\xi})=0\text{ and }M\vec{\xi}=\vec{w}]^{2}\notag\\
&\quad\quad \le\Pr[Q(\vec{\xi})=Q(\pvec{\xi})=0\text{ and }M\vec{\xi}=M\pvec{\xi}=\vec{w}]\notag\\
 &\quad\quad =\Pr[Q(\vec{\xi})=0\text{ and }Q(\vec{\xi})-Q(\pvec{\xi})=0\text{ and }M\vec{\xi}=\vec{w}\text{ and }M(\vec{\xi}-\pvec{\xi})=\vec 0]\notag\\
 &\quad\quad =\Pr\Bigl[Q(\vec{\xi})=0\text{ and }Q(\vec{\xi})-Q(\pvec{\xi})=0\text{ and }M\vec{\xi}=\vec{w}\text{ and }M[[k]\times I](\vec{\xi}[I]-\pvec{\xi}[I])=\vec 0\Bigr].\label{eq:recursion-decoupling}
\end{align}
Note that the event $M[[k]\times I](\vec{\xi}[I]-\pvec{\xi}[I])=\vec 0$ depends only on the outcomes of $\vec{\xi}[I]$ and $\pvec{\xi}[I]$ (and not on the outcome of $\vec{\xi}[J]$). We will later bound the probability of this event using \cref{cor:halasz-sub-version}. For fixed outcomes of $\vec{\xi}[I]$ and $\pvec{\xi}[I]$, we can express the conditions $Q(\vec{\xi})=0$ and $Q(\vec{\xi})-Q(\pvec{\xi})=0$ and $M\vec{\xi}=\vec{w}$ in terms of $\vec \xi [J]$; namely we can write them in the form $Q_{\vec \xi[I]}(\vec{\xi}[J])=0$ and $M_{\vec \xi[I],\pvec{\xi}[I]}\vec\xi[J]=\vec{w}_{\vec \xi[I],\pvec{\xi}[I]}$ for some quadratic polynomial $Q_{\vec \xi[I]}$ and some matrix $M_{\vec \xi[I],\pvec{\xi}[I]}$ (which depend on the outcomes of $\vec{\xi}[I]$ and $\pvec{\xi}[I]$).

More specifically, for a given outcome of $\vec{\xi}[I]$,  the event
\begin{equation}\label{eq:recursion-rewrite-1}
\{ Q(\vec{\xi})=0\}\quad \text{can be written as}\quad \{Q_{\vec \xi[I]}(\vec{\xi}[J])=0\},
\end{equation}
where $Q_{\vec \xi[I]}$ is the quadratic polynomial in the variables $x_j$ for $j\in J$ obtained from $Q$ by replacing each variable $x_i$ for $i\in I=[n]\setminus J$ by the given value of $\vec\xi[i]$. Then, by definition, we have $Q_{\vec \xi[I]}(\vec{\xi}[J])=Q(\vec\xi)$ for any outcome of $\vec{\xi}[J]$. Also note that the quadratic part of the polynomial $Q_{\vec \xi[I]}$ is given by $\vec{x}[J]^{\transpose}A[J\times J]\vec{x}[J]$.

Furthermore, for given outcomes of $\vec{\xi}[I]$ and $\pvec{\xi}[I]$, the event
\begin{equation}\label{eq:recursion-rewrite-2}
\{ M\vec{\xi}=\vec{w}\text{ and }Q(\vec{\xi})-Q(\pvec{\xi})=0\}\quad \text{can be written as}\quad \{M_{\vec \xi[I],\pvec{\xi}[I]}\vec{\xi}[J]=\vec{w}_{\vec \xi[I],\pvec{\xi}[I]}\}
\end{equation}
where the matrix $M_{\vec \xi[I],\pvec{\xi}[I]}\in \RR^{(k+1)\times J}$ and the vector $\vec{w}_{\vec \xi[I],\pvec{\xi}[I]}\in \RR^{k+1}$ are defined as follows: When plugging in the given values for $\vec{\xi}[i]$ for $i\in I$, we can interpret $M\vec{\xi}=\vec{w}$ as a system of linear equations in the variables $\vec{\xi}[j]$ for $j\in J=[n]\setminus I$. This system of linear equations has the form $M[[k]\times J]\vec{\xi}[J]=\vec w-M[[k]\times I](\vec{\xi}[I]-\pvec{\xi}[I])$. Furthermore, when plugging the given values for $\vec{\xi}[i]$  and $\pvec{\xi}[i]$ for $i\in I$ into $Q(\vec{\xi})-Q(\pvec{\xi})=0$, we obtain another linear equation in the variables $\vec{\xi}[j]$ for $j\in J$ (indeed, the quadratic terms in the variables $\vec{\xi}[j]$ for $j\in J$ cancel out in the difference $Q(\vec{\xi})-Q(\pvec{\xi})$). The coefficient vector of this linear equation is precisely $2A[J\times I](\vec{\xi}[I]-\pvec{\xi}[I])$. Appending this additional linear equation to our previous system of linear equations in the variables $\vec{\xi}[j]$ for $j\in J$, we obtain a system of $k+1$ equations that we can express in the form $M_{\vec \xi[I],\pvec{\xi}[I]}\vec{\xi}[J]=\vec{w}_{\vec \xi[I],\pvec{\xi}[I]}$ for a matrix $M_{\vec \xi[I],\pvec{\xi}[I]}\in \RR^{(k+1)\times J}$ and a vector $\vec{w}_{\vec \xi[I],\pvec{\xi}[I]}\in \RR^{k+1}$. Note that the first $k$ rows of the matrix $M_{\vec \xi[I],\pvec{\xi}[I]}\in \RR^{(k+1)\times J}$ are given by $M[[k]\times J]$, and the last row is given by (the transpose of) the vector $2A[J\times I](\vec{\xi}[I]-\pvec{\xi}[I])$.

For given outcomes of $\vec{\xi}[I]$ and $\pvec{\xi}[I]$, we would like to bound the probability of having $Q_{\vec \xi[I]}(\vec{\xi}[J])=0$ and $M_{\vec \xi[I],\pvec{\xi}[I]}\vec\xi[J]=\vec{w}_{\vec \xi[I],\pvec{\xi}[I]}$ (subject to the randomness of $\vec\xi[J]$). If $M_{\vec \xi[I],\pvec{\xi}[I]}\in \mathcal{H}^{(k+1)\times J}(s')$ for some suitable $s'$, then we will be able to bound this probability by $f(k+1,s')$; see \cref{def:function-f}. In the next step, we first handle the case where $M_{\vec \xi[I],\pvec{\xi}[I]}\not\in \mathcal{H}^{(k+1)\times J}(s')$.

\textbf{Step 3: Failure of the Hal\'asz condition.} We define $s'=\lfloor s/(48k+96)\rfloor$, then $k\le 6s'\le |I|$ (recalling \cref{eq:splitting-condition} and $s>(k+2)^{500}$). Our goal in this step is to bound the probability that the outcomes of $\vec{\xi}[I]$ and $\pvec{\xi}[I]$ are such that $M[[k]\times I](\vec{\xi}[I]-\pvec{\xi}[I])=\vec 0$ and $M_{\vec \xi[I],\pvec{\xi}[I]}\not\in \mathcal{H}^{(k+1)\times J}(s')$. To do so, let us condition on an arbitrary outcome of $\vec\xi[I]$. Let us apply \cref{key-lemma-corollary} with $0\le k\le 6s'\le |I|\le |J|$, the random vector $\pvec{\xi}[I]$, the matrices $T=M[[k]\times I]$, $U=M[[k]\times J]$ and $-2A[J\times I]$ (noting that then $(T,U,-2A[J\times I])\in \mathcal{M}_2^{k,I,J}(6s')$ by \cref{eq:splitting-condition}), as well as the vectors $\vec y=M[[k]\times I]\vec \xi[I]$ and $\vec b=-2A[J\times I]\vec \xi[I]$. Note that then the random matrix $U'_{\vec\xi}$ in \cref{key-lemma-corollary} is precisely the matrix $M_{\vec \xi[I],\pvec{\xi}[I]}$ (namely, the matrix obtained from $U=M[[k]\times J]$ by appending the vector $-2A[J\times I]\pvec\xi[I]-\vec b=2A[J\times I](\vec\xi[I]-\pvec{\xi}[I])$ as an additional row). So, we obtain 
\begin{equation}\label{eq:recursion-bound-Halasz-failure}
\Pr\left[M[[k]\times I](\vec{\xi}[I]-\pvec{\xi}[I])=\vec 0\text{ and }M_{\vec \xi[I],\pvec{\xi}[I]}\notin \mathcal H^{(k+1)\times J}(s') \,\middle|\,\vec{\xi}[I]\right]\le \left(\frac{6s'}{10^{61}(k+2)^{20}}\right)^{-(k+2)/2}\le s_*^{-(k+2)/2}
\end{equation}
for any outcome of $\vec\xi[I]$ (here, the probability is with respect to the randomness of $\pvec\xi[I]$).  In the last step, we used that $s'=\lfloor s/(48k+96)\rfloor$ and therefore $s'/(10^{61}(k+2)^{20})\ge s/(10^{63}(k+2)^{21})\ge s/(k+2)^{273}>s_*$ (since $k+2\ge 2$).

\textbf{Step 4: Bounding the main term.} Recall from Step 2 that $\on{rank} B[S\times S]\ge 3$ for any matrix $B \in \RR^{n\times n}$ which can be obtained from an $(M,M)$-perturbation of $A$ by changing its diagonal entries, and any subset $S\su J$ of size $|S|\ge |J|-s'$ (here we are using that $|J|-s'
\ge n-2s/3$ as $s'\le s/6$ and $|J|=n-|I|\ge n-s/2$). Recalling that the matrix $M_{\vec \xi[I],\pvec{\xi}[I]}$ is obtained from $M[[k]\times J]$ by adding one row, this implies that $\on{rank} B[S\times S]\ge 1$ for any matrix $B \in \RR^{J\times J}$ obtained from an $(M_{\vec \xi[I],\pvec{\xi}[I]},M_{\vec \xi[I],\pvec{\xi}[I]})$-perturbation of $A[J\times J]$ by changing its diagonal entries, and any subset $S\su J$ of size $|S|\ge |J|-s'$. Thus, $A'[S\times S]$ must have at least one nonzero entry outside the diagonal for every subset $S\subseteq J$ of size $|S|\ge |J|-s'$ and every $(M_{\vec \xi[I],\pvec{\xi}[I]},M_{\vec \xi[I],\pvec{\xi}[I]})$-perturbation $A'$ of $A[J\times J]$.

This means that whenever we have $M_{\vec \xi[I],\pvec{\xi}[I]}\in \mathcal{H}^{(k+1)\times J}(s')$, the
quadruple $(|J|,Q_{\vec \xi[I]},M_{\vec \xi[I],\pvec{\xi}[I]},\vec{w}_{\vec \xi[I],\pvec{\xi}[I]})$ satisfies the conditions in \cref{def:function-f}, with parameters $k+1$ and $s'$ (recall that $Q_{\vec \xi[I]}$ has quadratic part $\vec{x}[J]^{\transpose}A[J\times J]\vec{x}[J]$). Hence, for any outcomes of $\vec{\xi}[I]$ and $\pvec{\xi}[I]$ such that $M_{\vec \xi[I],\pvec{\xi}[I]}\in \mathcal{H}^{(k+1)\times J}(s')$, we have the conditional probability bound
\begin{align}
    &\Pr\left[Q(\vec{\xi})=0\text{ and }M\vec{\xi}=\vec{w}\text{ and }Q(\vec{\xi})-Q(\pvec{\xi})=0\,\middle|\,\vec{\xi}[I],\pvec{\xi}[I]\right]\notag\\
    &\quad\quad=\Pr\left[Q_{\vec \xi[I]}(\vec{\xi}[J])=0\text{ and }M_{\vec \xi[I],\pvec{\xi}[I]}\vec{\xi}[J]=\vec{w}_{\vec \xi[I],\pvec{\xi}[I]}\,\middle|\,\vec{\xi}[I],\pvec{\xi}[I]\right]\le f(k+1,s')\le f(k+1,s_*)\label{eq:recursion-bound-main-term}
\end{align}
(for the first step, recall \cref{eq:recursion-rewrite-1} and \cref{eq:recursion-rewrite-2}, and for the last step note that $s'=\lfloor s/(48k+96)\rfloor\ge s/(k+2)^{500}=s_*$).

\textbf{Step 5: Concluding.} For ease of notation, let us abbreviate $\mathcal H^{(k+1)\times J}(s')$ by $\mathcal{H}$. By \cref{eq:recursion-decoupling} we have
\begin{align*}
&\Pr[Q(\vec{\xi})=0\text{ and }M\vec{\xi}=\vec{w}]^{2}\\
 &\quad \le \Pr\Bigl[Q(\vec{\xi})=0\text{ and }Q(\vec{\xi})-Q(\pvec{\xi})=0\text{ and }M\vec{\xi}=\vec{w}\text{ and }M[[k]\times I](\vec{\xi}[I]-\pvec{\xi}[I])=\vec 0\Bigr]\\
 &\quad \le\Pr\Bigl[Q(\vec{\xi})=0\text{ and }Q(\vec{\xi})-Q(\pvec{\xi})=0\text{ and }M\vec{\xi}=\vec{w}\text{ and }M[[k]\times I](\vec{\xi}[I]-\pvec{\xi}[I])=\vec 0\text{ and }M_{\vec \xi[I],\pvec{\xi}[I]}\in \mathcal H\Bigr]\\
 &\qquad\qquad\qquad+[Q(\vec{\xi})=0\text{ and }M\vec{\xi}=\vec{w}\text{ and }M[[k]\times I](\vec{\xi}[I]-\pvec{\xi}[I])=\vec 0\text{ and }M_{\vec \xi[I],\pvec{\xi}[I]}\notin \mathcal H]\\
 &\quad \le\Pr\left[M[[k]\times I](\vec{\xi}[I]-\pvec{\xi}[I])=\vec 0\right]\!\sup_{\substack{\vec{\xi}[I],\pvec{\xi}[I]\\M_{\vec \xi[I],\pvec{\xi}[I]}\in \mathcal H}}\!\!\Pr\left[Q(\vec{\xi})=0\text{ and }M\vec{\xi}=\vec{w}\text{ and }Q(\vec{\xi})-Q(\pvec{\xi})=0\,\middle|\,\vec{\xi}[I],\pvec{\xi}[I]\right]\\
 &\qquad\qquad\qquad+\Pr\left[Q(\vec{\xi})=0\text{ and }M\vec{\xi}=\vec{w}\right]\sup_{\vec{\xi}[I]}\Pr\left[M[[k]\times I](\vec{\xi}[I]-\pvec{\xi}[I])=\vec 0\text{ and }M_{\vec \xi[I],\pvec{\xi}[I]}\notin \mathcal H \,\middle|\,\vec{\xi}[I]\right].
\end{align*}
Note that $M[[k]\times I]\in \mathcal H^{k\times I}(6s')$ by \cref{eq:splitting-condition} and \cref{rem:M-implies-H}. Therefore, by \cref{cor:halasz-sub-version} (applied with $\vec w=M[[k]\times I]\pvec{\xi}[I]$ for any fixed outcome of $\pvec{\xi}[I]$) we have $\Pr\big[M[[k]\times I](\vec{\xi}[I]-\pvec{\xi}[I])=\vec 0\big]\le (6s'/k)^{-k/2}\le s_*^{-k/2}$ if $k\ge 1$ (recalling that $s'=\lfloor s/(48k+96)\rfloor$ and $s_*=s/(k+2)^{500}$). If $k=0$, then we trivially have $\Pr\big[M[[k]\times I](\vec{\xi}[I]-\pvec{\xi}[I])=\vec 0\big]\le s_*^{-k/2}$, so this inequality holds in either case. Plugging this observation, as well as \cref{eq:recursion-bound-main-term} and \cref{eq:recursion-bound-Halasz-failure} into the above chain of inequalities, we obtain
\[\Pr[Q(\vec{\xi})=0\text{ and }M\vec{\xi}=\vec{w}]^{2}\le s_*^{-k/2}\cdot f(k+1,s_*)+\Pr[Q(\vec{\xi})=0\text{ and }M\vec{\xi}=\vec{w}]\cdot s_*^{-(k+2)/2}\]
Applying \cref{lem:simple-inequality}, this yields
\[\Pr[Q(\vec{\xi})=0\text{ and }M\vec{\xi}=\vec{w}]\le s_{*}^{-(k+2)/2}+s_{*}^{-k/4}\cdot f(k+1,s_*)^{1/2},\]
implying the desired bound in \cref{eq:recursion-to-show}.
\end{proof}

\section{Deducing a non-recursive bound}\label{sec:calculation-deduction}

In this section, we prove \cref{thm:main}. First, from the recursive bound for $f(k,s)$ in \cref{thm:recursion}, one can deduce the following non-recursive bound.

\begin{corollary}\label{cor:induction}
For any integers $0\le k\le \ell$ and any real $s>0$ we have
\begin{equation}\label{eq:complicated-inductive-bound}
f(k,s)\le (s_{k,\ell})^{-\ell/2^{\ell-k+1}}\prod_{j=k}^{\ell-1}(s_{k,j})^{-j/2^{j-k+2}}+\sum_{i=k}^{\ell-1}(s_{k,i})^{-(i+2)/2^{i-k+1}}\prod_{j=k}^{i-1}(s_{k,j})^{-j/2^{j-k+2}},
\end{equation}
where for $i=k,\dots,\ell$ we define
\[
s_{k,i}=\frac{s}{(i+2)^{500(i-k+1)}}.
\]
\end{corollary}

In order to analyse the bound on the left-hand side of \cref{eq:complicated-inductive-bound}, we make the following simple observation.

\begin{lemma}\label{lem:sum-identity}
For any integers $0\le k\le i$, we have
\[
\sum_{j=k}^{i-1}\frac{j}{2^{j-k+2}}=\frac{k+1}{2}-\frac{i+1}{2^{i-k+1}}.
\]
\end{lemma}
\begin{proof}
The identity can easily be shown by induction on $i-k$. The base case $i-k=0$ (i.e., $i=k$) is trivial, as both sides are zero. For the induction step, with $i-k\ge 1$, we compute
\[\sum_{j=k}^{i-1}\frac{j}{2^{j-k+2}}=\sum_{j=k}^{i-2}\frac{j}{2^{j-k+2}}+\frac{i-1}{2^{i-k+1}}=\frac{k+1}{2}-\frac{i}{2^{i-k}}+\frac{i-1}{2^{i-k+1}}=\frac{k+1}{2}-\frac{i+1}{2^{i-k+1}}.\qedhere\]
\end{proof}

Using the above observation, let us now deduce \cref{cor:induction} from \cref{thm:recursion} by induction on $\ell-k$ and some tedious but straightforward calculations.

\begin{proof}[Proof of \cref{cor:induction}]
We prove the desired bound by induction on $\ell-k$. If $\ell-k=0$ (i.e., if $\ell=k$), we need to show that $f(k,s)\le (s_{k,k})^{-k/2}$. If $k=0$, this is trivially true as $f(k,s)\le 1=(s_{0,0})^{-0/2}$. To check $f(k,s)\le (s_{k,k})^{-k/2}$ for $k\ge 1$, recall that for every quadruple $(n,Q,M,\vec w)$ in the supremum in the definition of $f(k,s)$ in \cref{def:function-f}, we have $M\in \mathcal{H}^{k\times n}(s)$ and therefore by \cref{cor:halasz-sub-version}
\[\Pr[Q(\vec{\xi})=0\text{ and }M\vec{\xi}=\vec w]\le \Pr[M\vec{\xi}=\vec w]\le (s/k)^{-k/2}\le \left(\frac{s}{(k+2)^{500}}\right)^{-k/2}=(s_{k,k})^{-k/2}.\]
This shows $f(k,s)\le (s_{k,k})^{-k/2}$, as desired.

Let us now assume that $\ell-k\ge 1$ and that we already proved the desired bound for all smaller values of $\ell-k$. Note that now by \cref{thm:recursion} we have
\[f(k,s)\le \max\left\{s_*^{-(k+1)/2}, \quad s_*^{-(k+2)/2}+s_*^{-k/4}\cdot f(k+1,s_*)^{1/2}\right\},\]
with $s_*=s/(k+2)^{500}$. So it suffices to show that both terms in this maximum are bounded by the left-hand side of \cref{eq:complicated-inductive-bound}.

First, note that $s_*=s/(k+2)^{500}\ge s_{k,i}$ for all $i=k,\dots,\ell$. This means that in the case $s_*\le 1$, we also have $s_{k,i}\le 1$ for all $i=k,\dots,\ell$ and hence \cref{eq:complicated-inductive-bound} follows trivially from $f(k,s)\le 1$. We may therefore assume that $s_*> 1$.

For the first term in the maximum above, note that by \cref{lem:sum-identity} applied to $i=\ell$, we have
\[\frac{\ell}{2^{\ell-k+1}}+\sum_{j=k}^{\ell-1}\frac{j}{2^{j-k+2}}=\frac{\ell}{2^{\ell-k+1}}+\frac{k+1}{2}-\frac{\ell+1}{2^{\ell-k+1}}=\frac{k+1}{2}-\frac{1}{2^{\ell-k+1}}<\frac{k+1}{2}\]
Hence
\[s_*^{-(k+1)/2}\le (s_*)^{-\ell/2^{\ell-k+1}}\prod_{j=k}^{\ell-1}(s_*)^{-j/2^{j-k+2}}\le (s_{k,\ell})^{-\ell/2^{\ell-k+1}}\prod_{j=k}^{\ell-1}(s_{k,j})^{-j/2^{j-k+2}},\]
which in particular shows that $s_*^{-(k+1)/2}$ is bounded by the left-hand side of \cref{eq:complicated-inductive-bound}.

To bound the second term, first note that by the inductive assumption we have
\[f(k+1,s_*)\le(s'_{k+1,\ell})^{-\ell/2^{\ell-k}}\prod_{j=k+1}^{\ell-1}(s'_{k+1,j})^{-j/2^{j-k+1}}+\sum_{i=k+1}^{\ell-1}(s'_{k+1,i})^{-(i+2)/2^{i-k}}\prod_{j=k+1}^{i-1}(s'_{k+1,j})^{-j/2^{j-k+1}},
\]
defining
\[
s'_{k+1,i}=\frac{s_*}{(i+2)^{500(i-k)}}=\frac{s}{(i+2)^{500(i-k)}\cdot (k+2)^{500}}.
\]
for $i=k+1,\dots,\ell$. Note that we have $s'_{k+1,i}\ge s_{k,i}$ for all $i=k+1,\dots,\ell$, implying
\[f(k+1,s_*)\le (s_{k,\ell})^{-\ell/2^{\ell-k}}\prod_{j=k+1}^{\ell-1}(s_{k,j})^{-j/2^{j-k+1}}+\sum_{i=k+1}^{\ell-1}(s_{k,i})^{-(i+2)/2^{i-k}}\prod_{j=k+1}^{i-1}(s_{k,j})^{-j/2^{j-k+1}}.\]
As $\sqrt{x+y}\le \sqrt{x}+\sqrt{y}$ for all $x,y\ge 0$, this implies
\[f(k+1,s_*)^{1/2}\le (s_{k,\ell})^{-\ell/2^{\ell-k+1}}\prod_{j=k+1}^{\ell-1}(s_{k,j})^{-j/2^{j-k+2}}+\sum_{i=k+1}^{\ell-1}(s_{k,i})^{-(i+2)/2^{i-k+1}}\prod_{j=k+1}^{i-1}(s_{k,j})^{-j/2^{j-k+2}}.\]
Also using that $s_{k,k}=s_*$, we now obtain
\begin{align*}
&s_*^{-(k+2)/2}+s_*^{-k/4}\cdot f(k+1,s_*)^{1/2}\\
&\qquad\le (s_{k,k})^{-(k+2)/2}+(s_{k,\ell})^{-\ell/2^{\ell-k+1}}(s_{k,k})^{-k/4}\prod_{j=k+1}^{\ell-1}(s_{k,j})^{-j/2^{j-k+2}}\\
&\qquad\qquad\qquad\qquad\qquad+\sum_{i=k+1}^{\ell-1}(s_{k,i})^{-(i+2)/2^{i-k+1}} (s_{k,k})^{-k/4}\prod_{j=k+1}^{i-1}(s_{k,j})^{-j/2^{j-k+2}}\\
&\qquad=(s_{k,k})^{-(k+2)/2}+(s_{k,\ell})^{-\ell/2^{\ell-k+1}}\prod_{j=k}^{\ell-1}(s_{k,j})^{-j/2^{j-k+2}}+\sum_{i=k+1}^{\ell-1}(s_{k,i})^{-(i+2)/2^{i-k+1}}\prod_{j=k}^{i-1}(s_{k,j})^{-j/2^{j-k+2}}\\
&\qquad=(s_{k,\ell})^{-\ell/2^{\ell-k+1}}\prod_{j=k}^{\ell-1}(s_{k,j})^{-j/2^{j-k+2}}+\sum_{i=k}^{\ell-1}(s_{k,i})^{-(i+2)/2^{i-k+1}}\prod_{j=k}^{i-1}(s_{k,j})^{-j/2^{j-k+2}}.
\end{align*}
Hence the second term in the maximum above also satisfies the desired bound, completing the proof.
\end{proof}

Finally, we show how to deduce \cref{thm:main} from the $k=0$ case of \cref{cor:induction}.
\begin{proof}[Proof of \cref{thm:main}]
Let $Q\in \RR[x_1,\dots,x_n]$, $A\in \RR^{n\times n}$ and $s$ be as in \cref{thm:main}. We may assume that $s\ge 4$ (otherwise the bound in \cref{thm:main} holds trivially as long as $C'\ge 2$). Let $k=0$, let $M\in \mathcal{H}^{0\times n}(s)$ be the empty $0\times n$ matrix, and let $\vec w\in \RR^0$ be the empty vector. Note that the quadruple $(n,Q,M,\vec{w})$ satisfies  condition ($*$) in \cref{def:function-f}, because the only $(M,M)$-perturbation of $A$ is the matrix $A'=A$, and by the assumption in \cref{thm:main}, for every subset $S\su [n]$ of size $|S|\ge n-s$, the submatrix $A[S\times S]$ has at least one nonzero entry outside its diagonal. Thus, for a sequence $\vec \xi\in \{-1,1\}^n$ of independent Rademacher random variables, we have
\[\Pr[Q(\vec \xi)=0]=\Pr[Q(\vec{\xi})=0\text{ and }M\vec{\xi}=\vec w]\le f(0,s)\]
and hence by \cref{cor:induction}
\[\Pr[Q(\vec \xi)=0]\le (s_{0,\ell})^{-\ell/2^{\ell+1}}\prod_{j=0}^{\ell-1}(s_{0,j})^{-j/2^{j+2}}+\sum_{i=0}^{\ell-1}(s_{0,i})^{-(i+2)/2^{i+1}}\prod_{j=0}^{i-1}(s_{0,j})^{-j/2^{j+2}}\]
for every integer $\ell\ge 0$ (with $s_{0,i}=s/(i+2)^{500(i+1)}$ for $i=0,\dots,\ell$ as defined in \cref{cor:induction}). Then for every integer $\ell\ge 0$ we obtain
\begin{align*}
\Pr[Q(\vec \xi)=0]&\le \prod_{i=0}^{\ell}(i+2)^{500(i+1)(i+2)/2^{i+1}}\cdot \left(s^{-\ell/2^{\ell+1}}\prod_{j=0}^{\ell-1}s^{-j/2^{j+2}}+\sum_{i=0}^{\ell-1}s^{-(i+2)/2^{i+1}}\prod_{j=0}^{i-1}s^{-j/2^{j+2}}\right)\\
&= \exp\left(500\sum_{i=0}^{\ell}\frac{(i+2)^2\ln (i+2)}{2^{i+1}}\right)\cdot \left(s^{-1/2+1/2^{\ell+1}}+\sum_{i=0}^{\ell-1}s^{-1/2-1/2^{i+1}}\right)
\end{align*}
using \cref{lem:sum-identity} for $k=0$. Noting that the series $\sum_{i=0}^{\infty} (i+2)^2\ln (i+2)/2^{i+1}=\sum_{i=2}^{\infty} i^2\ln i/2^{i-1}$ converges, for all integers $\ell\ge 0$ we obtain
\[\Pr[Q(\vec \xi)=0]\le C_1\cdot \left(s^{-1/2+1/2^{\ell+1}}+\sum_{i=0}^{\ell-1}s^{-1/2-1/2^{i+1}}\right)=C_1\cdot s^{-1/2}\cdot \left(s^{1/2^{\ell+1}}+\sum_{i=0}^{\ell-1}s^{-2^{\ell-i}/2^{\ell+1}}\right)\]
for some absolute constant $C_1\ge 1$. Plugging in $\ell=\lfloor\log\log s\rfloor-1$, we have
\[2\le s^{1/2^{\ell+1}}\le 4\]
and therefore
\[\Pr[Q(\vec \xi)=0]\le C_1\cdot s^{-1/2}\cdot \left(4+\sum_{i=0}^{\ell-1}(1/2)^{2^{\ell-i}}\right)\le \frac{C_1}{\sqrt{s}}\cdot \left(4+\sum_{i=0}^{\ell-1}(1/2)^{\ell-i}\right)\le \frac{5C_1}{\sqrt{s}}.\]
Setting $C'=5C_1$, this gives the desired bound.
\end{proof}

\section{Deduction of main results}\label{sec:arbitrary-distributions}
In this section, we deduce our main theorem (\cref{thm:main-shiny}) and its generalisation to arbitrary distributions (\cref{thm:general-distributions}) from the slightly more technical statement of \cref{thm:main}. Note that \cref{thm:general-distributions} directly implies \cref{thm:main-shiny}, taking $\xi_{1},\dots,\xi_{n}$ to be independent Rademacher random variables, and taking $\delta=1/2$. So, we will just prove \cref{thm:general-distributions} (however, we remark that the proof gets easier and certain steps can be skipped if one is only interested in \cref{thm:main-shiny}).

First, to apply \cref{thm:main} in the general setting of \cref{thm:general-distributions}, we use that any discrete random variable can be expressed in a way such that after some further conditioning one basically obtains a  Rademacher random variable. 
The following lemma strengthens an observation made by Meka,
Nguyen and Vu~\cite{MNV16} for a similar purpose.
\begin{lemma}\label{lem-representation}
For every discrete random variable $\zeta\in \RR$, we can find a representation of the form
\[\zeta=\alpha+\xi \beta\]
for a Rademacher random variable $\xi\in\{-1,1\}$ and a discrete random vector $(\alpha,\beta)\in \RR^2$ which is independent of $\xi$.
Moreover, this representation can be chosen such that the distribution of $(\alpha,\beta)$ satisfies the following two conditions:
\begin{compactenum}
\item [(a)] There is at most one real number $a\in \RR$ with $\Pr[(\alpha,\beta)=(a,0)]>0$
\item [(b)] If there is some value $z\in \RR$ with $\Pr[\zeta=z]>1/2$, then we always have $\alpha+\beta=z$ (for any outcome of the random vector $(\alpha,\beta)\in \RR^2$).
\end{compactenum} 
\end{lemma}
\begin{proof}
    If the random variable $\zeta$ is constant, i.e., if $\Pr[\zeta=z]=1$ for some $z\in \RR$, we can define the random vector $(\alpha,\beta)\in \RR^2$ to always take the constant value $(z,0)$. So let us from now on assume that $\zeta$ is not constant.
    
    \textbf{Case 1: There is a majority outcome.} First, we consider the case where $\Pr[\zeta=z]\ge 1/2$ for some $z\in \RR$. Let $\rho=\Pr[\zeta\ne z]\le 1/2$. Then, we can take the random vector $(\alpha,\beta)\in \RR^2$ to be equal to $(z,0)$ with probability $1-2\rho$ and equal to $((z+Y)/2,(z-Y)/2)$ with probability $2\rho$, where $Y$ is a sample from the distribution of $\zeta$ conditioned on the event $\zeta\ne z$. Note that we then always have $\alpha+\beta=z$ (as $\alpha+\beta=z+0=z$ or $\alpha+\beta=(z+Y)/2+(z-Y)/2=z$). Also note for all $a\in \RR$ with $a\ne z$ we have $\Pr[(\alpha,\beta)=(a,0)]=0$, since $Y$ never takes the value $z$ and so we can only have $\beta=0$ when $\alpha=z$. Now, taking $\xi\in\{-1,1\}$ to be a Rademacher random variable that is independent of the random vector $(\alpha,\beta)$, the expression $\alpha+\xi \beta$ evaluates to $z\pm 0=z$ with probability $1-2\rho$, to $(z+Y)/2+(z-Y)/2=z$ with probability $\rho$ and to $(z+Y)/2-(z-Y)/2=Y$ with probability $\rho$. Hence the distribution of $\alpha+\xi \beta$ coincides with the distribution of $\zeta$. We can therefore find a coupling of $\zeta$ with $(\alpha,\beta)$ and $\xi$ such that $(\alpha,\beta)\in \RR^2$ and $\xi\in\{-1,1\}$ are independent and $\zeta=\alpha+\xi \beta$.
    
    \textbf{Case 2: There is no majority outcome.} Now, let us consider the case where $\Pr[\zeta=z]< 1/2$ for all $z\in \RR$. Let $x$ be a median of $\zeta$ (meaning that $\Pr[\zeta\ge x]\ge 1/2$ and $\Pr[\zeta\le x]\ge 1/2$; for example we can take $x=\sup \{x\in \RR:\Pr[\zeta\ge x]\ge 1/2\}$).
    As $\Pr[\zeta= x]<1/2$, we can see that $0<\Pr[\zeta< x]\le 1/2$ and $0<\Pr[\zeta> x]\le 1/2$.

    Now, let $\rho_1=\Pr[\zeta< x]$ and $\rho_2=\Pr[\zeta> x]$. Let us assume without loss of generality that $\rho_1\ge \rho_2$ (the case $\rho_2\ge \rho_1$ is analogous). Let $Y_1$ (respectively, $Y_2$) be a sample from the distribution of $\zeta$ conditioned on the event $\zeta< x$ (respectively, the event $\zeta>x$). We can now take the random vector $(\alpha,\beta)\in \RR^2$ to be equal to $(x,0)$ with probability $1-2\rho_1$, equal to $((x+Y_1)/2,(x-Y_1)/2)$ with probability $2\rho_1-2\rho_2$, and equal to $((Y_2+Y_1)/2,(Y_2-Y_1)/2)$ with probability $2\rho_2$. Note that we can only have $\beta=0$ when $\alpha=x$, since we always have $Y_1<x$ and $Y_2>x$. Now, taking $\xi\in\{-1,1\}$ to be a Rademacher random variable that is independent of the random vector $(\alpha,\beta)\in \RR^2$, the expression $\alpha+\xi \beta$ evaluates to $x\pm 0=x$ with probability $1-2\rho_1$, to $(x+Y_1)/2+(x-Y_1)/2=x$ with probability $\rho_1-\rho_2$, to $(x+Y_1)/2-(x-Y_1)/2=Y_1$ with probability $\rho_1-\rho_2$, to $(Y_2+Y_1)/2-(Y_2-Y_1)/2=Y_1$ with probability $\rho_2$, and to $(Y_2+Y_1)/2+(Y_2-Y_1)/2=Y_2$ with probability $\rho_2$. So all in all, $\alpha+\xi \beta$ evaluates to $x$ with probability $1-\rho_1-\rho_2=\Pr[\zeta= x]$, to $Y_1$ with probability $\rho_1$ and to $Y_2$ with probability $\rho_2$. Hence the distribution of $\alpha+\xi \beta$ agrees with the distribution of $\zeta$, and we can again find the desired coupling such that $\zeta=\alpha+\xi \beta$.
\end{proof}

In the setting of \cref{lem-representation}, note that if there is a real number $a\in \RR$ with $\Pr[(\alpha,\beta)=(a,0)]>0$, then we have $\Pr[\beta=0]=\Pr[(\alpha,\beta)=(a,0)]\le \Pr[\zeta=a]$ (here we used (a)). On the other hand, if there is no such $a\in \RR$, we clearly have $\Pr[\beta=0]=0$. Thus, in any case we can conclude that
\begin{equation}\label{eq:prob-beta-zero}
    \Pr[\beta=0] \le \sup_{z\in \RR}\Pr[\zeta=z].
\end{equation}

We will also need a generalisation of the Erd\H os--Littlewood--Offord theorem to arbitrary discrete random variables (i.e., not just for Rademacher random variables). The following theorem follows directly from the result of \cite{Kol58}, and can also be deduced from the ordinary Erd\H os--Littlewood--Offord theorem.

\begin{theorem}\label{thm:general-ELO}
    Fix $\delta>0$, and let $X_1,\dots,X_t\in \RR$ be independent discrete random variables satisfying $\sup_{z\in \RR}\Pr[X_i=z]\le 1-\delta$ for all $i=1,\dots,t$. Then we have
    \[
\sup_{z\in \RR}\Pr[X_1+\dots+X_t=z]\le\frac{C_\delta'}{\sqrt{t}},
\]
for some constant $C_\delta'$ only depending on $\delta$.
\end{theorem}

Let us first prove \cref{thm:general-distributions} under the additional assumption that all point probabilities of the variables $\zeta_1,\dots,\zeta_n$ are at most $1-\delta$. This does not require the full strength of \cref{lem-representation} (we will not need (b) in \cref{lem-representation}).

\begin{proposition}\label{prop:general-distributions-assumption}
The statement of \cref{thm:general-distributions} holds under the additional assumption that we have $\sup_{z\in \RR}\Pr[\zeta_i=z]\le 1-\delta$ for all $i=1,\dots,n$.
\end{proposition}

\begin{proof}
    For ease of notation, we write $\vec \zeta=(\zeta_1,\dots,\zeta_n)$. Let $C'\ge 1$ be an absolute constant such that \cref{thm:main} holds, and let $C_\delta=\max(8C'/\delta,2C'_\delta)$ for the constant $C'_\delta$ in \cref{thm:general-ELO}. Note that the claimed probability bound is trivially true if $\sqrt{m}\le C_\delta$, so we may assume that $m\ge (C_\delta)^2\ge 64/\delta^2$.
    
    Recall that $Q\in \RR[x_1,\dots,x_n]$ is a quadratic polynomial, and that we are assuming that for any fixing box $R_1\times \dots\times R_n$ of $Q$ (where $R_1,\dots,R_n$ are nonempty subsets of the supports of $\zeta_1,\dots,\zeta_n$, respectively) there are at least $m$ indices $i\in \{1,\dots,n\}$ with $\Pr[\zeta_i\in R_i]\le 1-\delta$. Note that this assumption is not affected by changing the constant term of $Q$. It therefore suffices to prove
    \begin{equation}\label{eq:general-dist-prop-to-show}
    \Pr[Q(\vec \zeta)=0]\le\frac{C_\delta}{\sqrt{m}}.
    \end{equation}
    Indeed, for any $z\in \RR$, applying this inequality to the polynomial $Q-z$ gives $\Pr[Q(\vec \zeta)=z]\le C_\delta/\sqrt{m}$, as claimed. To prove \cref{eq:general-dist-prop-to-show}, we distinguish two cases.

    \textbf{Case 1: Quadratic anticoncentration.} First, consider the case that for some $\ell\ge m/4$ there exist distinct indices $i_1,\dots,i_\ell,j_1,\dots,j_\ell\in [n]$ such that for each $h=1,\dots,\ell$ the coefficient of $x_{i_h}x_{j_h}$ in the quadratic polynomial $Q$ is nonzero (and let $c_h\ne 0$ denote this coefficient). In this case, our goal is to apply \cref{thm:main}.

    For each $i=1,\dots,n$, let us represent the random variable $\zeta_i$ as $\zeta_i=\alpha_i+\xi_i \beta_i$, for a Rademacher random variable $\xi_i\in\{-1,1\}$ and a random vector $(\alpha_i,\beta_i)\in \RR^2$, as in \cref{lem-representation}, in such a way that $\xi_1,\dots,\xi_n$ and $(\alpha_1,\beta_1),\dots,(\alpha_n,\beta_n)$ are all mutually independent. Note that by \cref{eq:prob-beta-zero} and our assumption in the proposition, we have
    \[\Pr[\beta_i=0] \le \sup_{z\in \RR}\Pr[\zeta_i=z]\le 1-\delta\]
    for $i=1,\dots,n$. Thus, for each $h=1,\dots,\ell$, we obtain
    \[\Pr[\beta_{i_h}\beta_{j_h}\ne 0] =\Pr[\beta_{i_h}\ne 0]\cdot \Pr[\beta_{i_h}\ne 0]\ge \delta^2,\]
    and furthermore the events $\beta_{i_h}\beta_{j_h}\ne 0$ are independent for all $h=1,\dots,\ell$. Thus, by a Chernoff bound (see for example \cite[Theorem~2.1]{JLR}), we have
    \begin{equation}\label{eq:few-activated}
        \Pr\bigl[\text{fewer than }\delta^2\ell/2\text{ indices }h\in \{1,\dots,\ell\}\text{ satisfy }\beta_{i_h}\beta_{j_h}\ne 0\bigr]\le e^{-\delta^2\ell/8}\le e^{-\delta^2m/32}\le\frac{32}{\delta^2m}\le\frac{4/\delta}{\sqrt{m}},
    \end{equation}
    recalling $\ell\ge m/4$ and $m\ge 64/\delta^2$ (i.e., $
    \delta \sqrt{m}\ge 8$). Conditioning on any outcomes of $(\alpha_1,\beta_1),\dots,(\alpha_n,\beta_n)$, such that $\beta_{i_h}\beta_{j_h}\ne 0$ for at least $\delta^2\ell/2$ indices $h\in \{1,\dots,\ell\}$, we can interpret $Q(\zeta_1,\dots,\zeta_n)$ as a quadratic polynomial in the independent Rademacher random variables $\xi_1,\dots,\xi_n$ by plugging in $\zeta_i=\alpha_i+\xi_i \beta_i$ for $i=1,\dots,n$. Note that for $h=1,\dots,\ell$ the coefficient of $\xi_{i_h}\xi_{j_h}$ in this quadratic polynomial is $c_h\beta_{i_h}\beta_{j_h}$. Hence there are at least $\delta^2\ell/2\ge \delta^2m/8$ indices $h\in \{1,\dots,\ell\}$ such that the coefficient of $\xi_{i_h}\xi_{j_h}$ in this quadratic polynomial is nonzero, and so this quadratic polynomial satisfies the condition in \cref{thm:main} for any positive integer $s<\delta^2m/8$ (indeed, for at least $\delta^2m/8$ indices $h\in \{1,\dots,\ell\}$ the $(i_h,j_h)$-entry of the matrix $A$ appearing in this condition is nonzero, and for every subset $S\su [n]$ of size $|S|\ge n-s>n-\delta^2m/8$, the submatrix $A[S\times S]$ must contain one of these nonzero entries). Taking $s=\lceil\delta^2m/8\rceil-1\ge \delta^2m/8-1\ge \delta^2m/16$ (recalling that $m\ge 64/\delta^2$), by \cref{thm:main} we obtain
    \[\Pr\bigl[Q(\vec \zeta)=0\,\big|\,(\alpha_1,\beta_1),\dots,(\alpha_n,\beta_n)\bigr]\le \frac{C'}{\sqrt{\delta^2m/16}}=\frac{4C'/\delta}{\sqrt{m}}\]
    when conditioning on any outcomes of $(\alpha_1,\beta_1),\dots,(\alpha_n,\beta_n)$ such that $\beta_{i_h}\beta_{j_h}\ne 0$ for at least $\delta^2\ell/2$ indices $h\in \{1,\dots,\ell\}$. All in all, together with \cref{eq:few-activated}, this yields
     \[\Pr[Q(\vec \zeta)=0]\le \frac{4C'/\delta}{\sqrt{m}}+\frac{4/\delta}{\sqrt{m}}\le \frac{8C'/\delta}{\sqrt{m}}\le \frac{C_\delta}{\sqrt{m}},\]
     showing the desired bound \cref{eq:general-dist-prop-to-show}.

\textbf{Case 2: Linear anticoncentration.} From now on we may assume that the condition in Case 1 does not hold. For the maximum possible $\ell$, consider distinct indices $i_1,\dots,i_\ell,j_1,\dots,j_\ell\in [n]$ such that for $h=1,\dots,\ell$ the coefficient of $x_{i_h}x_{j_h}$ in $Q$ is nonzero. By our assumption for this case, we have $\ell<m/4$. Let $J=\{i_1,\dots,i_\ell,j_1,\dots,j_\ell\}$ and $I=[n]\setminus J$. Note  that then, by the maximality of $\ell$, for any distinct $i,i'\in I$, the coefficient of $x_ix_{i'}$ in $Q$ is zero.

Our plan is to condition on an arbitrary outcome of $\vec \zeta[J]$, and to apply \cref{thm:general-ELO} in the resulting conditional probability space (only using the randomness of $\zeta_i$ for $i\in I$).

For any outcome of $\vec \zeta[J]$, we can interpret $Q(\vec \zeta)$ as a polynomial in the remaining variables $\zeta_i$ for $i\in I$ (with coefficients depending on $\vec\zeta[J]$). Writing $Q_{\vec \zeta[J]}$ for this polynomial, we always have $Q(\vec \zeta)=Q_{\vec \zeta[J]}(\vec \zeta[I])$. For any distinct $i,i'\in I$, the coefficient of $\zeta_i\zeta_{i'}$ in $Q_{\vec \zeta[J]}$ is zero, so $Q_{\vec \zeta[J]}$ can be written as a sum $\sum_{i\in I} P^{(i)}_{\vec \zeta[J]}(\zeta_i)$, where for each $i\in I$, the summand $P^{(i)}_{\vec \zeta[J]}(\zeta_i)$ is a quadratic polynomial in the single variable $\zeta_i$ (with coefficients depending on $\vec \zeta[J]$). We now have
\[\Pr\Bigl[Q(\vec \zeta)=0\,\Big|\,\vec \zeta[J]\Bigr]=\Pr\Bigl[Q_{\vec \zeta[J]}(\vec \zeta[I])=0\,\Big|\,\vec \zeta[J]\Bigr]=\Pr\Biggl[\,\sum_{i\in I} P^{(i)}_{\vec \zeta[J]}(\zeta_i)=0\,\Bigg|\,\vec \zeta[J]\Biggr]\]
for every outcome of $\vec \zeta[J]$.

For any outcome of $\vec \zeta[J]$, let $T_{\vec \zeta[J]}\su I=[n]\setminus J$ be the set of indices $i\in I$ with $\sup_{z\in\RR}\Pr\bigl[P^{(i)}_{\vec \zeta[J]}(\zeta_i)=z\,\big|\,\vec \zeta[J]\bigr]\le 1-\delta$. We claim that we must always have $|T_{\vec \zeta[J]}|\ge m/2$, due to the assumption in \cref{thm:general-distributions} concerning fixing boxes. 
Indeed, suppose for the purpose of contradiction that there is an outcome $\vec w=(w_j)_{j\in J}\in \RR^J$ of $\vec \zeta[J]$ such that $|T_{\vec w}|< m/2$. Then, we have $|I\setminus T_{\vec w}|> n-m$ (recalling that $|I|=n-|J|=n-2\ell>n-m/2$).
We can construct a fixing box $R_1\times \dots\times R_n$ for the polynomial $Q$ as follows:
\begin{compactitem}
    \item For $j\in J$, take $R_j=\{w_j\}$;
    \item For $t\in T_{\vec w}$, take $R_t=\{y_t\}$ for some arbitrary element $y_t$ of the support of $\zeta_t$;
    \item For $i\in I\setminus T_{\vec w}$, take $z_i\in \RR$ such that $\Pr\bigl[P^{(i)}_{\vec w}(\zeta_i)=z_i\bigr]>1-\delta$ (such a value $z_i$ exists by the definition of $T_{\vec w}$), and let $R_i$ be the set of all $y$ in the support of $\zeta_i$ such that $P_{\vec w}^{(i)}(y)=z_i$ (i.e., $R_i=(P_{\vec w}^{(i)})^{-1}(z_i)\cap \operatorname{supp}(\zeta_i)$).
\end{compactitem}
Note that $Q$ is constant on $R_1\times \dots\times R_n$: indeed, for any $(\zeta_1,\dots,\zeta_n)\in R_1\times\dots\times R_n$ we have \[Q(\zeta_1,\dots,\zeta_n)=Q_{\vec w}(\vec \zeta[I])=\sum_{i\in I} P_{\vec w}^{(i)}(\zeta_i)=\sum_{t\in T_{\vec w}}P_{\vec w}^{(t)}(y_t)+\sum_{i\in I\setminus T_{\vec w}}z_i.\] So, $R_1\times \dots\times R_n$ is indeed a fixing box of $Q$. On the other hand we have $\Pr[\zeta_i\in R_i]=\Pr\bigl[P^{(i)}_{\vec \zeta[J]}(\zeta_i)=z_i\bigr]> 1-\delta$ for all $i\in I\setminus T_{\vec w}$. As $|I\setminus T_{\vec w}|> n-m$, this means that there are strictly fewer than $m$ indices $i\in \{1,\dots,n\}$ with $\Pr[\zeta_i\in R_i]\le 1-\delta$, contradicting our assumption.

We have established that $|T_{\vec \zeta[J]}|\ge m/2$ for any outcome of $\vec \zeta[J]$. 
For any outcomes of $\vec \zeta[J]$ and $\vec\zeta[I\setminus T_{\vec \zeta[J]}]$, we now have
\begin{align*}\Pr\Biggl[\,\sum_{i\in I} P_{\vec \zeta[J]}^{(i)}(\zeta_i)=0\,\Bigg|\,\vec \zeta[J],\vec\zeta[I\setminus T_{\vec \zeta[J]}]\Biggr]&=\Pr\Biggl[\,\sum_{i\in T_{\vec \zeta[J]}} P_{\vec \zeta[J]}^{(i)}(\zeta_i)=-\!\!\!\sum_{i\in I\setminus T_{\vec \zeta[J]}} P_{\vec \zeta[J]}^{(i)}(\zeta_i)\,\Bigg|\,\vec \zeta[J],\vec\zeta[I\setminus T_{\vec \zeta[J]}]\Biggr]\\
&\le \frac{C_\delta'}{\sqrt{|T_{\vec \zeta[J]}|}}\le  \frac{C_\delta'}{\sqrt{m/2}} \le\frac{2C_\delta'}{\sqrt{m}}\end{align*}
by \cref{thm:general-ELO}. To be precise, for any possible outcomes of $\vec \zeta[J]$ and $\vec\zeta[I\setminus T_{\vec \zeta[J]}]$, we apply \cref{thm:general-ELO} in the conditional probability space given these outcomes, with the random variables $X_i=P_{\vec \zeta[J]}^{(i)}(\zeta_i)$ for $i\in T_{\vec \zeta[J]}$, noting that by the definition of $T_{\vec \zeta[J]}$ we then have
\[\sup_{z\in \RR}\Pr\Bigl[X_i=z\,\Big|\,\vec \zeta[J],\vec\zeta[I\setminus T_{\vec \zeta[J]}]\Bigr]\le 1-\delta\] for each $i\in T_{\vec \zeta[J]}$. So overall we obtain
\[\Pr\Bigl[Q(\vec \zeta)=0\,\Big|\,\vec \zeta[J]\Bigr]=\Pr\Biggl[\,\sum_{i\in I} P_{\vec \zeta[J]}^{(i)}(\zeta_i)=0\,\Bigg|\,\vec \zeta[J]\Biggr]\le \frac{2C_\delta'}{\sqrt{m}}\le\frac{C_\delta}{\sqrt{m}}\]
for any outcome of $\vec \xi[J]$. This implies the desired bound in \cref{eq:general-dist-prop-to-show}.
\end{proof}

Finally, we deduce the full statement of \cref{thm:general-distributions} from \cref{prop:general-distributions-assumption}.
\begin{proof}[Proof of \cref{thm:general-distributions}]
    We may assume without loss of generality that $0<\delta<1/2$. Let   $C_\delta$ be a constant such that \cref{prop:general-distributions-assumption} holds (i.e., such that the statement in \cref{thm:general-distributions} holds under the additional assumption that $\sup_{z\in \RR}\Pr[\zeta_i=z]\le 1-\delta$ for all $i=1,\dots,n$). Let us also write $\vec \zeta=(\zeta_1,\dots,\zeta_n)$.
    
    Let $J\su [n]$ be the set of indices $j$ for which $\Pr[\zeta_j=z_j]> 1-\delta$ holds for some $z_j\in \RR$ (such a $z_j$ is unique if it exists). For each $j\in J$, let us represent the random variable $\zeta_j$ as $\zeta_j=\alpha_j+\xi_j \beta_j$, for a Rademacher random variable $\xi_j\in\{-1,1\}$ and a random vector $(\alpha_j,\beta_j)\in \RR^2$, as in \cref{lem-representation}. We do this in such a way that the random variables $\xi_j$ and the random vectors  $(\alpha_j,\beta_j)$ are all mutually independent for all $j\in J$. Note that we always have $\alpha_j+\beta_j=z_j$ for all $j\in J$ (by (b) in \cref{lem-representation}, recalling that $\Pr[\zeta=z_j]> 1-\delta>1/2$).

Now, our plan is to condition on arbitrary outcomes of $(\alpha_j,\beta_j)$ for $j\in J$, and prove the desired bound using the randomness of $\xi_j$ for $j\in J$, and the randomness of $\zeta_i$ for $i\notin J$, applying \cref{prop:general-distributions-assumption}.

For each $j\in J$ and each outcome of $(\alpha_j,\beta_j)$, the conditional distribution of $\zeta_j$ given our outcome of $(\alpha_j,\beta_j)$ is described by $\Pr[\zeta_j=\alpha_j+\beta_j\,|\,(\alpha_j,\beta_j)]=\Pr[\zeta_j=\alpha_j-\beta_j\,|\,(\alpha_j,\beta_j)]=1/2$ if $\beta_j\ne 0$, and
$\Pr[\zeta_j=z_j\,|\,(\alpha_j,\beta_j)]=1$  if $\beta_j=0$ (then $\alpha_j-\beta_j=\alpha_j+\beta_j=z_j$, so $\zeta_j$ is constant).
Note that conditioning on outcomes of $(\alpha_j,\beta_j)$ for $j\in J$ does not change the distribution of $\zeta_i$ for $i\notin J$, and does not change the fact that the random variables $\zeta_1,\dots,\zeta_n$ are independent.

For any outcome of $\vec \beta[J]$ (i.e., for any outcomes of $\beta_j$ for $j\in J$), let $H_{\vec \beta[J]}\su J$ be the set of indices $j\in J$ such that $\beta_j=0$, and let $I_{\vec \beta[J]}=[n]\setminus H_{\vec \beta[J]}$ (i.e., $I_{\vec \beta[J]}$ is the subset of indices for which $\zeta_i$ ``still has some randomness'' after conditioning on the outcomes of $(\alpha_j,\beta_j)$ for $j\in J$). Note that we always have $[n]\setminus J\su I_{\vec \beta[J]}$. Furthermore note that for any outcome of $((\alpha_j,\beta_j))_{j\in J}$, we have
\[\sup_{z\in \RR}\Pr\Bigl[\zeta_i=z\,\Big|\,\bigl((\alpha_j,\beta_j)\bigr)_{j\in J}\Bigr]\le 1-\delta\]
for all $i\in I_{\vec \beta[J]}$ (here we are using that $\delta<1/2$) and
\[\Pr\Bigl[\zeta_i=z_i\,\Big|\,\bigl((\alpha_j,\beta_j)\bigr)_{j\in J}\Bigr]=1\]
for all $i\in H_{\vec \beta[J]}$.

Now, for any outcome of $\vec \beta[J]$ (which determines $H_{\vec \beta[J]}$ and $I_{\vec \beta[J]}$), we have $Q(\vec \zeta)=Q_{\vec\beta[J]}(\vec \zeta[I_{\vec \beta[J]}])$, where $Q_{\vec\beta[J]}$ is the polynomial in the entries of $\vec \zeta[I_{\vec \beta[J]}]$ obtained from $Q(\vec \zeta)$ by substituting $\zeta_i$ with $z_i$ for all $i\in H_{\vec \beta[J]}$ (recall that we always have $\zeta_i=z_i$ for all $i\in H_{\vec \beta[J]}$).

We claim that, if we condition on any outcomes of $(\alpha_j,\beta_j)$ for $j\in J$,
then with respect to the resulting conditional probability space, $Q_{\vec\beta[J]}(\vec \zeta_i[I_{\vec \beta[J]}])$ satisfies the fixing box assumption in \cref{thm:general-distributions} (and therefore satisfies the assumptions of \cref{prop:general-distributions-assumption}). Indeed, for any outcome of $((\alpha_j,\beta_j))_{j\in J}$, consider a fixing box $\prod_{i\in I_{\vec \beta[J]}} R_i$ for $Q_{\vec\beta[J]}(\vec \zeta[I_{\vec \beta[J]}])$, and
suppose for the purpose of contradiction that there are fewer than $m$ indices $i\in I_{\vec \beta[J]}$ with $\Pr[\zeta_i\in R_i\,|\,((\alpha_j,\beta_j))_{j\in J}]\le 1-\delta$. Then, we can extend our fixing box for $Q_{\vec\beta[J]}$ to a fixing box $R_1\times \dots\times R_n$ for $Q(\zeta_1,\dots,\zeta_n)$ by simply taking $R_j=\{z_j\}$ for all $j\in H_{\vec \beta[J]}\su J$. Note that for each $j\in H_{\vec \beta[J]}\su J$ we have $\Pr[\zeta_j\in R_j]=\Pr[\zeta_j=z_j]>1-\delta$. Also, for $j\in J\setminus H_{\vec \beta[J]}=J\cap I_{\vec \beta[J]}$ (i.e., for $j\in J$ with $\beta_j\ne 0$), we can only have $\Pr[\zeta_j\in R_j]\le 1-\delta$ if $z_j\not\in R_j$ (as $\Pr[\zeta_j=z_j]> 1-\delta$), and in this case we also have $\Pr[\zeta_j\in R_j\,|\,(\alpha_j,\beta_j)]\le 1/2\le 1-\delta$ (as $\Pr[\zeta_j=z_j\,|\,(\alpha_j,\beta_j)]=\Pr[\zeta_j=\alpha_j+\beta_j\,|\,(\alpha_j,\beta_j)]=1/2$). Finally, note that conditioning on $((\alpha_j,\beta_j))_{j\in J}$ does not change the distributions of $\zeta_i$ for $i\in [n]\setminus J\su I_{\vec \beta[J]}$.
So every index $i\in [n]$ with $\Pr[\zeta_i\in R_i]\le 1-\delta$ has the property that $i\in I_{\vec \beta[J]}$ and $\Pr[\zeta_i\in R_i\,|\,((\alpha_j,\beta_j))_{j\in J}]\le 1-\delta$. Hence, there are fewer than $m$ indices $i\in [n]$ with $\Pr[\zeta_i\in R_i]\le 1-\delta$, contradicting our assumption on $Q(\zeta_1,\dots,\zeta_n)$.
    
Given the above discussion, for any outcomes of $(\alpha_j,\beta_j)$ for $j\in J$, we can apply \cref{prop:general-distributions-assumption}, to obtain
    \[\sup_{z\in\RR} \Pr\Bigl[Q(\zeta_1,\dots,\zeta_n)=z\,\Big|\,\bigl((\alpha_j,\beta_j)\bigr)_{j\in J}\Bigr]\le \frac{C_\delta}{\sqrt{m}}.\]
    The desired desired unconditional probability bound follows.
\end{proof}

\section{Concluding remarks}\label{subsec:further-directions}
In this paper we have obtained essentially optimal bounds for the quadratic Littlewood--Offord problem. There are many interesting directions for further research.

\textbf{Immediate generalisations.}
\cref{thm:main-shiny} is only about point concentration, and it would be interesting
to prove a counterpart for small-ball concentration; i.e., under which
assumptions on $Q$ can we prove that
\[
\sup_{z\in\RR}\Pr[|Q(\xi_{1},\dots,\xi_{n})-z|\le1]\le O\left(\frac{1}{\sqrt{n}}\right)?
\]
It would also be nice to prove a generalisation of \cref{thm:main-shiny} to polynomials of
degree greater than 2. Generalising the conjecture of Nguyen and Vu,
we believe the statement of \cref{thm:main-shiny} should hold whenever $Q$ is a polynomial
of bounded degree (note that without a bounded-degree assumption we cannot hope for sensible anticoncentration bounds; consider for example the \emph{parity
function} $Q(x_{1},\dots,x_{n})=x_{1}x_{2}\dots x_{n}$).

It is sometimes the case that techniques for point concentration can be straightforwardly
adapted for small-ball concentration, and techniques for quadratic
polynomials can be straightforwardly adapted for higher-degree polynomials (in particular,
the previous bounds of Meka--Nguyen--Vu \cite{MNV16} and Kane \cite{Kan14}, mentioned in the introduction, handle small ball concentration for polynomials
of any bounded degree). The high-level strategy of the proof of \cref{thm:main-shiny} (as sketched in \cref{sec:outline}) makes sense in a very general context, but when trying to generalise \cref{thm:main-shiny} in the obvious ways, one runs into some subtle technical issues (see \cref{rem:high-degree-issues,rem:small-ball-issues}). We think it would be very interesting to investigate this further.

It would also be desirable to generalise our ``geometric'' theorem (\cref{thm:quadratic-geometric}) to geometric objects other than quadrics inside affine-linear subspaces. For example, we make the following conjecture (closely related to higher-degree generalisations of the Littlewood--Offord problem, and closely related to the directions in \cite{FKS23}).

\begin{conjecture}Let $0\le d<r$ and $q$ be integers. Let $\mathcal{Z}\su\RR^{r}$ be an algebraic variety of dimension $d$, with degree at most $q$. Consider vectors $\vec{a}_{1},\dots,\vec{a}_{n}\in\mathbb{R}^{r}$ such that for some positive integer $t$, one can form $t$ disjoint bases from the vectors $\vec{a}_{1},\dots, \vec{a}_n$. Let $(\xi_1,\dots,\xi_n)\in\{-1,1\}^{n}$ be a sequence of independent Rademacher random variables. Then 
\[
\Pr\left[\xi_{1}\vec{a}_{1}+\dots+\xi_{n}\vec{a}_{n}\in\mathcal{Z}\right]\le\frac{C_{d,r,q}}{t^{(r-d)/2}}
\]
for some $C_{d,r,q}$ only depending on $d,r,q$.
\end{conjecture}

\textbf{The Gotsman--Linial conjecture.}
The \emph{Gotsman--Linial conjecture} is a conjecture in Boolean
analysis which generalises nearly all polynomial Littlewood--Offord-type
theorems. To state this conjecture we need to introduce some notation.

The $i$-th\emph{ influence }of a Boolean function $F:\{-1,1\}^{n}\to\{-1,1\}$
is defined as
\[
\on{Inf}_{i}(F)=\Pr[F(\xi_{1},\dots,\xi_{i-1},\xi_{i},\xi_{i+1},\dots\xi_{n})\ne F(\xi_{1},\dots,\xi_{i-1},-\xi_{i},\xi_{i+1},\dots\xi_{n})]
\]
(i.e., the probability that changing the $i$th bit changes the output
of the function). The \emph{total influence} (also sometimes called
the \emph{average sensitivity}) $\on{Inf}(F)$ of $F$ is $\on{Inf}_{1}(F)+\dots+\on{Inf}_{n}(F)$. For a function $Q:\{-1,1\}^n\to \mathbb{R}$, the \emph{threshold function} $F_{Q}:\{-1,1\}^{n}\to\{-1,1\}$ of
$Q$ is the Boolean function that detects whether $Q(x_{1},\dots,x_{n})\ge0$
or not. If $Q$ is a degree-$d$ polynomial, we say $F_{Q}$ is a
degree-$d$ threshold function.

Gotsman and Linial~\cite{GL94} made a notorious conjecture on the highest
total influence that an $n$-variable degree-$d$ polynomial threshold
function can have (namely, they conjectured that the total influence
is maximised when $Q(x_{1},\dots,x_{n})$ is a certain product of $d$ terms
of the form $(x_{1}+\dots+x_{n}+a)$). Unfortunately, the precise
form of this conjecture has been falsified~\cite{Cha18,KMW} (it holds
for the linear case $d=1$ but already fails in the quadratic case
$d=2$).

Nonetheless, it still seems plausible that for every $n$-variable degree-$d$
polynomial the total influence of threshold function can be bounded by $O(d\sqrt{n})$,
or at least by $C_{d}\sqrt{n}$ for some constant $C_{d}$ depending on $d$
(conjectures in this direction are sometimes variously called the
\emph{weak} Gotsman--Linial conjecture). If such a bound were to
hold, it would be an easy exercise to deduce \cref{thm:main-shiny} (and the generalisations
discussed above). It would be interesting to investigate whether the techniques in this paper can be used to make progress on (at least the quadratic form of the) weak Gotsman--Linial conjecture.

\textbf{Inverse theory.} For both the linear and quadratic Littlewood--Offord problems, one
cannot hope for a general bound stronger than $O(1/\sqrt{n})$. However,
it is natural to investigate assumptions under which one can prove
stronger bounds.

In the linear case (studying random variables of the form $X=a_{1}\xi_{1}+\dots+a_{n}\xi_{n}$),
this has been an enormously successful direction of research. Early
highlights include Stanley's solution of the so-called \emph{Erd\H os--Moser
problem}~\cite{Sta80} (introducing tools from algebraic topology to prove an optimal
bound under the assumption that all of the coefficients $a_{1},\dots,a_n$
are distinct), and a paper of Hal\'asz~\cite{Hal77} which introduced
Fourier-analytic methods to prove stronger bounds when $a_{1},\dots,a_n$
are in a certain sense ``additively unstructured''. Perhaps most
famously, Tao and Vu~\cite{TV09} proved a so-called \emph{Inverse
Littlewood--Offord theorem}, which shows that $\sup_{z\in\RR}\Pr[X=z]$
is \emph{extremely} small (smaller than $n^{-C}$ for any constant
$C$) unless $a_{1},\dots,a_n$ have very special additive structure (roughly
speaking, most of $a_{1},\dots,a_n$ lie inside a \emph{generalised arithmetic
progression}). This had a number of important applications in random
matrix theory~\cite{TV07,TV09a,TV09}. The inverse theory of the linear
Littlewood--Offord theorem is now essentially complete, thanks to \emph{optimal inverse theorems} of Nguyen and Vu~\cite{NV11} and Rudelson
and Vershinyn~\cite{RV08} (see also \cite{TV10}) that give a precise quantification of the
extent to which anticoncentration is controlled by the additive structure
of $a_{1},\dots,a_n$.

In the quadratic case (studying random variables of the form $X=Q(\xi_{1},\dots,\xi_{n})$
for a quadratic polynomial $Q$), much less is known.
An analogue of the original Tao--Vu inverse theorem was proved by
Nguyen~\cite{Ngu12}, but it is not clear how an optimal inverse theorem
should even be formulated. Such a theorem would have to somehow incorporate
the additive information that is relevant for the linear Littlewood--Offord
problem, in addition to algebraic considerations (e.g., whether $Q$
factorises into linear factors or not). See \cite{Cos13,Kan17,KS20} for some progress and conjectures
related to algebraic aspects of the inverse theory of the quadratic Littlewood--Offord problem.

\bibliographystyle{amsplain_initials_nobysame_nomr}
\bibliography{main.bib}

\end{document}